\documentclass[12pt]{article}

\usepackage{amsmath,amssymb,amsthm,amscd,bm}
\usepackage{graphicx}

\allowdisplaybreaks
\numberwithin{equation}{section}

\makeatletter
\renewcommand\section{\@startsection{section}{1}{0pt}
{-3.5ex plus -1ex minus -.2ex}{1.0ex plus .2ex}{\large\bf}}
\renewcommand\subsection{\@startsection{subsection}{1}{0pt}
{2.5ex plus 1ex minus .2ex}{-1em}{\bf}}
\makeatother

\newcommand{\Fg}{\mathfrak{g}}
\newcommand{\Fh}{\mathfrak{h}}

\newcommand{\BZ}{\mathbb{Z}}

\newcommand{\BR}{\mathbb{R}}
\newcommand{\BC}{\mathbb{C}}
\newcommand{\BB}{\mathbb{B}}
\newcommand{\BX}{\mathbb{X}}
\newcommand{\BY}{\mathbb{Y}}

\newcommand{\CB}{\mathcal{B}}

\newcommand{\CO}{\mathcal{O}}

\newcommand{\sT}{\mathsf{T}}
\newcommand{\sD}{\mathsf{D}}
\newcommand{\sS}{\mathsf{S}}
\newcommand{\sP}{\mathsf{P}}

\newcommand{\RH}{\mathrm{H}}
\newcommand{\RL}{\mathrm{L}}

\newcommand{\ve}{\varepsilon}
\newcommand{\vp}{\varphi}
\newcommand{\vpi}{\varpi}

\newcommand{\ba}{\mathbf{a}}
\newcommand{\bx}{\mathbf{x}}
\newcommand{\bp}{\mathbf{p}}

\newcommand{\be}{\mathbf{e}}
\newcommand{\bv}{\mathbf{v}}
\newcommand{\bw}{\mathbf{w}}
\newcommand{\bF}{\mathbf{F}}
\newcommand{\bG}{\mathbf{G}}
\newcommand{\bH}{\mathbf{H}}

\newcommand{\bzero}{\mathbf{0}}

\newcommand{\bchi}{\bm{\chi}}
\newcommand{\bom}{\bm{\omega}}
\newcommand{\bsg}{\bm{\sigma}}

\newcommand{\q}{\mathsf{v}}

\newcommand{\cl}{\mathop{\rm cl}\nolimits}
\newcommand{\fin}{\mathop{\rm fin}\nolimits}
\newcommand{\nul}{\mathop{\rm nul}\nolimits}

\newcommand{\gch}{\mathop{\rm gch}\nolimits}
\newcommand{\wt}{\mathop{\rm wt}\nolimits}

\newcommand{\Hom}{\mathop{\rm Hom}\nolimits}
\newcommand{\Ker}{\mathop{\rm Ker}\nolimits}
\newcommand{\Deg}{\mathop{\rm deg}\nolimits}
\newcommand{\Par}{\mathop{\rm Par}\nolimits}
\newcommand{\Conn}{\mathop{\rm Conn}\nolimits}
\newcommand{\Turn}{\mathop{\rm Turn}\nolimits}

\newcommand{\lng}{w_{\circ}}

\newcommand{\af}{\mathrm{af}}

\newcommand{\rr}{\Delta_{\af}}
\newcommand{\prr}{\Delta_{\af}^{+}}

\newcommand{\QLS}{\mathrm{QLS}}
\newcommand{\SLS}{\mathbb{B}^{\frac{\infty}{2}}}
\newcommand{\SM}{\mathbb{S}^{\frac{\infty}{2}}}

\newcommand{\sls}[3]{%
 \mathbb{B}^{\frac{\infty}{2}}_{%
    \begin{subarray}{c}
    \io{\bullet}{#2}=#1 \\
    \kappa(\bullet) \sige \PS{\J_{#3}}(#2)
    \end{subarray}}(#3)%
}

\newcommand{\si}{\frac{\infty}{2}}
\newcommand{\sell}{\ell^{\frac{\infty}{2}}}
\newcommand{\sil}{\prec}
\newcommand{\sile}{\preceq}
\newcommand{\sig}{\succ}
\newcommand{\sige}{\succeq}

\newcommand{\edge}[1]{ \xrightarrow{\hspace{2pt}#1\hspace{2pt}} }
\newcommand{\mcr}[1]{\lfloor #1 \rfloor}

\newcommand{\pair}[2]{\langle #1,\,#2 \rangle}

\newcommand{\ol}[1]{\overline{#1}}
\newcommand{\ti}[1]{\widetilde{#1}}
\newcommand{\ha}[1]{\widehat{#1}}

\newcommand{\SB}{\mathrm{BG}^{\si}(W_{\af})}
\newcommand{\QB}{\mathrm{QBG}(W)}

\newcommand{\J}{J}

\newcommand{\Ja}{\J^{\ast}}
\newcommand{\Jl}{\J_{\lambda}}
\newcommand{\Jm}{\J_{\mu}}
\newcommand{\Jlm}{\J_{\lambda+\mu}}
\newcommand{\Jr}{\J_{\vpi_{r}}}

\newcommand{\Jc}{I \setminus \J}
\newcommand{\QJ}{Q_{\J}}
\newcommand{\QJv}{Q_{\J}^{\vee}}
\newcommand{\QJvp}{Q_{\J}^{\vee,+}}
\newcommand{\DeJ}{\Delta_{\J}}
\newcommand{\PJ}{\Pi^{\J}}
\newcommand{\WJ}{W_{\J}}
\newcommand{\WJu}{W^{\J}}
\newcommand{\WJa}{(W^{\J})_{\af}}
\newcommand{\SBJ}{\mathrm{BG}^{\si}(\WJa)}
\newcommand{\SBa}{\mathrm{BG}^{\si}_{a\lambda}(\WJa)}
\newcommand{\SBb}[1]{\mathrm{BG}^{\si}_{#1\lambda}(\WJa)}
\newcommand{\QBJ}{\mathrm{QBG}(\WJu)}
\newcommand{\QBa}{\mathrm{QBG}_{a\mu}(\WJu)}
\newcommand{\QBb}[1]{\mathrm{QBG}_{#1\mu}(\WJu)}

\newcommand{\PS}[1]{\Pi^{#1}}
\newcommand{\WSu}[1]{W^{#1}}
\newcommand{\WS}[1]{W_{#1}}
\newcommand{\QSv}[1]{Q_{#1}^{\vee}}

\newcommand{\DeS}[1]{\Delta_{#1}}
\newcommand{\WSa}[1]{(W^{#1})_{\af}}
\newcommand{\SBS}[1]{\mathrm{BG}^{\si}((W^{#1})_{\af})}
\newcommand{\SBSa}[1]{\mathrm{BG}^{\si}_{a#1}((W^{\J_{#1}})_{\af})}

\newcommand{\kap}[2]{\kappa(#1,#2)}
\newcommand{\io}[2]{\iota(#1,#2)}
\newcommand{\ip}[1]{\iota(#1)}

\newcommand{\Lige}[2]{\mathrm{Lift}_{\sige #1}(#2)}
\newcommand{\Lile}[2]{\mathrm{Lift}_{\sile #1}(#2)}
\newcommand{\Lift}{\mathrm{Lift}}

\newcommand{\tb}[1]{\le_{#1}}
\newcommand{\dtb}[1]{\le_{#1}^{\ast}}
\newcommand{\tbmin}[3]{\min(#1W_{#2},\le_{#3}\nobreak)}
\newcommand{\tbmax}[3]{\max(#1W_{#2},\le_{#3}^{\ast}\nobreak)}

\newcommand{\bQG}{\mathbf{Q}_{G}}
\newcommand{\bQGr}{\mathbf{Q}_{G}^{\mathrm{rat}}}
\newcommand{\Fun}{\mathrm{Fun}}
\newcommand{\ngt}{\mathrm{neg}}
\newcommand{\ess}{\mathrm{ess}}

\newcommand{\bra}[1]{[\hspace{-1pt}[#1]\hspace{-1pt}]}

\newcommand{\pra}[1]{(\hspace{-1pt}(#1)\hspace{-1pt})}

\theoremstyle{plain}
\newtheorem{lem}{Lemma}[section]
\newtheorem{prop}[lem]{Proposition}
\newtheorem{thm}[lem]{Theorem}
\newtheorem{cor}[lem]{Corollary}
\newtheorem{claim}{Claim}[lem]

\newtheorem{ithm}{Theorem}
\newtheorem{icor}[ithm]{Corollary}

\theoremstyle{definition}
\newtheorem{dfn}[lem]{Definition}

\theoremstyle{remark}

\newtheorem{rem}[lem]{Remark}

\newcommand{\bqed}{\quad \hbox{\rule[-0.5pt]{3pt}{8pt}}}

\newenvironment{enu}{%
 \begin{enumerate}%
}{\end{enumerate}}

\begin{document}

\setlength{\baselineskip}{14.75pt}

\title{\Large\bf 
Chevalley formula for anti-dominant weights \\[2mm]
in the equivariant $K$-theory \\[2mm]
of semi-infinite flag manifolds%
}
\author{%
Satoshi Naito \\ 
 \small Department of Mathematics, Tokyo Institute of Technology, \\
 \small 2-12-1 Oh-okayama, Meguro-ku, Tokyo 152-8551, Japan \\
 \small (e-mail: {\tt naito@math.titech.ac.jp}) \\[5mm]
 Daniel Orr \\ 
 \small Department of Mathematics (MC 0123), 460 McBryde Hall, 
        Virginia Tech, \\
 \small 225 Stanger St., Blacksburg, VA 24061, 
 U.\,S.\,A. \ 
 (e-mail: {\tt dorr@vt.edu}) \\[5mm]
%
%
Daisuke Sagaki \\ 
 \small Institute of Mathematics, University of Tsukuba, \\
 \small 1-1-1 Tennodai, Tsukuba, Ibaraki 305-8571, Japan \\
 \small (e-mail: {\tt sagaki@math.tsukuba.ac.jp})
}
\date{}

\maketitle

%
\begin{abstract} \setlength{\baselineskip}{12pt}
We prove a Chevalley formula for anti-dominant weights 
in the torus-equivariant $K$-group of 
semi-infinite flag manifolds, which is described explicitly 
in terms of semi-infinite Lakshmibai-Seshadri paths 
(or, equivalently, quantum Lakshmibai-Seshadri paths).
Based on the isomorphism, recently established by Kato, 
between the (small) torus-equivariant quantum $K$-group of 
finite-dimensional flag manifolds and the torus-equivariant $K$-group of 
semi-infinite flag manifolds, these formulas give 
an explicit description of the structure constants for 
the quantum multiplication in the quantum $K$-group of 
finite-dimensional flag manifolds. 
Our proof of these formulas is based on
standard monomial theory for semi-infinite Lakshmibai-Seshadri paths, 
which is established in our previous work, and also uses 
a string property of Demazure-like subsets of the crystal basis of 
a level-zero extremal weight module over a quantum affine algebra.
\end{abstract}
%
%
\section{Introduction.} 
\label{sec:intro}

Let $\bQGr$ denote the (whole) semi-infinite flag manifold, 
which is the reduced ind-scheme whose set of 
$\BC$-valued points is $G(\BC\pra{z})/(H \cdot N(\BC\pra{z})$ 
(see \cite{Kat2} for details), where $G$ is 
a simply-connected simple algebraic group over $\BC$ 
with Borel subgroup $B = H N$, $H$ maximal torus and $N$ unipotent radical.
In this paper, we concentrate on the semi-infinite Schubert (sub)variety 
$\bQG := \bQG(e) \subset \bQGr$ associated to the identity element $e$ of 
the affine Weyl group $W_{\af} = W \ltimes Q^{\vee}$, 
with $W = \langle s_i \mid i \in I \rangle$ 
the (finite) Weyl group and $Q^{\vee} = \bigoplus_{i \in I} \BZ \alpha_i^{\vee}$ 
the coroot lattice of $G$;
we also call $\bQG$ the semi-infinite flag manifold.
In \cite{KNS}, we proposed a definition of an equivariant 
(with respect to the Iwahori subgroup) $K$-group of $\bQG$, 
and also proved a Chevalley formula for dominant weights 
in this $K$-group; this formula describes, 
in terms of semi-infinite Lakshmibai-Seshadri (LS for short) paths, 
the tensor product of the class of the line bundle $[\CO_{\bQG}(\lambda)]$ 
associated to a dominant weight $\lambda$ with the class of 
the structure sheaf $[\CO_{\bQG(x)}]$ of 
a semi-infinite Schubert (sub)variety $\bQG(x) \subset \bQG$ 
for $x = w t_{\xi} \in W_{\af}^{\geq 0} := W \times Q^{\vee,+} \subset W_{\af}$, 
where $Q^{\vee,+} := \sum_{i \in I} \BZ_{\geq 0} \alpha_{i}^{\vee} \subset Q^{\vee}$;
note that the class of the structure sheaf $[\CO_{\bQG(x)}]$ 
in this paper is denoted by $[\overline{\CO}_{x}]$ in \cite{Orr}.
The purpose of this paper is to give a Chevalley formula 
for anti-dominant weights in the $H \times \BC^*$-equivariant 
$K$-group $K_{H \times \BC^*}(\bQG)$ of $\bQG$, 
and also in an $H$-equivariant $K$-group $K_{H}^{\prime}(\bQG)$.

A breakthrough in the study of 
the equivariant $K$-group of $\bQG$ was achieved in \cite{Kat1} and \cite{Kat3}, 
in which Kato established a $\BC[P]$-module isomorphism 
from the (small) $H$-equivariant quantum $K$-group 
$QK_{H}(G/B) := K_{H}(G/B) \otimes \BC[Q^{\vee,+}]$ 
of the finite-dimensional flag manifold $G/B$ onto 
an $H$-equivariant $K$-group $K_{H}^{\prime}(\bQG)$ of $\bQG$, 
where $\BZ[P] \subset \BC[P]$ is the group ring of 
the weight lattice $P = \bigoplus_{i \in I} \BZ \varpi_i$ of $G$, 
identified with the representation ring of $H$, 
and $\BC[Q^{\vee,+}]$ is the polynomial (not formal power series) ring 
in the (Novikov) variables $Q_{i}$, $i \in I$;
also, it is noteworthy that Kato verified conjectures in \cite{LLMS} 
about the relationship between $QK_{H}(G/B)$ and 
the (suitably localized) $H$-equivariant $K$-group 
of the affine Grassmannian associated to $G$.
This $\BC[P]$-module isomorphism sends 
the (opposite) Schubert class $[\CO^{w}]$ in $QK_{H}(G/B)$ 
to the corresponding semi-infinite Schubert class $[\CO_{\bQG(w)}]$ 
in $K_{H}^{\prime}(\bQG)$ for each $w \in W$.
Moreover, it respects the quantum multiplication $\star$ 
in $QK_{H}(G/B)$ and the tensor product in $K_{H}^{\prime}(\bQG)$; 
to be more precise, it respects the quantum multiplication $\star$ 
with the class of the line bundle $[\CO_{G/B}(- \varpi_i)]$ and 
the tensor product $\otimes$ with the class of the line bundle 
$[\CO_{\bQG}(\lng \varpi_{i})]$ for each $i \in I$, 
where $\lng$ is the longest element of the finite Weyl group $W$.
Namely, the following diagram commutes:
\begin{equation*}
\begin{CD}
QK_{H}(G/B) @>{\simeq}>> K_{H}^{\prime}(\bQG) \\
@V{\cdot \, \star [\CO_{G/B}(- \varpi_i)]}VV  @VV{\cdot \, \otimes [\CO_{\bQG}(\lng \varpi_i)]}V \\
QK_{H}(G/B) @>>{\simeq}> K_{H}^{\prime}(\bQG);
\end{CD}
\end{equation*}
here we warn the reader that the convention in \cite{Kat1} 
differs from that in \cite{KNS} and in this paper 
by the twist coming from the involution $- \lng$, 
and the line bundles $\CO_{\bQG}(\mu)$ for $\mu \in P$ are 
normalized in such a way that the quantum multiplication 
with $[\CO_{G/B}(- \varpi_i)]$ corresponds to 
the tensor product with $[\CO_{\bQG}(- \varpi_i)]$ 
under the $\BC[P]$-module isomorphism above.
Note that the quantum multiplicative structure of $QK_{H}(G/B)$ 
is determined by the quantum multiplication 
with the class of the line bundles $[\CO_{\bQG}(-\varpi_{i})]$, $i \in I$, 
and the $\BC[P]$-module structure of $QK_{H}(G/B)$; 
for the $\BC[P]$-module structure, see \cite{Orr}.

The $H \times \BC^*$-equivariant $K$-group $K_{H \times \BC^*}(\bQG)$ 
has a topological $\BC[q, q^{-1}][P]$-basis of 
the semi-infinite Schubert classes $[\CO_{\bQG(w t_{\xi})}]$ 
for $w \in W$ and $\xi \in Q^{\vee, +} = 
\sum_{i \in I} \BZ_{\geq 0} \alpha_i^{\vee} \subset Q^{\vee}$; 
the $H$-equivariant $K$-group $K_{H}^{\prime}(\bQG)$ above is, 
by definition, obtained by the specialization $q = 1$ 
from the $\BC[q, q^{-1}][P]$-submodule 
$K_{H \times \BC^*}^{\prime}(\bQG) \subset K_{H \times \BC^*}(\bQG)$ 
consisting of all finite $\BC[q, q^{-1}][P]$-linear combinations 
of these semi-infinite Schubert classes.
Hence a Chevalley formula for anti-dominant weights 
in $K_{H}^{\prime}(\bQG)$ (and, in particular, 
a Chevalley formula for the quantum multiplication 
with $[\CO_{G/B}(- \varpi_i)]$, $i \in I$, in $QK_{T}(G/B)$) is 
obtained from our Chevalley formula in $K_{H \times \BC^*}(\bQG)$ 
by the specialization $q = 1$.

Let us state our Chevalley formula more precisely.
Denote by $P^{+} := \sum_{i \in I} \BZ_{\ge 0} \varpi_{i}$ 
the set of dominant integral weights.
For a dominant weight $\mu \in P^{+}$, we set 
$\Jm := \bigl\{ i \in I \mid \pair{\mu}{\alpha_{i}^{\vee}} = 0 \bigr\}$;
we denote by $W^{\Jm}$ the set of minimal-length coset representatives 
for the cosets in $W/W_{\Jm}$, where 
$W_{\Jm} := \langle s_{i} \mid i \in \Jm \rangle$ is the stabilizer of $\mu$ in $W$. 
Let $\QLS(\mu)$ denote the set of quantum LS paths of shape $\mu$.
Recall that an element $\eta \in \QLS(\mu)$ is a sequence of elements of $W^{\Jm}$ 
that satisfies a certain condition described in terms of 
the parabolic quantum Bruhat graph (see Definition~\ref{dfn:QLS}),
and we can endow the set $\QLS(\mu)$ with a crystal structure with 
weights in $P$. 
For $\eta \in \QLS(\mu)$ and $v \in W$, 
we define an element $\kap{\eta}{v} \in W$ 
(called the final direction of $\eta$ with respect to $v$) and 
an element $\zeta(\eta,v) \in Q^{\vee,+}$ 
in terms of the quantum version of the Deodhar lifts 
introduced in \cite{LNSSS} and 
the (quantum) weight of a directed path in the quantum Bruhat graph 
(see \eqref{eq:haw} and \eqref{eq:zeta}).
Also, we denote by $\Deg : \QLS(\mu) \rightarrow \BZ_{\le 0}$ 
the (tail) degree function on $\QLS(\mu)$ introduced in \cite{NS08} and 
\cite{LNSSS2}, which is described in terms of 
the parabolic quantum Bruhat graph (see \eqref{eq:deg}). 

One of the main results of this paper is the following Chevalley formula 
for anti-dominant weights in the $H \times \BC^{*}$-equivariant
$K$-group $K_{H \times \BC^{*}}(\bQG)$; in fact, 
this is a formula in $K_{H \times \BC^{*}}^{\prime}(\bQG)$ 
because of its significant finiteness. 
%
%
\begin{ithm}[Chevalley formula in 
$K_{H \times \BC^{*}}^{\prime}(\bQG) \subset 
K_{H \times \BC^{*}}(\bQG)$] \label{ithm1}
Let $\lambda \in P^{+}$ be a dominant integral weight, and 
let $x = w t_{\xi} \in W_{\af}^{\geq 0}$. Then, 
in the $H \times \BC^{*}$-equivariant $K$-group 
$K_{H \times \BC^{*}}^{\prime}(\bQG) \subset K_{H \times \BC^{*}}(\bQG)$, 
we have
%
%
\begin{equation} \label{eq:PC1}
\begin{split}
& [\CO_{\bQG}(- \lambda)] \otimes [\CO_{\bQG(x)}]  \\[3mm]
& \qquad =
\sum_{v \in W}
\sum_{
  \begin{subarray}{c}
   \eta \in \QLS(- \lng \lambda) \\
   \kappa(\eta,v) = w
  \end{subarray}}
  (-1)^{\ell(v) - \ell(w)} 
  q^{-\deg(\eta) + \pair{-\lng \lambda}{\xi}}
  \be^{- \wt (\eta)} \, 
  [\CO_{\bQG(v t_{\xi + \zeta(\eta, v)})}]. 
\end{split}
\end{equation}
\end{ithm}

We remark that the sum on the right-hand side of \eqref{eq:PC1} 
is clearly a finite sum, since $W$ is a finite Weyl group and 
$\QLS(-\lng \lambda)$ is a finite set; in fact, the set 
$\QLS(-\lng \lambda)$ provides a realization of 
the crystal basis of the quantum version of a local Weyl module 
(which is finite-dimensional). This finiteness assertion plays 
an important role in the proof of the existence of 
an isomorphism between $QK_{H}(G/B)$ and $K_{H}^{\prime}(\bQG)$ 
(see \cite{Kat3} and also \cite{ACT}).

Now we explain how to prove our Chevalley formula 
for anti-dominant weights in $K_{H \times \BC^*}(\bQG)$.
Denote by $K_{H \times \BC^*}^{\dagger}(\bQG)$ 
the $\BC[q, q^{-1}][P]$-module consisting of 
all (possibly infinite) linear combinations of 
the semi-infinite Schubert classes $[\CO_{\bQG(x)}]$ 
with coefficients $a_{x} \in \BC[q, q^{-1}][P]$ 
for $x \in W_{\af}^{\geq 0} = W_{\af} \times Q^{\vee,+}$ 
such that the sum $\sum_{x \in W_{\af}^{\geq 0}} \vert a_{x} \vert$ 
of the absolute values $\vert a_{x} \vert$ 
(where the absolute value of the coefficient 
of each monomial appearing in $a_{x}$ is taken) lie in $\BC\pra{q^{-1}}[P]$.
Also, let $\Fun_{P} (\BC\pra{q^{-1}}[P])$ 
denote the $\BC[q, q^{-1}][P]$-module of 
all functions on $P$ with values in $\BC\pra{q^{-1}}[P]$, and set
\begin{equation*}
\Fun_{P}^{\ess} (\BC\pra{q^{-1}}[P]) := 
\Fun_{P} (\BC\pra{q^{-1}}[P])/\Fun_{P}^{\ngt} (\BC\pra{q^{-1}}[P]),
\end{equation*}
where $\Fun_{P}^{\ngt} (\BC\pra{q^{-1}}[P])$ is 
the $\BC[q, q^{-1}][P]$-submodule of 
$\Fun_{P} (\BC\pra{q^{-1}}[P])$ consisting of 
those $f \in \Fun_{P} (\BC\pra{q^{-1}}[P])$ 
such that there exists some $\gamma \in P$ 
for which $f(\mu) = 0$ for all $\mu \in \gamma + P^+$, 
with $P^{+} := \sum_{i \in I} \BZ_{\geq 0} \varpi_i$.
Then, for each $x \in W_{\af}^{\geq 0}$, the assignment
\begin{equation*}
P \owns \mu \mapsto 
\gch H^{0}(\bQG, \CO_{\bQG(x)} \otimes \CO_{\bQG}(\mu)) \in \BC\pra{q^{-1}}[P]
\end{equation*}
defines an element of $\Fun_{P}^{\ess} (\BC\pra{q^{-1}}[P])$, 
which we denote by $f^{x}(\,\cdot\,)$; 
here we denote by $\gch H^{0}(\bQG, \CO_{\bQG(x)} \otimes \CO_{\bQG}(\mu))$ 
the graded character of the $H \times \BC^*$-module 
$H^{0}(\bQG, \CO_{\bQG(x)} \otimes \CO_{\bQG}(\mu))$, 
which is (as proved in \cite{KNS}) identical to 
the graded character of the Demazure submodule $V_{x}^{-}(- \lng \mu)$ 
of the level-zero extremal weight module $V(- \lng \mu)$ 
over the quantum affine algebra $U_{\q}(\Fg_{\af})$ 
if $\mu \in P^+$, and is zero if $\mu \notin P^{+}$,
where $\Fg_{\af}$ is the untwisted affine Lie algebra 
whose underlying finite-dimensional simple Lie algebra 
is the Lie algebra $\Fg$ of $G$.
Thus we obtain a $\BC[q, q^{-1}][P]$-linear map
\begin{equation*}
\Phi^{\dagger} : K_{H \times \BC^*}^{\dagger}(\bQG)
 \rightarrow \Fun_{P}^{\ess} (\BC\pra{q^{-1}}[P])
\end{equation*}
given by $\Phi^{\dagger}([\CO_{\bQG(x)}]) = 
f^{x}(\,\cdot\,)$ for each $x \in W_{\af}^{\geq 0}$.
Following \cite{Kat3}, we define an $H \times \BC^*$-equivariant 
$K$-group $K_{H \times \BC^*}(\bQG)$ to be 
the quotient $\BC[q, q^{-1}][P]$-module 
$K_{H \times \BC^*}^{\dagger}(\bQG)/\Ker(\Phi^{\dagger})$; 
also, we denote by $K_{H \times \BC^*}^{\prime}(\bQG)$ 
the $\BC[q, q^{-1}][P]$-submodule of 
$K_{H \times \BC^*}^{\dagger}(\bQG)$ consisting of 
all finite $\BC[q, q^{-1}][P]$-linear combinations of 
the semi-infinite Schubert classes $[\CO_{\bQG(x)}]$, $x \in W_{\af}^{\geq 0}$. 
Because the graded characters $\gch V_{x}^{-}(-\lng \mu)$, 
$x \in W_{\af}^{\geq 0}$, are linearly independent over $\BC[q, q^{-1}][P]$ 
when they are regarded as functions of sufficiently large $\mu \in P^{+}$,
the $\BC[q, q^{-1}][P]$-linear map $\Phi^{\dagger}$ restricted to 
$K_{H \times \BC^{*}}'(\bQG)$ is injective, 
and hence we have an embedding of $\BC[q, q^{-1}][P]$-modules:
\begin{equation*}
K_{H \times \BC^*}^{\prime}(\bQG) \hookrightarrow K_{H \times \BC^*}(\bQG),
\end{equation*}
which is compatible with the $\BC[q, q^{-1}][P]$-linear injective map
\begin{equation*}
\Phi : K_{H \times \BC^*}(\bQG) \rightarrow \Fun_{P}^{\ess} (\BC\pra{q^{-1}}[P])
\end{equation*}
induced by the map $\Phi^{\dagger}$ above.

From the explicit identities obtained in \cite{KNS} 
(in the case of dominant weights) and in this paper 
(in the case of anti-dominant weights) 
for the graded characters of Demazure submodules of 
level-zero extremal weight modules, it can be shown (see \cite{Kat3}) 
that there exist $\BC[q, q^{-1}][P]$-module endomorphisms 
$\Xi(\lambda)$, $\lambda \in P$, of $K_{H \times \BC^*}(\bQG)$ 
such that $\Xi(\lambda + \lambda^{\prime}) = 
\Xi(\lambda) \circ \Xi(\lambda^{\prime})$ 
for $\lambda,\,\lambda^{\prime} \in P$ and 
such that the following diagram commutes for all $\lambda \in P$:
\begin{equation*}
\begin{CD}
K_{H \times \BC^*}(\bQG) @>\Phi>> \Fun_{P}^{\ess} (\BC\pra{q^{-1}}[P]) \\
@V{\Xi(\lambda)}VV @VV{\Theta(\lambda)}V \\
K_{H \times \BC^*}(\bQG) @>>\Phi> \Fun_{P}^{\ess} (\BC\pra{q^{-1}}[P]),
\end{CD}
\end{equation*}
where $\Theta(\lambda) f(\,\cdot\,) = f(\,\cdot\,+ \lambda)$ 
for $f(\,\cdot\,) \in \Fun_{P}^{\ess} \BC\pra{q^{-1}}[P])$;
the $\BC[q, q^{-1}][P]$-module endomorphism $\Xi(\lambda)$ 
can be thought of as the tensor product 
$\cdot \, \otimes [\CO_{\bQG}(\lambda)]$ 
with the class of the line bundle $[\CO_{\bQG}(\lambda)]$ 
in $K_{H \times \BC^*}(\bQG)$ (see [KNS]).
In view of the commutativity of the diagram above 
and the injectivity of the $\BC[q, q^{-1}][P]$-linear map $\Phi$, 
the proof of our Chevalley formula for anti-dominant weights 
in $K_{H \times \BC^*}(\bQG)$ is reduced to the proof of 
the corresponding identity in Corollary~\ref{icor4} below 
for the graded characters of Demazure submodules of 
level-zero extremal weight modules; observe that 
the submodule $K_{H \times \BC^{*}}^{\prime}(\bQG)
\subset K_{H \times \BC^{*}}(\bQG)$ is stabilized 
by the endomorphisms $\Xi(-\lambda)$, $\lambda \in P^{+}$,
because of the finiteness of the sum on the right-hand side of 
equation \eqref{eq:gch2} in Corollary~\ref{icor4}.
Indeed, since the left-hand side of \eqref{eq:PC1} 
can be written as $\Xi(- \lambda)([\CO_{\bQG(x)}])$, 
its image under $\Phi$ is identical to 
$\Theta(- \lambda)(\Phi([\CO_{\bQG(x)}]))$ by the commutativity of 
the diagram above; by the definitions, this is identical to the graded character
$\gch V_{x}^{-}(-\lng(\mu - \lambda))$ (regarded as a function of $\mu \in P$)
if $\mu - \lambda \in P^{+}$, and is zero if $\mu - \lambda \notin P^{+}$.
Also, the image under $\Phi$ of the right-hand side of \eqref{eq:PC1} is identical to:
\begin{equation*}
\sum_{v \in W} \sum_{%
 \begin{subarray}{c} \eta \in \QLS(- \lng \lambda) \\ 
 \kappa(\eta) = w \end{subarray}} (-1)^{\ell(v) - \ell(w)} 
 q^{-\deg(\eta) + \pair{-\lng \lambda}{\xi}} 
 \be^{-\wt(\eta)} \gch V_{vt_{\xi + \zeta(\eta,v)}}^{-}(-\lng \mu)
\end{equation*}
(regarded as a function of $\mu \in P$) if $\mu \in P^{+}$;
note that this turns out to be zero if $\mu - \lambda \notin P^{+}$,
as shown in Appendix~\ref{sec:notdom}.
Because these two functions of $\mu$ coincide 
if $\lambda, \, \mu \in P^{+}$ are such that $\mu - \lambda \in P^{+}$ 
by Corollary~\ref{icor4} below, we deduce the equalty \eqref{eq:PC1} 
in $K_{H \times \BC^{*}}(\bQG)$ from the injectivity of 
the $\BC[q, q^{-1}][P]$-linear map $\Phi$.

Our second main result of this paper is a Chevalley formula 
for anti-dominant weights in the torus-equivariant $K$-group $K_{H}^{\prime}(\bQG)$, 
which is obtained from $K_{H \times \BC^{*}}^{\prime}(\bQG) \subset 
K_{H \times \BC^{*}}(\bQG)$ by the specialization $q = 1$ (of coefficients). 
Since the equality in Theorem~\ref{ithm1} above is, 
in fact, the one in $K_{H \times \BC^{*}}^{\prime}(\bQG)$, 
it immediately implies the following:
%
%
\begin{ithm}[Chevalley formula in $K_{H}^{\prime}(\bQG)$] \label{ithm2}
Let $\lambda \in P^{+}$ be a dominant weight, 
and let $x = w t_{\xi} \in W_{\af}^{\geq 0}$. 
Then, in the torus-equivariant $K$-group $K_{H}^{\prime}(\bQG)$, we have
%
%
\begin{equation} \label{eq:PC2}
\begin{split}
& [\CO_{\bQG}(- \lambda)] \otimes [\CO_{\bQG(x)}]  \\[3mm]
& \qquad =
\sum_{v \in W}
\sum_{
  \begin{subarray}{c}
   \eta \in \QLS(- \lng \lambda) \\
   \kappa(\eta,v) = w
  \end{subarray}}
  (-1)^{\ell(v) - \ell(w)} 
  \be^{- \wt (\eta)} \, 
  [\CO_{\bQG(v t_{\xi + \zeta(\eta, v)})}].
\end{split}
\end{equation}
\end{ithm}

By applying \eqref{eq:PC2} to 
the case that $\lambda=-\lng \vpi_{i}$, $i \in I$, and $w = e$ 
(see \S\ref{subsec:s1}), we deduce that 
$[\CO_{\bQG(s_{i})}] =  [\CO_{\bQG(e)}] - 
\be^{\varpi_{i}} [\CO_{\bQG}(\lng \vpi_{i})]$ 
(which agrees with \cite[Lemma~1.14]{Kat1}; 
recall the difference of the conventions mentioned above).
By combining this equality and \eqref{eq:PC2}, 
we obtain the following formula, 
which describes the tensor product in $K_{H}^{\prime}(\bQG)$ 
by the class $[\CO_{\bQG(s_{i})}]$: 
\begin{equation} \label{eq:Monk}
\begin{split}
& [\CO_{\bQG(s_{i})}] \otimes [\CO_{\bQG(w)}] \\
& \qquad = 
[\CO_{\bQG(w)}] + 
\sum_{v \in W}
\sum_{
  \begin{subarray}{c}
   \eta \in \QLS(\vpi_{i}) \\
   \kappa(\eta,v) = w
  \end{subarray}}
  (-1)^{\ell(v) - \ell(w) + 1} 
  \be^{\varpi_{i} - \wt (\eta)} \, 
  [\CO_{\bQG(v t_{\zeta(\eta, v)})}]
\end{split}
\end{equation}
for $i \in I$ and $w \in W$. 
In particular, if $\vpi_{i}$ is a minuscule weight, 
then (see \S\ref{subsec:minu})
%
%
\begin{equation} \label{eq:minu}
\begin{split}
& [\CO_{\bQG(s_{i})}] \otimes [\CO_{\bQG(w)}]  \\
& \quad 
= [\CO_{\bQG(w)}] - \be^{ \vpi_{i} - w \vpi_{i} } \sum_{
  \begin{subarray}{c}
  v \in W \\[1mm] 
  \tbmax{w}{I \setminus \{i\}}{v} = w
  \end{subarray}}
  (-1)^{\ell(v) - \ell(w)} 
  [\CO_{\bQG(v t_{\wt (w \Rightarrow v)})}], 
\end{split}
\end{equation}
where $\tbmax{w}{I \setminus \{i\}}{v}$ denotes the quantum version of 
the Deodhar lift (see Proposition~\ref{prop:tbmax}) and 
$\wt (w \Rightarrow v)$ is the (quantum) weight of a directed path from $w$ to $v$ 
in the quantum Bruhat graph (see \eqref{eq:wtdp}). 
Based on the $\BC[P]$-module isomorphism 
(established in \cite{Kat1} and \cite{Kat3}) 
from $QK_{H}(G/B)$ onto $K_{H}^{\prime}(\bQG)$, 
which respects the quantum multiplication in $QK_{H}(G/B)$ 
and the tensor product in $K_{H}^{\prime}(\bQG)$ and 
which sends each (opposite) Schubert class in $QK_{H}(G/B)$ to 
the corresponding semi-infinite Schubert class in $K_{H}^{\prime}(\bQG)$, 
we immediately obtain from formula \eqref{eq:Monk}
the corresponding formula for the quantum multiplication in $QK_{H}(G/B)$.
Moreover, a conjectural Monk formula (\cite[Conjecture~17.1]{LeP07}) 
in $QK_{T}(G/B)$, which is described in terms of the quantum alcove model,
follows from this formula (see \cite{LNS} for details).
Also, for the torus-equivariant quantum $K$-group of partial flag manifolds 
of minuscule type, we can deduce (see \cite{KoNS}) 
from formula \eqref{eq:minu} an explicit formula describing 
the quantum multiplication 
(cf. \cite{BCMP} in the case of cominuscule type).

Our formula \eqref{eq:PC2}
can also be thought of as a semi-infinite analog of 
the corresponding formula in \cite{LeSh14} 
in the torus-equivariant $K$-theory for Kac-Moody thick flag manifolds 
(see also \cite{GR04} for the finite-dimensional case), 
though our proof is quite different from the one 
by them and is much more difficult; 
this is mainly because an ordinary induction argument 
using a string property of Demazure-like subsets 
does not suffice in our case in contrast to 
the case of Kac-Moody thick flag manifolds.

Let us explain the representation-theoretic 
(or, crystal-theoretic) aspect of our main results; 
as mentioned above, Theorems~\ref{ithm1} and \ref{ithm2} follow
from Theorem~\ref{ithm3} and Corollary~\ref{icor4} below, 
which are proved by using crystal bases of 
level-zero extremal weight modules.
Let $\mu \in P^{+}$ be a dominant integral weight. 
We denote by $\WSa{\Jm}$ the set of Peterson's coset representatives 
for the cosets in $W_{\af}/(W_{\Jm})_{\af}$ (see \S\ref{subsec:SiBG}), 
with $\PS{\Jm}:W_{\af} \twoheadrightarrow \WSa{\Jm}$ the canonical projection. 
Let $\SLS(\mu)$ denote the set of semi-infinite LS paths of shape $\mu$.
Recall that an element $\pi \in \SLS(\mu)$ is a certain decreasing sequence of 
elements of $\WSa{\Jm} \subset W_{\af}$ 
in the semi-infinite Bruhat order $\succeq$ (see Definition~\ref{dfn:SLS}),
and we can endow the set $\SLS(\mu)$ with a crystal structure with 
weights in $P_{\af}^{0}=P \oplus \BZ \delta$, where 
$\delta$ is the (primitive) null root of the untwisted 
affine Lie algebra $\Fg_{\af}$. 
If we denote by $\kappa(\pi) \in \WSa{\Jm}$ 
the final direction of $\pi \in \SLS(\mu)$,
then we know from \cite{NS16} that for each $z \in W_{\af}$,
the graded character $\gch V_{z}^{-}(\mu)$ of the Demazure submodule
$V_{z}^{-}(\mu)$ of the level-zero extremal weight module $V(\mu)$ over
$U_{\q}(\Fg_{\af})$ is identical to the following sum:
\begin{equation*}
\sum_{
 \begin{subarray}{c}
   \pi \in \SLS(\mu) \\
   \kappa(\pi) \sige z
  \end{subarray}}
q^{\nul(\wt(\pi))} \be^{\fin(\wt(\pi))},
\end{equation*}
where for $\nu \in P_{\af}^{0} = P \oplus  \BZ \delta$, 
we write $\nu = \fin(\nu) + \nul(\nu) \delta$.
Based on this fact, we make essential use of 
standard monomial theory for semi-infinite LS paths (established in \cite{KNS}) 
to prove the following theorem (see Theorem~\ref{thm:main});
we also need a string property of Demazure-like subsets of 
the crystal basis of an extremal weight module (see Proposition~\ref{prop:string}).
%
%
\begin{ithm} \label{ithm3}
Let $\lambda,\,\mu \in P^{+}$ be dominant integral weights 
such that $\mu - \lambda \in P^{+}$, and set 
$\lambda_{i} := \pair{\lambda}{\alpha_{i}^{\vee}} \in \BZ_{\ge 0}$, 
$\mu_{i} := \pair{\mu}{\alpha_{i}^{\vee}} \in \BZ_{\ge 0}$ for $i \in I${\rm;} 
note that $\mu_{i} - \lambda_{i} \in \BZ_{\geq 0}$ for all $i \in I$.
Then, for $x \in W_{\af}$, we have
\begin{equation} \label{eq:gch0}
\begin{split}
& \prod_{i \in I} \prod_{k = \mu_{i} - \lambda_{i} + 1}^{\mu_{i}} 
  \frac{1}{1 - q^{-k}}  \gch V_{x}^{-}(\mu - \lambda)  \\
& \hspace*{5mm}
= \sum_{
   \begin{subarray}{c}
   y \in W_{\af} \\
   y \sige x
   \end{subarray}} 
\left(
  \sum_{
   \begin{subarray}{c}
   \eta \in \SLS(\lambda) \\
   \iota(\eta) \sile \PS{J_{\lambda}}(y), \, \kap{\eta}{y} = x
   \end{subarray}}
   (-1)^{\sell(y) - \sell(x)} q^{-\nul(\wt(\eta))}\be^{- \fin(\wt(\eta))}
\right) 
   \gch V_{y}^{-}(\mu),
\end{split}
\end{equation}
where $\sell : W_{\af} \rightarrow \BZ$ denotes 
the semi-infinite length function {\rm(}see Definition~{\rm \ref{dfn:sell})}, and 
$\kap{\eta}{y} \in W_{\af}$ denotes the final direction of $\eta$ 
with respect to $y$ {\rm(}see \eqref{eq:hax}{\rm)}. 
\end{ithm}

The canonical projection $\cl : W_{\af} = W \ltimes Q^{\vee} \twoheadrightarrow W$ 
induces a surjective map from $\SLS(\mu)$ onto $\QLS(\mu)$, 
which is also denoted by $\cl$. By using this surjective map,  
we can reformulate this theorem in terms of 
the (parabolic) quantum Bruhat graph as follows (see Corollary~\ref{cor:main}). 
%
%
\begin{icor} \label{icor4}
Let $\lambda,\,\mu \in P^{+}$ be dominant integral weights 
such that $\mu - \lambda \in P^{+}$. 
Then, for $x = w t_{\xi} \in W_{\af}$, we have
\begin{equation} \label{eq:gch2}
\begin{split}
& \gch V_{x}^{-}(\mu - \lambda) \\ 
& \quad = \sum_{v \in W}
\sum_{\begin{subarray}{c}
\eta \in \QLS(\lambda) \\
\kappa(\eta, v) = w
\end{subarray}}
(-1)^{\ell(v) - \ell(w)} 
q^{- \deg_{\lambda}(\eta) +\pair{\lambda}{\xi}} 
\be^{- \wt (\eta)} \, 
\gch V_{v t_{\xi + \zeta(\eta, v)}}^{-}(\mu).
\end{split}
\end{equation}
\end{icor}

Here we remark that in Theorem~\ref{ithm3} and Corollary~\ref{icor4} above, 
the right-hand sides of the equations \eqref{eq:gch0} and \eqref{eq:gch2} 
turn out to be zero unless $\mu - \lambda \in P^{+}$ (see Appendix~\ref{sec:notdom}). 

This paper is organized as follows. 
In Section~\ref{sec:BG}, we fix our notation for untwisted affine Lie algebras,
and then review some basic facts about the (parabolic) semi-infinite Bruhat graph, 
the (parabolic) quantum Bruhat graph, and analogs of the Deodhar lifts for these graphs. 
In Section~\ref{sec:main}, after recalling the notions of semi-infinite LS paths 
and quantum LS paths, we state the formulas (Theorem~\ref{ithm3} and Corollary~\ref{icor4}) above 
for the graded characters of Demazure submodules, 
from which our main results (Theorem~\ref{ithm1} and Theorem~\ref{ithm2}) immediately 
follow on the basis of the results established in \cite{KNS}. 
In Section~\ref{sec:SMT}, we show some results in standard monomial theory 
for semi-infinite LS paths, which will be needed in the proof of Theorem~\ref{ithm3}.
In Section~\ref{sec:prf1}, we prove a special case of Theorem~\ref{ithm3}, 
in which the dominant integral weight $\lambda \in P^{+}$ is 
a fundamental weight and $w = e$, the identity element, 
by using standard monomial theory for semi-infinite LS paths.
In Section~\ref{sec:dem}, we show a string property of 
certain Demazure-like subsets of the crystal of 
semi-infinite LS paths of a given shape; 
using this, in Section~\ref{sec:prf2}, 
we prove Theorem~\ref{ithm3} in the case that 
$\lambda$ is a fundamental weight and $w$ is an arbitrary element of $W$.
In Section~\ref{sec:prf3}, we complete the proof of Theorem~\ref{ithm3} 
by induction on the positive integer $\sum_{i \in I} \lambda_{i}$; 
the base case that $\lambda$ is a fundamental weight is already established in Section~\ref{sec:prf2}. 
Also, we reformulate Theorem~\ref{ithm3} in terms of 
the quantum Bruhat graph as Corollary~\ref{icor4} above.
In Appendix~\ref{sec:example}, we give examples of formula \eqref{eq:minu} 
in the case that $\Fg$ is of type $A_{2}$ and $\lambda = \vpi_{1}$.
In Appendix~\ref{sec:notdom}, 
we show that the right-hand side of equation \eqref{eq:gch2} is
indeed zero if $\mu - \lambda \notin P^{+}$. 
In Appendix~\ref{sec:PC-QLS}, 
we give a rephrasement of \cite[Theorem~5.10]{KNS}
in terms of quantum LS paths.
%
%
\subsection*{Acknowledgments.}
We would like to thank Syu Kato for related collaborations.
S.N. was partially supported by 
JSPS Grant-in-Aid for Scientific Research (B) 16H03920. 
D.O. was supported by NSF grant DMS-1600653.
D.S. was partially supported by 
JSPS Grant-in-Aid for Scientific Research (C) 15K04803, (C) 19K03415. 
%
%
\section{Semi-infinite Bruhat order and quantum Bruhat graph.}
\label{sec:BG}
%
%
\subsection{Affine Lie algebras.}
\label{subsec:liealg}

Let $\Fg$ be a finite-dimensional simple Lie algebra over $\BC$
with Cartan subalgebra $\Fh$. 
Denote by $\{ \alpha_{i}^{\vee} \}_{i \in I}$ and 
$\{ \alpha_{i} \}_{i \in I}$ the set of simple coroots and 
simple roots of $\Fg$, respectively, and set
$Q := \bigoplus_{i \in I} \BZ \alpha_i$, 
$Q^{+} := \sum_{i \in I} \BZ_{\ge 0} \alpha_i$, and 
$Q^{\vee} := \bigoplus_{i \in I} \BZ \alpha_i^{\vee}$, 
$Q^{\vee,+} := \sum_{i \in I} \BZ_{\ge 0} \alpha_i^{\vee}$; 
for $\xi,\,\zeta \in Q^{\vee}$, we write $\xi \ge \zeta$ if $\xi-\zeta \in Q^{\vee,+}$. 
Let $\Delta$, $\Delta^{+}$, and $\Delta^{-}$ be 
the set of roots, positive roots, and negative roots of $\Fg$, respectively, 
with $\theta \in \Delta^{+}$ the highest root of $\Fg$; 
we set $\rho:=(1/2) \sum_{\alpha \in \Delta^{+}} \alpha$. 
Also, let $\vpi_{i}$, $i \in I$, denote the fundamental weights for $\Fg$, and set
%
%
\begin{equation} \label{eq:P-fin}
P:=\bigoplus_{i \in I} \BZ \vpi_{i}, \qquad 
P^{+} := \sum_{i \in I} \BZ_{\ge 0} \vpi_{i}. 
\end{equation} 

Let $\Fg_{\af} = \bigl(\BC[z,z^{-1}] \otimes \Fg\bigr) \oplus \BC c \oplus \BC d$ be 
the (untwisted) affine Lie algebra over $\BC$ associated to $\Fg$, 
where $c$ is the canonical central element and $d$ is 
the scaling element (or degree operator), 
with Cartan subalgebra $\Fh_{\af} = \Fh \oplus \BC c \oplus \BC d$. 
We regard an element $\mu \in \Fh^{\ast}:=\Hom_{\BC}(\Fh,\,\BC)$ as an element of 
$\Fh_{\af}^{\ast}$ by setting $\pair{\mu}{c}=\pair{\mu}{d}:=0$, where 
$\pair{\cdot\,}{\cdot}:\Fh_{\af}^{\ast} \times \Fh_{\af} \rightarrow \BC$ denotes
the canonical pairing of $\Fh_{\af}^{\ast}:=\Hom_{\BC}(\Fh_{\af},\,\BC)$ and $\Fh_{\af}$. 
Let $\{ \alpha_{i}^{\vee} \}_{i \in I_{\af}} \subset \Fh_{\af}$ and 
$\{ \alpha_{i} \}_{i \in I_{\af}} \subset \Fh_{\af}^{\ast}$ be the set of 
simple coroots and simple roots of $\Fg_{\af}$, respectively, 
where $I_{\af}:=I \sqcup \{0\}$; note that 
$\pair{\alpha_{i}}{c}=0$ and $\pair{\alpha_{i}}{d}=\delta_{i,0}$ 
for $i \in I_{\af}$. 
Denote by $\delta \in \Fh_{\af}^{\ast}$ the null root of $\Fg_{\af}$; 
recall that $\alpha_{0}=\delta-\theta$. 
Also, let $\Lambda_{i} \in \Fh_{\af}^{\ast}$, $i \in I_{\af}$, 
denote the fundamental weights for $\Fg_{\af}$ such that $\pair{\Lambda_{i}}{d}=0$, 
and set 
%
%
\begin{equation} \label{eq:P}
P_{\af} := 
  \left(\bigoplus_{i \in I_{\af}} \BZ \Lambda_{i}\right) \oplus 
   \BZ \delta \subset \Fh^{\ast}, \qquad 
P_{\af}^{0}:=\bigl\{\mu \in P_{\af} \mid \pair{\mu}{c}=0\bigr\};
\end{equation}
notice that $P_{\af}^{0}=P \oplus \BZ \delta$, and that
$\pair{\mu}{\alpha_{0}^{\vee}} = - \pair{\mu}{\theta^{\vee}}$ 
for $\mu \in P_{\af}^{0}$. We remark that for each $i \in I$, 
$\vpi_{i}$ is identical to $\Lambda_{i}-\pair{\Lambda_{i}}{c}\Lambda_{0}$, 
which is called the level-zero fundamental weight in \cite{Kas02}. 

Let $W := \langle s_{i} \mid i \in I \rangle$ and 
$W_{\af} := \langle s_{i} \mid i \in I_{\af} \rangle$ be the (finite) Weyl group of $\Fg$ and 
the (affine) Weyl group of $\Fg_{\af}$, respectively, 
where $s_{i}$ is the simple reflection with respect to $\alpha_{i}$ 
for $i \in I_{\af}$. We denote by $\ell:W_{\af} \rightarrow \BZ_{\ge 0}$ 
the length function on $W_{\af}$, whose restriction to $W$ agrees with 
the one on $W$, by $e \in W \subset W_{\af}$ the identity element, 
and by $\lng \in W$ the longest element. 
For each $\xi \in Q^{\vee}$, let $t_{\xi} \in W_{\af}$ denote 
the translation in $\Fh_{\af}^{\ast}$ by $\xi$ (see \cite[Sect.~6.5]{Kac}); 
for $\xi \in Q^{\vee}$, we have 
%
%
\begin{equation}\label{eq:wtmu}
t_{\xi} \mu = \mu - \pair{\mu}{\xi}\delta \quad 
\text{if $\mu \in \Fh_{\af}^{\ast}$ satisfies $\pair{\mu}{c}=0$}.
\end{equation}
Then, $\bigl\{ t_{\xi} \mid \xi \in Q^{\vee} \bigr\}$ forms 
an abelian normal subgroup of $W_{\af}$, in which $t_{\xi} t_{\zeta} = t_{\xi + \zeta}$ 
holds for $\xi,\,\zeta \in Q^{\vee}$. Moreover, we know from \cite[Proposition 6.5]{Kac} that
\begin{equation*}
W_{\af} \cong 
 W \ltimes \bigl\{ t_{\xi} \mid \xi \in Q^{\vee} \bigr\} \cong W \ltimes Q^{\vee}. 
\end{equation*}

Denote by $\rr$ the set of real roots of $\Fg_{\af}$, and 
by $\prr \subset \rr$ the set of positive real roots; 
we know from \cite[Proposition~6.3]{Kac} that
$\rr = 
\bigl\{ \alpha + n \delta \mid \alpha \in \Delta,\, n \in \BZ \bigr\}$ 
and 
$\prr = 
\Delta^{+} \sqcup 
\bigl\{ \alpha + n \delta \mid \alpha \in \Delta,\, n \in \BZ_{> 0}\bigr\}$. 
For $\beta \in \rr$, we denote by $\beta^{\vee} \in \Fh_{\af}$ 
its dual root, and by $s_{\beta} \in W_{\af}$ the corresponding reflection; 
if $\beta \in \rr$ is of the form $\beta = \alpha + n \delta$ 
with $\alpha \in \Delta$ and $n \in \BZ$, then 
$s_{\beta} =s_{\alpha} t_{n\alpha^{\vee}} \in W \ltimes Q^{\vee}$.

Finally, let $U_{\q}(\Fg_{\af})$ (resp., $U_{\q}'(\Fg_{\af})$)
denote the quantized universal enveloping algebra over $\BC(\q)$ 
associated to $\Fg_{\af}$ (resp., $[\Fg_{\af},\Fg_{\af}]$), 
with $E_{i}$ and $F_{i}$, $i \in I_{\af}$, the Chevalley generators 
corresponding to $\alpha_{i}$. 
We denote by $U_{\q}^{-}(\Fg_{\af})$ 
the negative part of $U_{\q}(\Fg_{\af})$, that is, 
the $\BC(\q)$-subalgebra of $U_{\q}(\Fg_{\af})$ 
generated by the $F_{i}$, $i \in I_{\af}$. 
%
%
\subsection{Parabolic semi-infinite Bruhat graph.}
\label{subsec:SiBG}

In this subsection, we take and fix an arbitrary subset $\J \subset I$. We set 
$\QJ := \bigoplus_{i \in \J} \BZ \alpha_i$, 
$\QJv := \bigoplus_{i \in \J} \BZ \alpha_i^{\vee}$,  
$\QJvp := \sum_{i \in \J} \BZ_{\ge 0} \alpha_i^{\vee}$, 
$\DeJ := \Delta \cap \QJ$, 
$\DeJ^{\pm} := \Delta^{\pm} \cap \QJ$, 
$\WJ := \langle s_{i} \mid i \in \J \rangle$, and 
$\rho_{\J}:=(1/2) \sum_{\alpha \in \DeJ^{+}} \alpha$; 
we denote by
\begin{equation} \label{eq:pr}
[\,\cdot\,]^{\J} : Q^{\vee} \twoheadrightarrow Q_{\Jc}^{\vee} \quad 
\text{(resp., $[\,\cdot\,]_{\J} : Q^{\vee} \twoheadrightarrow \QJv$)}
\end{equation}
the projection from $Q^{\vee}=Q_{\Jc}^{\vee} \oplus \QJv$
onto $Q_{\Jc}^{\vee}$ (resp., $\QJv$) with kernel $\QJv$ (resp., $Q_{\Jc}^{\vee}$). 
Let $\WJu$ denote the set of minimal(-length) coset representatives 
for the cosets in $W/\WJ$; we know from \cite[Sect.~2.4]{BB} that 
%
%
\begin{equation} \label{eq:mcr}
\WJu = \bigl\{ w \in W \mid 
\text{$w \alpha \in \Delta^{+}$ for all $\alpha \in \DeJ^{+}$}\bigr\}.
\end{equation}
For $w \in W$, we denote by $\mcr{w}=\mcr{w}^{\J} \in \WJu$ 
the minimal coset representative for the coset $w \WJ$ in $W/\WJ$.
Also, we set
\begin{align}
(\DeJ)_{\af} 
  & := \bigl\{ \alpha + n \delta \mid 
  \alpha \in \DeJ,\,n \in \BZ \bigr\} \subset \Delta_{\af}, \\
(\DeJ)_{\af}^{+}
  &:= (\DeJ)_{\af} \cap \prr = 
  \DeJ^{+} \sqcup \bigl\{ \alpha + n \delta \mid 
  \alpha \in \DeJ,\, n \in \BZ_{> 0} \bigr\}, \\
\label{eq:stabilizer}
(\WJ)_{\af} 
 & := \WJ \ltimes \bigl\{ t_{\xi} \mid \xi \in \QJv \bigr\}
   = \bigl\langle s_{\beta} \mid \beta \in (\DeJ)_{\af}^{+} \bigr\rangle, \\
\label{eq:Pet}
(\WJu)_{\af}
 &:= \bigl\{ x \in W_{\af} \mid 
 \text{$x\beta \in \prr$ for all $\beta \in (\DeJ)_{\af}^{+}$} \bigr\};
\end{align}
if $\J = \emptyset$, then 
$(W^{\emptyset})_{\af}=W_{\af}$ and $(W_{\emptyset})_{\af}=\bigl\{e\bigr\}$. 
We know from \cite{Pet97} (see also \cite[Lemma~10.6]{LS10}) that 
for each $x \in W_{\af}$, there exist a unique 
$x_1 \in (\WJu)_{\af}$ and a unique $x_2 \in (\WJ)_{\af}$ 
such that $x = x_1 x_2$; let 
%
%
\begin{equation} \label{eq:PiJ}
\PJ : W_{\af} \twoheadrightarrow (\WJu)_{\af}, \quad x \mapsto x_{1}, 
\end{equation}
denote the projection, 
where $x= x_1 x_2$ with $x_1 \in (\WJu)_{\af}$ and $x_2 \in (\WJ)_{\af}$. 
%
%
\begin{lem}[{see, e.g., \cite[Lemma~2.1]{NS18}}] \label{lem:PiJ}
\mbox{}
\begin{enu}
\item It holds that 
%
%
\begin{equation} \label{eq:PiJ2}
\begin{cases}
\PJ(w)=\mcr{w}^{\J} 
  & \text{\rm for all $w \in W$},  \\[1mm]
\PJ(xt_{\xi})=\PJ(x)\PJ(t_{\xi}) 
  & \text{\rm for all $x \in W_{\af}$ and $\xi \in Q^{\vee}$};
\end{cases}
\end{equation}
in particular, 
$(\WJu)_{\af} 
  = \bigl\{ w \PJ(t_{\xi}) \mid w \in \WJu,\,\xi \in Q^{\vee} \bigr\}$.

\item For each $\xi \in Q^{\vee}$, 
the element $\PJ(t_{\xi}) \in (\WJu)_{\af}$ is 
of the form{\rm:} $\PJ(t_{\xi})=ut_{\xi+\gamma}$ 
for some $u \in \WJ$ and $\gamma \in \QJv$. 

\item For $\xi,\,\zeta \in Q^{\vee}$, 
$\PJ(t_{\xi}) = \PJ(t_{\zeta})$ if and only if $\xi-\zeta \in \QJv$.

\end{enu}
\end{lem}
%
%
\begin{dfn} \label{dfn:sell}
Let $x \in W_{\af}$, and 
write it as $x = w t_{\xi}$ with $w \in W$ and $\xi \in Q^{\vee}$. 
We define the semi-infinite length $\sell(x)$ of $x$ by:
$\sell (x) = \ell (w) + 2 \pair{\rho}{\xi}$. 
\end{dfn}
%
%
\begin{dfn} \label{dfn:SiB}
\mbox{}
\begin{enu}
\item The (parabolic) semi-infinite Bruhat graph $\SBJ$ 
is the $\prr$-labeled directed graph whose 
vertices are the elements of $(\WJu)_{\af}$, and 
whose directed edges are of the form: 
$x \edge{\beta} y$ for $x,y \in (\WJu)_{\af}$ and $\beta \in \prr$ 
such that $y=s_{\beta}x$ and $\sell (y) = \sell (x) + 1$. 
When $\J=\emptyset$, we write $\SB$ for 
$\mathrm{BG}^{\si}\bigl((W^{\emptyset})_{\af}\bigr)$. 

\item 
The (parabolic) semi-infinite Bruhat order is a partial order 
$\sile$ on $(\WJu)_{\af}$ defined as follows: 
for $x,\,y \in (\WJu)_{\af}$, we write $x \sile y$ 
if there exists a directed path in $\SBJ$ from $x$ to $y$; 
we write $x \sil y$ if $x \sile y$ and $x \ne y$. 
\end{enu}
\end{dfn}

\begin{rem} \label{rem:Bruhat}
On the (finite) Weyl group $W$, 
the semi-infinite Bruhat order agrees with 
the ordinary Bruhat order. 
\end{rem}

Let us recall some of 
the basic properties of the semi-infinite Bruhat order. 
Take and fix $\lambda \in P^{+}$ such that 
$\bigl\{ i \in I \mid \pair{\lambda}{\alpha_{i}^{\vee}}=0 \bigr\} = \J$. 
%
%
\begin{lem}[{\cite[Remark~4.1.3]{INS}}] \label{lem:si}
Let $x \in \WJa$ and $i \in I_{\af}$. Then, 
\begin{equation} \label{eq:si1}
s_{i}x \in \WJa \iff 
\pair{x\lambda}{\alpha_{i}^{\vee}} \ne 0 \iff 
x^{-1}\alpha_{i} \in (\Delta \setminus \DeJ)+\BZ\delta.
\end{equation}
Moreover, in this case, 
%
%
\begin{equation} \label{eq:simple}
\begin{cases}
x \edge{\alpha_{i}} s_{i}x \iff
\pair{x\lambda}{\alpha_{i}^{\vee}} > 0 \iff 
x^{-1}\alpha_{i} \in (\Delta^{+} \setminus \DeJ^{+})+\BZ\delta, & \\[1.5mm]
s_{i}x \edge{\alpha_{i}} x  \iff 
\pair{x\lambda}{\alpha_{i}^{\vee}} < 0 \iff 
x^{-1}\alpha_{i} \in (\Delta^{-} \setminus \DeJ^{-})+\BZ\delta. & 
\end{cases}
\end{equation}
\end{lem}
%
%
\begin{rem} \label{rem:si}
Let $x \in \WJa$ and $i \in I_{\af}$. Then, 
\begin{equation}
\PJ(s_{i}x) = x \iff 
\pair{x\lambda}{\alpha_{i}^{\vee}} = 0 \iff 
x^{-1}\alpha_{i} \in \DeJ+\BZ\delta.
\end{equation}
\end{rem}
%
%
\begin{lem}[{\cite[Lemma~2.3.6]{NS16}}] \label{lem:dmd}
Let $x,\,y \in \WJa$ be such that $x \sile y$, and let $i \in I_{\af}$. 
\begin{enu}

\item If $\pair{x\lambda}{\alpha_{i}^{\vee}} > 0$ and 
$\pair{y\lambda}{\alpha_{i}^{\vee}} \le 0$, then $s_{i}x \sile y$.

\item If $\pair{x\lambda}{\alpha_{i}^{\vee}} \ge 0$ and 
$\pair{y\lambda}{\alpha_{i}^{\vee}} < 0$, then $x \sile s_{i}y$. 

\item If $\pair{x\lambda}{\alpha_{i}^{\vee}} > 0$ and 
$\pair{y\lambda}{\alpha_{i}^{\vee}} > 0$, or 
if $\pair{x\lambda}{\alpha_{i}^{\vee}} < 0$ and 
$\pair{y\lambda}{\alpha_{i}^{\vee}} < 0$, then $s_{i}x \sile s_{i}y$.

\end{enu}
\end{lem}
%
%
\begin{lem}[{\cite[Lemmas~4.3.3--4.3.5]{NNS1}}] \label{lem:SiB}
\mbox{}
\begin{enu}

\item 
Let $w,\,v \in \WJu$, and $\xi,\,\zeta \in Q^{\vee}$. 
If $w\PJ(t_{\xi}) \sige v\PJ(t_{\zeta})$, 
then $[\xi]^{\J} \ge [\zeta]^{\J}${\rm;}
for the projection $[\,\cdot\,]^{\J}:
Q^{\vee} \twoheadrightarrow Q^{\vee}_{\Jc}$, see \eqref{eq:pr}. 

\item 
Let $w \in \WJu$, and $\xi,\,\zeta \in Q^{\vee}$. 
Then, $w\PJ(t_{\xi}) \sige w\PJ(t_{\zeta})$ 
if and only if $[\xi]^{\J} \ge [\zeta]^{\J}$. 

\item 
Let $x,\,y \in \WJa$ and $\beta \in \prr$ be such that 
$x \edge{\beta} y$ in $\SBJ$. Then, $\PJ(xt_{\xi}) \edge{\beta} \PJ(yt_{\xi})$ 
in $\SBJ$ for all $\xi \in Q^{\vee}$. Therefore, if $x \sige y$, then 
$\PJ(xt_{\xi}) \sige \PJ(yt_{\xi})$ for all $\xi \in Q^{\vee}$. 
\end{enu}
\end{lem}
%
%
\begin{lem}[{\cite[Lemma~6.1.1]{INS}}] \label{lem:611}
If $x,\,y \in W_{\af}$ satisfy $x \sige y$, then 
$\PJ(x) \sige \PJ(y)$. 
\end{lem}

We denote by $\ast:I \rightarrow I$, $i \mapsto i^{\ast}$, 
the Dynkin diagram automorphism (of order $1$ or $2$) 
induced by the longest element $\lng \in W$ as: 
$\lng(\alpha_{i})=-\alpha_{i^{\ast}}$ for $i \in I$, and then 
define the $\BZ$-linear automorphism 
$\ast:P \oplus \BZ\delta \rightarrow P \oplus \BZ\delta$ 
(resp., $\ast:Q^{\vee} \rightarrow Q^{\vee}$) by: 
$\vpi_{i}^{\ast} = \vpi_{i^{\ast}}$ for $i \in I$ and 
$\delta^{\ast}=\delta$ (resp., 
$(\alpha_{i}^{\vee})^{\ast} = \alpha_{i^{\ast}}^{\vee}$). 
Also, we define the group automorphism 
$\ast:W \rightarrow W$ by: $s_{i}^{\ast}=s_{i^{\ast}}$ for $i \in I$. 

We see by Lemma~\ref{lem:PiJ} that 
if $x \in \WJa$ is of the form 
$x = w \PJ(t_{\xi})$ for some $w \in \WJu$
and $\xi \in Q^{\vee}$, then 
\begin{equation} \label{eq:lng1}
\PS{\Ja}(x\lng) = \mcr{w\lng}^{\Ja} \PS{\Ja}(t_{-\xi^{\ast}}), 
\end{equation}
where $\Ja:=\bigl\{j^{\ast} \mid j \in J\bigr\} \subset I$; notice that 
\begin{equation} \label{eq:lng2}
\PJ(\PS{\Ja}(x\lng)\lng) = x.
\end{equation}
The next lemma follows from \cite[Proposition~A.1.2]{INS} and 
\cite[Proposition~4.3]{LNSSS}. 
%
%
\begin{lem} \label{lem:lng1}
Let $x,y \in \WJa$, and $\beta \in \prr$.
Then, $x \edge{\beta} y$ in $\SBJ$
if and only if $\PS{\Ja}(y\lng) \edge{\beta} \PS{\Ja}(x\lng)$ in $\SBS{\Ja}$. 
Therefore, $x \sige y$ in $\WJa$ if and only if 
$\PS{\Ja}(y\lng) \sige \PS{\Ja}(x\lng)$ in $\WSa{\Ja}$. 
\end{lem}
%
%
\subsection{Deodhar lifts.}
\label{subsec:deodhar}

Let $\J$ be a subset of $I$. For $x \in (\WJu)_{\af}$, 
let $\Lift(x)$ denote the set of lifts of $x$ in $W_{\af}$ 
for the projection $\PJ:W_{\af} \twoheadrightarrow \WJa$, that is, 
$\Lift(x):=\bigl\{ x' \in W_{\af} \mid \PJ(x') = x \bigr\}$; 
we know from \cite[Lemma~B.1]{KNS} that 
if $x = w \PJ(t_{\xi}) \in \WJa$ for $w \in \WJu$ and 
$\xi \in Q^{\vee}$ (see Lemma~\ref{lem:PiJ}\,(1)), then 
$\Lift(x) = \bigl\{ w' t_{\xi+\gamma} \mid 
 w' \in w\WJ,\, \gamma \in \QJv \bigr\}$. 
%
%
\begin{prop}[{\cite[Proposition~2.4]{KNS}}] \label{prop:dmin}
If $y \in W_{\af}$ and $x \in \WJa$ satisfy the condition that 
$x \sige \PJ(y)$, then the set
\begin{equation}
\Lige{y}{x}:= \bigl\{ x' \in \Lift(x) \mid x' \sige y \bigr\}
\end{equation}
has a {\rm(}necessarily unique{\rm)} minimum element with respect to 
the semi-infinite Bruhat order on $W_{\af}${\rm;}
we denote this element by $\min \Lige{y}{x}$. 
\end{prop}
%
%
\begin{prop} \label{prop:dmax}
If $y \in W_{\af}$ and $x \in \WJa$ satisfy the condition that 
$\PJ(y) \sige x$, then the set
\begin{equation}
\Lile{y}{x}:= \bigl\{ x' \in \Lift(x) \mid y \sige x'\bigr\}
\end{equation}
has a {\rm(}necessarily unique{\rm)} maximum element 
with respect to the semi-infinite Bruhat order on $W_{\af}${\rm;}
we denote this element by $\max \Lile{y}{x}$. 
\end{prop}

\begin{proof}
We first show that $x'\lng \in \Lige{y\lng}{\PS{\Ja}(x\lng)}$ 
for all $x' \in \Lile{y}{x}$. 
We see by Lemma~\ref{lem:lng1} that 
$x'\lng \sige y\lng$. Hence it follows from Lemma~\ref{lem:611} that 
$\PS{\Ja}(x'\lng) \sige \PS{\Ja}(y\lng)$. Also, since $\PJ(x')=x$, 
there exist $z \in \WJ$ and $\gamma \in \QJv$ such that 
$x' = x zt_{\gamma}$; note that $x'\lng = x zt_{\gamma}\lng = 
x\lng z^{\ast} t_{-\gamma^{\ast}}$, with 
$z^{\ast} \in \WS{\Ja}$ and $\gamma^{\ast} \in \QSv{\Ja}$. 
Therefore, we deduce that $\PS{\Ja}(x'\lng) = \PS{\Ja}(x\lng)$, 
which implies that $x'\lng \in \Lige{y\lng}{\PS{\Ja}(x\lng)}$, 
as desired. Similarly, we can show that 
if $x'' \in \Lige{y\lng}{\PS{\Ja}(x\lng)}$, then 
$x''\lng \in \Lile{y}{x}$ (see also \eqref{eq:lng2}). 
Recall from Proposition~\ref{prop:dmin} that 
$\min \Lige{y\lng}{\PS{\Ja}(x\lng)}$ denotes 
the minimum element of $\Lige{y\lng}{\PS{\Ja}(x\lng)}$. 
From Lemma~\ref{lem:lng1} and the assertions above, 
we deduce that $\bigl(\min \Lige{y\lng}{\PS{\Ja}(x\lng)}\bigr)\lng$ 
is the maximum element of $\Lile{y}{x}$. 
This proves the lemma. 
\end{proof}
%
%
\begin{cor} \label{cor:deo} \mbox{}
\begin{enu}
\item If $y \in W_{\af}$ and $x \in \WJa$ satisfy the condition that 
$\PJ(y) \sige x$, then $\PS{\Ja}(x\lng) \sige \PS{\Ja}(y\lng)$ and 
%
%
\begin{equation} \label{eq:deo1}
\max \Lile{y}{x} = 
\bigl(\min \Lige{y\lng}{\PS{\Ja}(x\lng)}\bigr)\lng. 
\end{equation}

\item If $y \in W_{\af}$ and $x \in \WJa$ satisfy the condition that 
$x \sige \PJ(y)$, then $\PS{\Ja}(y\lng) \sige \PS{\Ja}(x\lng)$ and
%
%
\begin{equation} \label{eq:deo2}
\min \Lige{y}{x} = 
\bigl(\max \Lile{y\lng}{\PS{\Ja}(x\lng)}\bigr)\lng. 
\end{equation}
\end{enu}
\end{cor}
%
%
\begin{lem}[{cf. \cite[Lemma~3.6]{NNS3}}] \label{lem:dia}
Let $\J$ be a subset of $I$. Let $y \in W_{\af}$ and $x \in \WJa$ be 
such that $x \sige \PJ(y)$, and set $\ti{x}:=\min \Lige{y}{x}$. 
Let $i \in I_{\af}$ be such that $y^{-1}\alpha_{i} \in \Delta^{+}+\BZ\delta$. 
\begin{enu}
\item Assume that $x^{-1}\alpha_{i} \in (\Delta^{+} \setminus \DeJ^{+})+\BZ\delta${\rm;} 
note that $s_{i}x \in \WJa$ by Lemma~\ref{lem:si}. 
It holds that $s_{i}x \sige \PJ(s_{i}y)$ and 
$\min \Lige{s_{i}y}{s_{i}x} = s_{i}\ti{x}$. 

\item If $x^{-1}\alpha_{i} \in (\Delta^{-} \setminus \DeJ^{-})+\BZ\delta$, 
then $x \sige \PJ(s_{i}y)$ and $\min \Lige{s_{i}y}{x} = \ti{x}$. 

\item If $x^{-1}\alpha_{i} \in \DeJ+\BZ\delta$, 
then $\ti{x}^{-1}\alpha_{i} \in \Delta^{+}+\BZ\delta$. 
Moreover, we have $x \sige \PJ(s_{i}y)$, and 
$\min \Lige{s_{i}y}{x}$ is identical to $\ti{x}$ or $s_{i}\ti{x}$. 
\end{enu}
\end{lem}
\begin{proof}
Although we can show this lemma by combining 
\cite[Lemma~3.6]{NNS3} and Proposition~\ref{prop:mins} below, 
we give a direct proof below.

(1) Since $x^{-1}\alpha_{i} \in (\Delta^{+} \setminus \DeJ^{+})+\BZ\delta$ by the assumption, 
and since $\Delta^{+} \setminus \DeJ^{+}$ is stable under the action of $\WJ$, 
we see that $\ti{x}^{-1}\alpha_{i} \in(\Delta^{+} \setminus \DeJ^{+})+\BZ\delta$. 
Also, we have $y^{-1}\alpha_{i} \in \Delta^{+}+\BZ\delta$ by the assumption, 
and $\ti{x} \sige y$ by the definition of $\ti{x}$. 
Therefore, it follows from Lemma~\ref{lem:dmd}\,(3) that $s_{i}\ti{x} \sige s_{i}y$, 
and hence $\PJ(s_{i}\ti{x}) \sige \PJ(s_{i}y)$ by Lemma~\ref{lem:611}. 
Here it is easily checked that $\PJ(s_{i}\ti{x}) = s_{i}x$. 
From these, we deduce that $s_{i}x \sige \PJ(s_{i}y)$ and 
$s_{i}\ti{x} \in \Lige{s_{i}y}{s_{i}x}$. 
If we set $z:=\min \Lige{s_{i}y}{s_{i}x}$, then 
$s_{i}\ti{x} \sige z \sige s_{i}y$. 
Since $\PJ(z) = s_{i}x$ and 
$(s_{i}x)^{-1}\alpha_{i} = -x^{-1}\alpha_{i} \in (\Delta^{-} \setminus \DeJ^{-})+\BZ\delta$, 
it is easily seen that 
$z^{-1}\alpha_{i} \in (\Delta^{-} \setminus \DeJ^{-})+\BZ\delta$. 
Therefore, it follows from Lemma~\ref{lem:dmd}\,(3) that 
$\ti{x} \sige s_{i}z \sige y$. 
Here we note that $\PJ(s_{i}z) = x$, 
and hence $s_{i}z \in \Lige{y}{x}$. 
Since $\ti{x} = \min \Lige{y}{x}$, we obtain
$\ti{x} = s_{i}z$, and hence $z = s_{i}\ti{x}$, as desired. 

(2) As in the proof of (1), we see that 
$\ti{x}^{-1}\alpha_{i} \in(\Delta^{-} \setminus \DeJ^{-})+\BZ\delta$. 
Also, we have $y^{-1}\alpha_{i} \in \Delta^{+}+\BZ\delta$ by the assumption, 
and $\ti{x} \sige y$ by the definition of $\ti{x}$. 
Therefore, it follows from Lemma~\ref{lem:dmd}\,(1) 
that $\ti{x} \sige s_{i}y$, which implies that 
$x = \PJ(\ti{x}) \sige \PJ(s_{i}y)$ by Lemma~\ref{lem:611}. 
Thus we have $\ti{x} \in \Lige{s_{i}y}{x}$. 
If we set $z:=\min \Lige{s_{i}y}{x}$, then 
$\ti{x} \sige z \sige s_{i}y$. 
Since $y^{-1}\alpha_{i} \in \Delta^{+}+\BZ\delta$, 
we deduce that $s_{i}y \sig y$, and hence 
$\ti{x} \sige z \sige y$. 
Since $\ti{x} = \min \Lige{y}{x}$, we obtain
$z = \ti{x}$, as desired. 

(3) Suppose, for a contradiction, that 
$\ti{x}^{-1}\alpha_{i} \in \Delta^{-}+\BZ\delta$.
Then we have $\ti{x} \sig s_{i}\ti{x}$. 
Since $y^{-1}\alpha_{i} \in \Delta^{+}+\BZ\delta$, 
it follows from Lemma~\ref{lem:dmd}\,(2) that $s_{i}\ti{x} \sige y$. Also, since 
$x^{-1}\alpha_{i} \in \DeJ+\BZ\delta$, 
we see that $\PJ(s_{i}\ti{x}) = x$. 
Hence we deduce that 
$s_{i}\ti{x} \in \Lige{y}{x}$. However, 
since $\ti{x} \sig s_{i}\ti{x}$ as seen above, 
this contradicts $\ti{x} = \min \Lige{y}{x}$. 
Thus we conclude that 
$\ti{x}^{-1}\alpha_{i} \in \Delta^{+}+\BZ\delta$. 

Since $\ti{x}^{-1}\alpha_{i},\,y^{-1}\alpha_{i} \in \Delta^{+}+\BZ\delta$, 
and $\ti{x} \sige y$, it follows from Lemma~\ref{lem:dmd}\,(3) that 
$s_{i}\ti{x} \sige s_{i}y$. 
Therefore, by Lemma~\ref{lem:611}, we have 
$x = \PJ(s_{i}\ti{x}) \sige \PJ(s_{i}y)$, 
and hence $s_{i}\ti{x} \in \Lige{s_{i}y}{x}$. 
Here we set $z:=\min \Lige{s_{i}y}{x}$.
Assume first that $\ti{x} \sige s_{i}y$. 
Then we have $\ti{x} \sige z \sige s_{i}y$. 
Since $y^{-1}\alpha_{i} \in \Delta^{+}+\BZ\delta$, 
we have $s_{i}y \sig y$, and hence 
$\ti{x} \sige z \sige y$. Since $\ti{x} = \min \Lige{y}{x}$, 
we obtain $z = \ti{x}$. 
Assume next that $\ti{x} \not\sige s_{i}y$.
Since $s_{i}\ti{x} \in \Lige{s_{i}y}{x}$ as seen above, 
we have $s_{i}\ti{x} \sige z \sige s_{i}y$. 
Note that $(s_{i}\ti{x})^{-1}\alpha_{i} \in \Delta^{-}+\BZ\delta$. 
If $z^{-1}\alpha_{i} \in \Delta^{+}+\BZ\delta$, 
then it follows from Lemma~\ref{lem:dmd}\,(2) that 
$\ti{x} \sige z \sige s_{i}y$, which contradicts 
the assumption that $\ti{x} \not\sige s_{i}y$.
Hence we obtain $z^{-1}\alpha_{i} \in \Delta^{-}+\BZ\delta$.
Since $(s_{i}\ti{x})^{-1}\alpha_{i}$ and $(s_{i}y)^{-1}\alpha_{i}$ are 
contained in $\Delta^{-}+\BZ\delta$, 
we see by Lemma~\ref{lem:dmd}\,(3) that $\ti{x} \sige s_{i}z \sige y$. 
Since $x^{-1}\alpha_{i} \in \DeJ+\BZ\delta$ and $z \in \Lift(x)$, 
it is easily checked that $\PJ(s_{i}z) = x$. 
Thus, $s_{i}z \in \Lige{y}{x}$. 
Since $\ti{x} = \min \Lige{y}{x}$, we obtain 
$s_{i}z = \ti{x}$, and hence $z=s_{i}\ti{x}$. 
This proves the lemma. 
\end{proof}

%
\subsection{Quantum Bruhat graph.}
\label{subsec:QBG}

We take and fix a subset $\J$ of $I$. 

\begin{dfn}
The (parabolic) quantum Bruhat graph $\QBJ$ is 
the ($\Delta^{+} \setminus \DeJ^{+})$-labeled
directed graph whose vertices are the elements of $\WJu$, and 
whose directed edges are of the form: $w \edge{\beta} v$ 
for $w,v \in \WJu$ and $\beta \in \Delta^{+} \setminus \DeJ^{+}$ 
such that $v= \mcr{ws_{\beta}}$, and such that either of 
the following holds: 
(i) $\ell(v) = \ell (w) + 1$; 
(ii) $\ell(v) = \ell (w) + 1 - 2 \pair{\rho-\rho_{\J}}{\beta^{\vee}}$.
An edge satisfying (i) (resp., (ii)) is called a Bruhat (resp., quantum) edge. 
When $\J=\emptyset$, we write $\QB$ for $\mathrm{QBG}(W^{\emptyset})$. 
\end{dfn}
%
%
\begin{rem} \label{rem:PQBG}
We know from \cite[Remark~6.13]{LNSSS} that for each $w,\,v \in \WJu$, 
there exists a directed path in $\QBJ$ from $w$ to $v$.
\end{rem}

Let $w,\,v \in \WJu$, and let 
$\bp:w=
 v_{0} \edge{\beta_{1}}
 v_{1} \edge{\beta_{2}} \cdots 
       \edge{\beta_{l}}
 v_{l}=v$
be a directed path in $\QBJ$ from $w$ to $v$. 
We define the weight $\wt^{\J}(\bp)$ of $\bp$ by
%
%
\begin{equation} \label{eq:wtdp}
\wt^{\J}(\bp) := \sum_{
 \begin{subarray}{c}
 1 \le k \le l\,; \\[1mm]
 \text{$v_{k-1} \edge{\beta_{k}} v_{k}$ is} \\[1mm]
 \text{a quantum edge}
 \end{subarray}}
\beta_{k}^{\vee} \in Q^{\vee,+}; 
\end{equation}
when $\J=\emptyset$, we write $\wt(\bp)$ for $\wt^{\emptyset}(\bp)$. 
For $w,\,v \in \WJu$, we take a shortest directed path $\bp$ in 
$\QBJ$ from $w$ to $v$, and set 
$\wt^{\J}(w \Rightarrow v):=[\wt^{\J}(\bp)]^{\J} \in Q_{\Jc}^{\vee,+}$; 
we know from \cite[Sect.~4.1]{LNSSS2} that 
$\wt^{\J}(w \Rightarrow v)$ does not depend on the choice of 
a shortest directed path $\bp$. 
When $\J=\emptyset$, we write $\wt(u \Rightarrow v)$ 
for $\wt^{\emptyset}(u \Rightarrow v)$. 
%
%
\begin{lem}[{\cite[Lemma~7.2]{LNSSS2}}] \label{lem:wtS} \mbox{}
Let $w,\,v \in \WJu$, and let $w_{1} \in w\WJ$, $v_{1} \in v\WJ$. 
Then we have $\wt^{\J}(w \Rightarrow v) = [\wt(w_{1} \Rightarrow v_{1})]^{\J}$. 
\end{lem}
%
%
\begin{lem}[{\cite[Lemmas~4.3.6 and 4.3.7]{NNS1}}] \label{lem:siwt}
Let $w,\,v \in \WJu$, and $\xi,\zeta \in Q^{\vee}$. Then, 
\begin{equation}
w\PJ(t_{\xi}) \sige v\PJ(t_{\gamma}) \iff 
[\xi]^{\J} \ge [\gamma+\wt(v \Rightarrow w)]^{\J}. 
\end{equation}
\end{lem}
%
%
\subsection{Tilted Bruhat order.}
\label{subsec:tilted}
For $w,\,v \in W$, we denote by $\ell(w \Rightarrow v)$ 
the length of a shortest directed path from $w$ to $v$ in 
$\QB = \mathrm{QBG}(W^{\emptyset})$.
%
%
\begin{dfn}[tilted Bruhat order] \label{dfn:tilted}
For each $v \in W$, we define the $v$-tilted Bruhat order $\tb{v}$ on $W$ as follows:
for $w_{1},w_{2} \in W$, 
%
%
\begin{equation} \label{eq:tilted}
w_{1} \tb{v} w_{2} \iff \ell(v \Rightarrow w_{2}) = 
 \ell(v \Rightarrow w_{1}) + \ell(w_{1} \Rightarrow w_{2}).
\end{equation}
Namely, $w_{1} \tb{v} w_{2}$ if and only if 
there exists a shortest directed path in $\QB$ 
from $v$ to $w_{2}$ passing through $w_{1}$; 
or equivalently, if and only if 
the concatenation of a shortest directed path 
from $v$ to $w_{1}$ and one from $w_{1}$ to $w_{2}$ 
is one from $v$ to $w_{2}$. 
\end{dfn}
%
%
\begin{prop}[{\cite[Theorem~7.1]{LNSSS}}] \label{prop:tbmin}
Let $\J$ be a subset of $I$, and let $v \in W$. 
Then each coset $u\WJ$ for $u \in W$ has a unique minimal element
with respect to $\tb{v}${\rm;} we denote it by $\tbmin{u}{\J}{v}$. 
\end{prop}
%
%
\begin{rem} \label{rem:tbmin}
Let $\J$ be a subset of $I$, and let $u,\,v \in W$. 
It is obvious by the definition of the $v$-tilted Bruhat order 
that if $u\WJ=v\WJ$, then $\tbmin{u}{\J}{v} = v$. 
\end{rem}
%
%
\begin{prop} \label{prop:mins}
Let $\J$ be a subset of $I$. 
Let $y \in W_{\af}$ and $x \in \WJa$ be such that $x \sige \PJ(y)$, and 
write these as\,{\rm:} 
\begin{equation*}
\begin{cases}
y = v_{y}t_{\xi_y} & \text{\rm with $v_y \in W$ and $\xi_y \in Q^{\vee}$;} \\[1mm]
x = v_{x}\PJ(t_{\xi_{x}}) 
  & \text{\rm with $v_{x} \in W^{\J}$ and $\xi_{x} \in Q^{\vee}$}, 
\end{cases}
\end{equation*}
respectively. Also, write $\min \Lige{y}{x} \in W_{\af}$ as\,{\rm:}
\begin{equation*}
\min \Lige{y}{x} = w t_{\gamma} \quad 
 \text{\rm with $w \in W$ and $\gamma \in Q^{\vee}$}. 
\end{equation*}
Then, $w = \tbmin{v_x}{\J}{v_y}$ and 
$\gamma = [\xi_x]^{\J} + [\xi_y+\wt (v_y \Rightarrow w)]_{\J}$.
\end{prop}

\begin{proof}
Since $wt_{\gamma} \in \Lift(x)$, 
we see that $w \in v_{x}\WJ$ and $\xi_{x} - \gamma \in Q_{\J}^{\vee}$. 
If we set $w':=\tbmin{v_x}{\J}{v_y}$, then we have $w' \tb{v_y} w$. 
First we show that $w'=w$.
Suppose, for a contradiction, that $w' \ne w$. Then there exists 
a shortest directed path in $\QB$ from $v_y$ to $w$ that passes through $w'$. 
Hence we have
$\wt ( v_y \Rightarrow w ) = 
\wt ( v_y \Rightarrow w' ) + 
\wt ( w' \Rightarrow w )$. 
Since $w$ and $w'$ are contained in the same coset $v_x\WJ$, 
we deduce by Lemma~\ref{lem:wtS} that 
\begin{equation*}
0 = \wt^{\J}(v_x \Rightarrow v_x)
  = \wt^{\J}(\mcr{w'} \Rightarrow \mcr{w})
  = [\wt (w' \Rightarrow w)]^{\J},
\end{equation*}
which implies that 
$\wt ( w' \Rightarrow w ) \in Q_{\J}^{\vee,+}$. 
Hence we obtain $w't_{\gamma-\wt ( w' \Rightarrow w )} \in \Lift(x)$. 
Let us show that $w't_{\gamma-\wt ( w' \Rightarrow w )} \sige y$.
Since $wt_{\gamma}= \min \Lige{y}{x} \sige y = v_{y}t_{\xi_{y}}$, 
it follows from Lemma~\ref{lem:siwt} that 
$\gamma \ge \wt (v_{y} \Rightarrow w)+\xi_{y}$. 
Hence we deduce that 
\begin{equation*}
\gamma - \wt ( w' \Rightarrow w ) \ge 
\wt (v_{y} \Rightarrow w)  + \xi_{y} - \wt ( w' \Rightarrow w ) = 
\wt (v_{y} \Rightarrow w') + \xi_{y}.
\end{equation*}
We see by Lemma~\ref{lem:siwt} that 
$w't_{\gamma - \wt ( w' \Rightarrow w )} \sige v_y t_{\xi_y} = y$. 
Therefore, we conclude that $w't_{\gamma - \wt ( w' \Rightarrow w )} \in \Lige{y}{x}$. 
However, by Lemma~\ref{lem:siwt}, we have  
$\min \Lige{y}{x} = wt_{\gamma} \sige w't_{\gamma - \wt ( w' \Rightarrow w )}$,
which is a contradiction; note that $w \ne w'$ by our assumption.  
Thus we obtain $w = \tbmin{v_x}{\J}{v_y}$, as desired. 

Next, we set $\gamma':=[\xi_x]^{\J} + [\xi_{y}+\wt (v_{y} \Rightarrow w)]_{\J}$, 
and show that $\gamma = \gamma'$. 
Since $\xi_{x} - \gamma \in Q_{\J}^{\vee}$ as seen above, 
$\gamma=[\xi_x]^{\J} + \zeta$ for some $\zeta \in Q_{\J}$. 
Also, since $\min \Lige{y}{x}=wt_{\gamma} \sige v_{y}t_{\xi_{y}}=y$, 
we see by Lemma~\ref{lem:siwt} that 
$[\xi_x]^{\J} + \zeta = \gamma \ge \xi_{y}+\wt (v_{y} \Rightarrow w)$. 
It follows that 
$\zeta=[\gamma']_{\J} \ge [\xi_{y}+\wt (v_{y} \Rightarrow w)]_{\J}$, 
and hence $\gamma = [\xi_x]^{\J} + \zeta \ge 
[\xi_x]^{\J} + [\xi_{y}+\wt (v_{y} \Rightarrow w)]_{\J}=\gamma'$. 
Let us show the opposite inequality $\gamma' \ge \gamma$. 
Since $\xi_{x} - \gamma' \in Q_{\J}^{\vee}$ 
and $w \in v_{x}\WJ$, it follows that $wt_{\gamma'} \in \Lift(x)$. 
Because $v_{x}\PJ(t_{\xi_x}) = x \sige \PJ(y) = \mcr{v_{y}}^{\J}\PJ(t_{\xi_y})$ 
by the assumption, it follows from Lemmas~\ref{lem:wtS} and \ref{lem:siwt} 
that $[\xi_{x}]^{\J} \ge [\xi_{y}+\wt(v_{y} \Rightarrow w)]^{\J}$. 
Therefore, we deduce that 
$\gamma'=[\xi_x]^{\J} + [\xi_{y}+\wt (v_{y} \Rightarrow w)]_{\J} \ge 
\xi_{y}+\wt (v_{y} \Rightarrow w)$, which implies that 
$wt_{\gamma'} \sige v_{y}t_{\xi_y}=y$ by Lemma~\ref{lem:siwt}. 
Hence we obtain $wt_{\gamma'} \in \Lige{y}{x}$. 
Since $w t_{\gamma} = \min \Lige{y}{x}$, 
it follows from Lemma~\ref{lem:SiB}\,(2) that $\gamma' \ge \gamma$. 
Thus we conclude that $\gamma=\gamma$, as desired. 
This completes the proof of the proposition. 
\end{proof}
%
%
\begin{cor} \label{cor:mins}
Let $y \in W_{\af}$ and $\xi \in Q^{\vee}$ be such that 
$\mcr{\lng}^{\J}\PJ(t_{\xi}) \sige \PJ(y)$. 
We write $y$ as $y = v_{y}t_{\xi_{y}}$ 
with $v_{y} \in W$ and $\xi_{y} \in Q^{\vee}$. 
Then, 
%
%
\begin{equation} \label{eq:mins}
\min \Lige{y}{\mcr{\lng}^{\J}\PJ(t_{\xi})} = \tbmin{\lng}{\J}{v_{y}} \cdot 
t_{[\xi]^{\J}+[\xi_{y}]_{\J}}. 
\end{equation}
\end{cor}

\begin{proof}
We set $w:=\tbmin{\lng}{\J}{v_{y}}$. By Proposition~\ref{prop:mins}, 
we see that 
\begin{equation*}
\min \Lige{y}{\mcr{\lng}^{\J}\PJ(t_{\xi})} = w
t_{[\xi]^{\J}+[\xi_{y}+\wt(v_{y} \Rightarrow w)]_{\J}}. 
\end{equation*}
Since $\lng \in \lng \WJ$ and $w=\tbmin{\lng}{\J}{v_{y}}$, 
we have $w \tb{v_{y}} \lng$. Therefore, there exists a shortest directed path 
in $\QB$ from $v_{y}$ to $\lng$ passing through $w$, which implies that 
\begin{equation*}
\wt (v_{y} \Rightarrow \lng) = \wt(v_{y} \Rightarrow w) + \wt (w \Rightarrow \lng).
\end{equation*}
Because $\lng$ is greater than or equal to $v_{y}$ in the ordinary Bruhat order on $W$, 
there exists a shortest directed path from $v_{y}$ to $\lng$ in $\QB$ 
whose directed edges are all Bruhat edges. Hence it follows that 
$\wt (v_{y} \Rightarrow \lng) = 0$. 
Similarly, we have $\wt (w \Rightarrow \lng) = 0$. 
Thus we obtain $\wt(v_{y} \Rightarrow w) = 0$. 
This proves the corollary.
\end{proof}
%
%
\begin{dfn}[dual tilted Bruhat order] \label{dfn:dtilted}
For each $v \in W$, we define the dual $v$-tilted Bruhat order 
$\dtb{v}$ on $W$ as follows:
for $w_{1},w_{2} \in W$, 
%
%
\begin{equation} \label{eq:dtilted}
w_{1} \dtb{v} w_{2} \iff \ell(w_{1} \Rightarrow v) = 
 \ell(w_{1} \Rightarrow w_{2}) + \ell(w_{2} \Rightarrow v).
\end{equation}
Namely, $w_{1} \dtb{v} w_{2}$ if and only if 
there exists a shortest directed path in $\QB$ 
from $w_{1}$ to $v$ passing through $w_{2}$; 
or equivalently, if and only if the concatenation of a shortest directed path 
from $w_{1}$ to $w_{2}$ and one from $w_{2}$ to $v$ 
is one from $w_{1}$ to $v$. 
\end{dfn}

We see by \cite[Lemma~2.1.3]{NNS1} that 
$w_{1} \dtb{v} w_{2}$ if and only if 
$w_{2}\lng \tb{v\lng} w_{1}\lng$ for $w_{1},\,w_{2} \in W$ and $v \in W$.
The next proposition follows from Proposition~\ref{prop:tbmin} and this observation.
%
%
\begin{prop} \label{prop:tbmax}
Let $v \in W$, and let $\J$ be a subset of $I$. 
Then each coset $u\WJ$ for $u \in W$ has a unique maximal element
with respect to $\dtb{v}$\,{\rm;} 
we denote it by $\tbmax{u}{\J}{v}$. Then we have
\begin{equation}
\tbmax{u}{\J}{v} = \bigl(\tbmin{u\lng}{\J^{\ast}}{v\lng}\bigr) \lng.
\end{equation} 
\end{prop}

We can prove the following proposition by using 
Proposition~\ref{prop:mins}, together with 
Corollary~\ref{cor:deo}\,(1) and \eqref{eq:lng1}. 
%
%
\begin{prop} \label{prop:maxs}
Let $\J$ be a subset of $I$. 
Let $y \in W_{\af}$ and $x \in \WJa$ be such that $\PJ(y) \sige x$, and 
write these as\,{\rm:} 
\begin{equation*}
\begin{cases}
y = v_{y}t_{\xi_{y}} & 
 \text{\rm with $v_{y} \in W$ and $\xi_{y} \in Q^{\vee}$;} \\[1mm]
x = v_{x}\PJ(t_{\xi_{x}}) 
  & \text{\rm with $v_{x} \in \WJu$ and $\xi_{x} \in Q^{\vee}$}, 
\end{cases}
\end{equation*}
respectively. We write $\max \Lile{y}{x} \in W_{\af}$ as\,{\rm:}
\begin{equation*}
\max \Lile{y}{x} = w t_{\gamma} \quad 
 \text{\rm with $w \in W$ and $\gamma \in Q^{\vee}$}. 
\end{equation*}
Then, $w = \tbmax{v_{x}}{\J}{v_{y}}$ and 
$\gamma = [\xi_{x}]^{\J} + [\xi_{y} - \wt (w \Rightarrow v_{y})]_{\J}$. 
\end{prop}
%
%
%
\section{Main result.}
\label{sec:main}
%
%
\subsection{Semi-infinite Lakshmibai-Seshadri paths.}
\label{subsec:SLS}

We fix $\lambda \in P^{+} \subset P_{\af}^{0}$ 
(see \eqref{eq:P-fin} and \eqref{eq:P}), and set 
%
%
\begin{equation} \label{eq:J}
\J=\Jl:= 
\bigl\{ i \in I \mid \pair{\lambda}{\alpha_{i}^{\vee}}=0 \bigr\} \subset I.
\end{equation}
%
%
\begin{dfn} \label{dfn:SBa}
For a rational number $0 < a < 1$, 
we define $\SBa$ to be the subgraph of $\SBJ$ 
with the same vertex set but having only 
those edges of the form $x \edge{\beta} y$ for which 
$a \pair{x\lambda}{\beta^{\vee}} \in \BZ$ holds.
\end{dfn}
%
%
\begin{dfn}\label{dfn:SLS}
A semi-infinite Lakshmibai-Seshadri (LS for short) path of 
shape $\lambda $ is a pair 
%
%
\begin{equation} \label{eq:SLS}
\pi = (\bx \,;\, \ba) 
     = (x_{1},\,\dots,\,x_{s} \,;\, a_{0},\,a_{1},\,\dots,\,a_{s}), \quad s \ge 1, 
\end{equation}
of a strictly decreasing sequence $\bx : x_1 \sig \cdots \sig x_s$ 
of elements in $(\WJu)_{\af}$ and an increasing sequence 
$\ba : 0 = a_0 < a_1 < \cdots  < a_s =1$ of rational numbers 
satisfying the condition that there exists a directed path 
from $x_{u+1}$ to  $x_{u}$ in $\SBb{a_{u}}$ 
for each $u = 1,\,2,\,\dots,\,s-1$. 
We denote by $\SLS(\lambda)$ 
the set of all semi-infinite LS paths of shape $\lambda$.
\end{dfn}

Following \cite[Sect.~3.1]{INS} (see also \cite[Sect.~2.4]{NS16}), 
we endow the set $\SLS(\lambda)$ 
with a (regular) crystal structure with weights in $P_{\af}$ by 
the map $\wt:\SLS(\lambda) \rightarrow P_{\af}$ 
and the root operators $e_{i}$, $f_{i}$, $i \in I_{\af}$, 
defined as follows. 
Let $\pi \in \SLS(\lambda)$ be of the form \eqref{eq:SLS}. 
Define $\ol{\pi}:[0,1] \rightarrow \BR \otimes_{\BZ} P_{\af}$ 
to be the piecewise-linear, continuous map 
whose ``direction vector'' on the interval 
$[a_{u-1},\,a_{u}]$ is $x_{u}\lambda \in P_{\af}$ 
for each $1 \le u \le s$, that is, 
%
%
\begin{equation} \label{eq:olpi}
\ol{\pi} (t) := 
\sum_{k = 1}^{u-1}(a_{k} - a_{k-1}) x_{k}\lambda + (t - a_{u-1}) x_{u}\lambda
\quad
\text{for $t \in [a_{u-1},\,a_u]$, $1 \le u \le s$}; 
\end{equation}
we know from \cite[Proposition~3.1.3]{INS} that $\ol{\pi}$ is 
an (ordinary) LS path of shape $\lambda$, 
introduced in \cite[Sect.~4]{Lit95}. We set
%
%
\begin{equation} \label{eq:wt}
\wt (\pi):= \ol{\pi}(1) = \sum_{u = 1}^{s} (a_{u}-a_{u-1})x_{u}\lambda \in P_{\af}.
\end{equation}
Now, we set 
%
%
\begin{equation} \label{eq:H}
\begin{cases}
H^{\pi}_{i}(t) := \pair{\ol{\pi}(t)}{\alpha_{i}^{\vee}} \quad 
\text{for $t \in [0,1]$}, \\[1.5mm]
m^{\pi}_{i} := 
 \min \bigl\{ H^{\pi}_{i} (t) \mid t \in [0,1] \bigr\}. 
\end{cases}
\end{equation}
As explained in \cite[Remark~2.4.3]{NS16}, 
all local minima of the function $H^{\pi}_{i}(t)$, $t \in [0,1]$, 
are integers; in particular, 
the minimum value $m^{\pi}_{i}$ is a nonpositive integer 
(recall that $\ol{\pi}(0)=0$, and hence $H^{\pi}_{i}(0)=0$).

We define $e_{i}\pi$ as follows. 
If $m^{\pi}_{i}=0$, then we set $e_{i} \pi := \bzero$, 
where $\bzero$ is an additional element not 
contained in any crystal. 
If $m^{\pi}_{i} \le -1$, then we set
%
%
\begin{equation} \label{eq:t-e}
\begin{cases}
t_{1} := 
  \min \bigl\{ t \in [0,\,1] \mid 
    H^{\pi}_{i}(t) = m^{\pi}_{i} \bigr\}, \\[1.5mm]
t_{0} := 
  \max \bigl\{ t \in [0,\,t_{1}] \mid 
    H^{\pi}_{i}(t) = m^{\pi}_{i} + 1 \bigr\}; 
\end{cases}
\end{equation}
notice that $H^{\pi}_{i}(t)$ is 
strictly decreasing on the interval $[t_{0},\,t_{1}]$. 
Let $1 \le p \le q \le s$ be such that 
$a_{p-1} \le t_{0} < a_p$ and $t_{1} = a_{q}$. 
Then we define $e_{i}\pi$ by
%
%
\begin{equation} \label{eq:epi}
\begin{split}
& e_{i} \pi := ( 
  x_{1},\,\ldots,\,x_{p},\,s_{i}x_{p},\,s_{i}x_{p+1},\,\ldots,\,
  s_{i}x_{q},\,x_{q+1},\,\ldots,\,x_{s} ; \\
& \hspace*{40mm}
  a_{0},\,\ldots,\,a_{p-1},\,t_{0},\,a_{p},\,\ldots,\,a_{q}=t_{1},\,
\ldots,\,a_{s});
\end{split}
\end{equation}
if $t_{0} = a_{p-1}$, then we drop $x_{p}$ and $a_{p-1}$, and 
if $s_{i} x_{q} = x_{q+1}$, then we drop $x_{q+1}$ and $a_{q}=t_{1}$.

Similarly, we define $f_{i}\pi$ as follows. 
Note that $H^{\pi}_{i}(1) - m^{\pi}_{i}$ is a nonnegative integer. 
If $H^{\pi}_{i}(1) - m^{\pi}_{i} = 0$, then we set $f_{i} \pi := \bzero$. 
If $H^{\pi}_{i}(1) - m^{\pi}_{i}  \ge 1$, 
then we set
%
%
\begin{equation} \label{eq:t-f}
\begin{cases}
t_{0} := 
 \max \bigl\{ t \in [0,1] \mid H^{\pi}_{i}(t) = m^{\pi}_{i} \bigr\}, \\[1.5mm]
t_{1} := 
 \min \bigl\{ t \in [t_{0},\,1] \mid H^{\pi}_{i}(t) = m^{\pi}_{i} + 1 \bigr\};
\end{cases}
\end{equation}
notice that $H^{\pi}_{i}(t)$ is 
strictly increasing on the interval $[t_{0},\,t_{1}]$. 
Let $0 \le p \le q \le s-1$ be such that $t_{0} = a_{p}$ and 
$a_{q} < t_{1} \le a_{q+1}$. Then we define $f_{i}\pi$ by
%
%
\begin{equation} \label{eq:fpi}
\begin{split}
& f_{i} \pi := ( x_{1},\,\ldots,\,x_{p},\,s_{i}x_{p+1},\,\dots,\,
  s_{i} x_{q},\,s_{i} x_{q+1},\,x_{q+1},\,\ldots,\,x_{s} ; \\
& \hspace{40mm} 
  a_{0},\,\ldots,\,a_{p}=t_{0},\,\ldots,\,a_{q},\,t_{1},\,
  a_{q+1},\,\ldots,\,a_{s});
\end{split}
\end{equation}
if $t_{1} = a_{q+1}$, then we drop $x_{q+1}$ and $a_{q+1}$, and 
if $x_{p} = s_{i} x_{p+1}$, then we drop $x_{p}$ and $a_{p}=t_{0}$.
In addition, we set $e_{i} \bzero = f_{i} \bzero := \bzero$ 
for all $i \in I_{\af}$.
%
%
\begin{thm}[{see \cite[Theorem~3.1.5]{INS}}] \label{thm:SLS}
\mbox{}
\begin{enu}
\item The set $\SLS(\lambda) \sqcup \{ \bzero \}$ is 
stable under the action of the root operators 
$e_{i}$ and $f_{i}$, $i \in I_{\af}$.

\item For each $\pi \in \SLS(\lambda)$ 
and $i \in I_{\af}$, we set 
\begin{equation*}
\begin{cases}
\ve_{i} (\pi) := 
 \max \bigl\{ n \ge 0 \mid e_{i}^{n} \pi \neq \bzero \bigr\}, \\[1.5mm]
\vp_{i} (\pi) := 
 \max \bigl\{ n \ge 0 \mid f_{i}^{n} \pi \neq \bzero \bigr\}.
\end{cases}
\end{equation*}
Then, the set $\SLS(\lambda)$, 
equipped with the maps $\wt$, $e_{i}$, $f_{i}$, $i \in I_{\af}$, 
and $\ve_{i}$, $\vp_{i}$, $i \in I_{\af}$, 
defined above, is a crystal with weights in $P_{\af}$.
\end{enu}
\end{thm}

We denote by $\SLS_{0}(\lambda)$ the connected component of 
$\SLS(\lambda)$ containing $\pi_{\lambda}:=(e\,;\,0,1) \in \SLS(\lambda)$; 
for the description of the connected components of $\SLS(\lambda)$, 
see \S\ref{subsec:conn} below. 

If $\pi \in \SLS(\lambda)$ is of the form \eqref{eq:SLS}, 
then we deduce from Lemma~\ref{lem:lng1} that 
\begin{equation} \label{eq:dual}
\pi^{\ast} := (\PS{\Ja}(x_{s}\lng),\dots,\PS{\Ja}(x_{1}\lng);
1-a_{s},\dots,1-a_{1},1-a_{0})
\end{equation}
is an element of $\SLS(\lambda^{\ast})=\SLS(-\lng\lambda)$; 
we call $\pi^{\ast}$ the dual path of $\pi$ (cf. \cite[\S2]{Lit95}). 
Notice that $(\pi^{\ast})^{\ast}= \pi$, and 
\begin{equation} \label{eq:dual2}
\begin{cases}
\wt(\pi^{\ast}) = - \wt (\pi) 
   & \text{for $\pi \in \SLS(\lambda)$}, \\[1mm]
(e_{i}\pi)^{\ast} = f_{i}\pi^{\ast}, \quad 
(f_{i}\pi)^{\ast} = e_{i}\pi^{\ast} 
   & \text{for $\pi \in \SLS(\lambda)$ and $i \in I_{\af}$},
\end{cases}
\end{equation}
where $\bzero^{\ast}$ is understood to be $\bzero$. 

If $\pi \in \SLS(\lambda)$ is of the form \eqref{eq:SLS}, 
then we set 
%
%
\begin{equation} \label{eq:dir}
\iota(\pi):=x_{1} \in (\WJu)_{\af}, \qquad 
\kappa(\pi):=x_{s} \in (\WJu)_{\af};
\end{equation}
we call $\iota(\pi)$ and $\kappa(\pi)$ 
the initial direction and the final direction of $\pi$, respectively. 
For $x \in W_{\af}$, we set
%
%
\begin{align}
& \SLS_{\sige x}(\lambda) := 
  \bigl\{
   \pi \in \SLS(\lambda) \mid \kappa(\pi) \sige \PJ(x)
  \bigr\}, \label{eq:SLSge}  \\
& \SLS_{\sile x}(\lambda) := 
  \bigl\{
   \pi \in \SLS(\lambda) \mid \PJ(x) \sige \iota(\pi)
  \bigr\}. \label{eq:SLSle}
\end{align}
%
%
\begin{rem}[{\cite[Lemma~5.3.1 and Proposition~5.3.2\,(2)]{NS16}}] \label{rem:stable}
The set $\SLS_{\sige x}(\lambda) \cup \{\bzero\}$ 
is stable under the action of $f_{i}$ for all $i \in I_{\af}$. 
Moreover, for $i \in I_{\af}$ such that $\pair{x\lambda}{\alpha_{i}^{\vee}} \ge 0$, 
the set $\SLS_{\sige x}(\lambda) \cup \{\bzero\}$ is stable under the action of $e_{i}$. 
\end{rem}

If $\pi = (x_{1},\dots,\,x_{s};a_{0},\dots,a_{s}) \in \SLS(\lambda)$ and $y \in W_{\af}$ 
satisfy $\PJ(y) \sige x_{1}$, that is, if $\pi \in \SLS_{\sile y}(\lambda)$, 
then we define $\kap{\pi}{y} \in W_{\af}$ by 
the following recursive formula:
%
%
\begin{equation} \label{eq:hax}
\begin{cases}
\ha{x}_{0}:=y, & \\[2mm]
\ha{x}_{u}:=\max \Lile{\ha{x}_{u-1}}{x_{u}} & \text{for $1 \le u \le s$}, \\[2mm]
\kap{\pi}{y}:=\ha{x}_{s};
\end{cases}
\end{equation}
we call $\kap{\pi}{y} \in W_{\af}$ the final direction of $\pi$ with respect to $y$.
Similarly, if $\pi = (x_{1},\dots,\,x_{s};a_{0},\dots,a_{s}) \in \SLS(\lambda)$ and $y \in W_{\af}$ 
satisfy $x_{s} \sige \PJ(y)$, that is, if $\pi \in \SLS_{\sige y}(\lambda)$, 
then we define $\io{\pi}{y} \in W_{\af}$ by 
the following recursive formula (from $u=s+1$ to $u=1$): 
%
%
\begin{equation} \label{eq:tix}
\begin{cases}
\ti{x}_{s+1}:=y, & \\[2mm]
\ti{x}_{u}:=\min \Lige{\ti{x}_{u+1}}{x_{u}} & \text{for $1 \le u \le s$}, \\[2mm]
\io{\pi}{y}:=\ti{x}_{1};
\end{cases}
\end{equation}
we call $\io{\pi}{y} \in W_{\af}$ the initial direction of $\pi$ with respect to $y$. 
The next lemma follows from Lemma~\ref{lem:lng1} and Corollary~\ref{cor:deo}. 
%
%
\begin{lem} \label{lem:io-kap}
Let $y \in W_{\af}$. 
\begin{enu}
\item If $\pi \in \SLS_{\sige y}(\lambda)$, 
then $\pi^{\ast} \in \SLS_{\sile y\lng}(-\lng\lambda)$, and 
$\io{\pi}{y} = (\kap{\pi^{\ast}}{y\lng})\lng$.

\item If $\pi \in \SLS_{\sile y}(\lambda)$, 
then $\pi^{\ast} \in \SLS_{\sige y\lng}(-\lng\lambda)$, and 
$\kap{\pi}{y} = (\io{\pi^{\ast}}{y\lng})\lng$.
\end{enu}
\end{lem}
%
%
\subsection{Extremal weight modules and their Demazure submodules.}
\label{subsec:extremal}

We take an arbitrary $\lambda \in P^{+} \subset P_{\af}^{0}$. 
Let $V(\lambda)$ denote the (level-zero) extremal weight module of 
extremal weight $\lambda$ over $U_{\q}(\Fg_{\af})$, 
which is defined to be 
the integrable $U_{\q}(\Fg_{\af})$-module generated by 
a single element $v_{\lambda}$ with 
the defining relation that ``$v_{\lambda}$ is 
an extremal weight vector of weight $\lambda$''. 
Here, recall from \cite[Sect.~3.1]{Kas02} and \cite[Sect.~2.6]{Kas05} that 
$v_{\lambda}$ is an extremal weight vector of weight $\lambda$ 
if and only if ($v_{\lambda}$ is a weight vector of weight $\lambda$ and) 
there exists a family $\{ v_{x} \}_{x \in W_{\af}}$ 
of weight vectors in $V(\lambda)$ such that $v_{e}=v_{\lambda}$, 
and such that for each $i \in I_{\af}$ and $x \in W_{\af}$ with 
$n:=\pair{x\lambda}{\alpha_{i}^{\vee}} \ge 0$ (resp., $\le 0$),
the equalities $E_{i}v_{x}=0$ and $F_{i}^{(n)}v_{x}=v_{s_{i}x}$ 
(resp., $F_{i}v_{x}=0$ and $E_{i}^{(-n)}v_{x}=v_{s_{i}x}$) hold, 
where for $i \in I_{\af}$ and $k \in \BZ_{\ge 0}$, 
the $E_{i}^{(k)}$ and $F_{i}^{(k)}$ are the $k$-th divided powers of 
the Chevalley generators $E_{i}$ and $F_{i}$ of $U_{\q}(\Fg_{\af})$, respectively;
note that the weight of $v_{x}$ is $x\lambda$. 
Also, for each $x \in W_{\af}$, we define 
the Demazure submodule $V_{x}^{-}(\lambda)$ of $V(\lambda)$ by
%
%
\begin{equation} \label{eq:dem}
V_{x}^{-}(\lambda):=U_{\q}^{-}(\Fg_{\af})v_{x}. 
\end{equation}

We know from \cite[Proposition~8.2.2]{Kas94} that $V(\lambda)$ has 
a crystal basis $\CB(\lambda)$ and the corresponding 
global basis $\bigl\{G(b) \mid b \in \CB(\lambda)\bigr\}$.
Also, we know from \cite[Sect.~2.8]{Kas05} (see also \cite[Sect.~4.1]{NS16}) that 
$V_{x}^{-}(\lambda) \subset V(\lambda)$ is compatible 
with the global basis of $V(\lambda)$, that is, 
there exists a subset $\CB_{x}^{-}(\lambda)$ of 
the crystal basis $\CB(\lambda)$ such that 
%
%
\begin{equation} \label{eq:deme}
V_{x}^{-}(\lambda) = 
\bigoplus_{b \in \CB_{x}^{-}(\lambda)} \BC(\q) G(b) 
\subset
V(\lambda) = 
\bigoplus_{b \in \CB(\lambda)} \BC(\q) G(b).
\end{equation}

\begin{rem}[{\cite[Lemma~4.1.2]{NS16}}] \label{rem:412}
For every $x \in W_{\af}$, we have 
$V_{x}^{-}(\lambda) = V_{\PJ(x)}^{-}(\lambda)$ and 
$\CB_{x}^{-}(\lambda) = \CB_{\PJ(x)}^{-}(\lambda)$.
\end{rem}

Denote by $u_{\lambda}$ the element of $\CB(\lambda)$ 
such that $G(u_{\lambda})=v_{\lambda}$; 
recall that $\pi_{\lambda}=(e\,;\,0,1) \in \SLS(\lambda)$. 
We know the following from 
\cite[Theorem~3.2.1]{INS} and \cite[Theorem~4.2.1]{NS16}. 
%
%
\begin{thm} \label{thm:isom}
There exists an isomorphism of crystals 
from $\CB(\lambda)$ to $\SLS(\lambda)$, 
which sends $u_{\lambda}$ to $\pi_{\lambda}$, and 
sends $\CB_{x}^{-}(\lambda)$ to 
$\SLS_{\sige x}(\lambda)$ 
for all $x \in W_{\af}$. 
\end{thm}

Recall that $P_{\af}^{0}=P \oplus \BZ\delta$; 
for $\nu \in P_{\af}^{0}$, we write $\nu = \fin(\nu) + \nul(\nu) \delta$ 
with $\fin(\nu) \in P$ and $\nul(\nu) \in \BZ$. 
Let $\be^{\nu}$ denote the formal exponential for $\nu \in P_{\af}^{0}$, 
and set $q:=\be^{\delta}$; note that 
the formal exponential $\be^{\nu}$ for $\nu \in P_{\af}^{0}$ 
is identical to $q^{\nul(\nu)}\be^{\fin(\nu)}$.
Following \cite[Sect.~2.4]{KNS}, 
we define the graded character $\gch V_{x}^{-}(\lambda)$ of 
$V_{x}^{-}(\lambda)$ by
\begin{equation} \label{eq:gch}
\gch V_{x}^{-}(\lambda) : = 
 \sum_{k \in \BZ}
\Biggl(
 \sum_{\gamma \in Q} 
 \dim \bigl( V_{x}^{-}(\lambda)_{\lambda+\gamma+k\delta} \bigr) 
 \be^{\lambda+\gamma}
\Biggr) q^{k}; 
\end{equation}
if $x = wt_{\xi}$ for $w \in W$ and $\xi \in Q^{\vee}$, 
then $\gch V_{x}^{-}(\lambda)$ is an element of 
$(\BZ[P])\bra{q^{-1}}q^{-\pair{\lambda}{\xi}}$ 
(see \cite[(2.22)]{KNS}). It follows from Theorem~\ref{thm:isom} that 
%
\begin{equation} \label{eq:gch1}
\gch V_{x}^{-}(\lambda) = 
\sum_{\pi \in \SLS_{\sige x}(\lambda)} \be^{\wt(\pi)} =
\sum_{\pi \in \SLS_{\sige x}(\lambda)} q^{\nul(\wt(\pi))} \be^{\fin(\wt(\pi))}. 
\end{equation}
%
%
\begin{prop}[{\cite[Proposition~D.1]{KNS}}] \label{prop:gch-tx}
For each $x \in W_{\af}$ and $\xi \in Q^{\vee}$, 
the equality $\gch V_{xt_{\xi}}^{-}(\lambda) = 
q^{-\pair{\lambda}{\xi}}\gch V_{x}^{-}(\lambda)$ holds. 
\end{prop}
%
%
\subsection{Statement of the main result in terms of semi-infinite LS paths.}
\label{subsec:main}

Let $\lambda,\,\mu \in P^{+}$ be such that 
$\lambda-\mu \in P^{+}$, and 
define $\Jl,\,\Jm,\,\J_{\lambda-\mu} \subset I$ 
as in \eqref{eq:J}. Write $\lambda$ and $\mu$ as:
$\lambda = \sum_{i \in I} \lambda_{i} \vpi_{i}$ and 
$\mu = \sum_{i \in I} \mu_{i} \vpi_{i}$, respectively, 
where $\lambda_{i},\mu_{i} \in \BZ_{\ge 0}$, 
with $\lambda_{i}-\mu_{i} \ge 0$ for all $i \in I$. 
The following theorem is one of the main results of this paper. 
%
%
\begin{thm} \label{thm:main}
For $x \in W_{\af}$, the following equality holds\,{\rm:}
%
%
\begin{equation} \label{eq:main}
\begin{split}
& \prod_{i \in I} \prod_{k = \lambda_{i}-\mu_{i}+1}^{\lambda_{i}} \frac{1}{1-q^{-k}}
  \gch V_{x}^{-}(\lambda-\mu) \\[3mm]
& \hspace*{5mm} = 
 \sum_{ \begin{subarray}{c}
   y \in W_{\af} \\[1mm]
   y \sige x \end{subarray}}\,
 \sum_{ \begin{subarray}{c}
  \pi \in \SLS(\mu) \\[1mm] 
  \iota(\pi) \sile \PS{\Jm}(y),\, \kap{\pi}{y}=x 
  \end{subarray}}  (-1)^{\sell(y)-\sell(x)}
  \underbrace{q^{-\nul(\wt(\pi))}
  \be^{-\fin(\wt(\pi))}}_{=\be^{-\wt(\pi)}} \gch V_{y}^{-}(\lambda).
\end{split}
\end{equation}
\end{thm}

The outline of our proof of this theorem is as follows. 
First, we prove \eqref{eq:main} in the case that 
$\mu$ is a fundamental weight and $x$ is a translation element 
in $W_{\af}$ (see \S\ref{subsec:s1}); we show 
a key formula in this case in Proposition~\ref{prop:s1} 
on the basis of standard monomial theory for semi-infinite LS paths, 
established in \cite{KNS} (see \S\ref{sec:SMT}). 
Next, by making use of Demazure operators (see \S\ref{subsec:rec}), 
we show \eqref{eq:main} in the case that 
$\mu$ is a fundamental weight and $x$ is an arbitrary element of $W$ 
(see \S\ref{subsec:s2}).
Finally, using Proposition~\ref{prop:smt}, 
we prove \eqref{eq:main} in the general case (see \S\ref{subsec:s3}). 
%
%
\subsection{Quantum Lakshmibai-Seshadri paths.}
\label{subsec:QLS}
In this subsection, we fix $\mu \in P^{+}$, 
and take $\J=\J_{\mu}$ as in \eqref{eq:J}. 
%
%
\begin{dfn} \label{dfn:QBa}
For a rational number $0 < a < 1$, 
we define $\QBa$ to be the subgraph of $\QBJ$ 
with the same vertex set but having only those directed edges of 
the form $w \edge{\beta} v$ for which 
$a\pair{\mu}{\beta^{\vee}} \in \BZ$ holds.
\end{dfn}
%
%
\begin{dfn}\label{dfn:QLS}
A quantum LS path of shape $\mu$ is a pair 
%
%
\begin{equation} \label{eq:QLS}
\eta = (\bw \,;\, \ba) = 
(w_{1},\,\dots,\,w_{s} \,;\, a_{0},\,a_{1},\,\dots,\,a_{s}), \quad s \ge 1, 
\end{equation}
of a sequence $w_{1},\,\dots,\,w_{s}$ 
of elements in $\WJu$, with $w_{u} \ne w_{u+1}$ 
for any $1 \le u \le s-1$, and an increasing sequence 
$0 = a_0 < a_1 < \cdots  < a_s =1$ of rational numbers 
satisfying the condition that there exists a directed path 
in $\QBb{a_{u}}$ from $w_{u+1}$ to  $w_{u}$ 
for each $u = 1,\,2,\,\dots,\,s-1$. 
\end{dfn}

Denote by $\QLS(\mu)$ 
the set of all quantum LS paths of shape $\mu$.
In the same manner as for $\SLS(\mu)$, we can endow 
the set $\QLS(\mu)$ with a crystal structure with weights in 
$P \cong P_{\af}^{0}/\BZ\delta \subset P_{\af}/\BZ\delta$; 
for the details, see \cite[Sect.~4.2]{LNSSS2} and \cite[Sect.~2.5]{NNS3}. 

Define a projection $\cl : \WJa \twoheadrightarrow \WJu$ by
$\cl (x) := w$ for $x \in \WJa$ of the form 
$x = w\PJ(t_{\xi})$ with $w \in \WJu$ and $\xi \in Q^{\vee}$.
For $\pi = (x_{1},\,\dots,\,x_{s}\,;\,
a_{0},\,a_{1},\,\dots,\,a_{s}) \in \SLS(\mu)$, we define 
\begin{equation} \label{eq:clpi}
\cl(\pi):=(\cl(x_{1}),\,\dots,\,\cl(x_{s})\,;\,
  a_{0},\,a_{1},\,\dots,\,a_{s});
\end{equation}
here, for each $1 \le p < q \le s$ such that $\cl(x_{p})= \cdots = \cl(x_{q})$, 
we drop $\cl(x_{p}),\,\dots,\,\cl(x_{q-1})$ and $a_{p},\,\dots,\,a_{q-1}$; 
we set $\cl(\bzero):=\bzero$ by convention. 
We know from \cite[Sect.~6.2]{NS16} that $\cl(\pi) \in \QLS(\mu)$ 
for all $\pi \in \SLS(\mu)$. 
Also, we know the following lemma from \cite[Lemma~6.2.3]{NS16}; 
recall that $\SLS_{0}(\mu)$ denotes the connected component of $\SLS(\mu)$ 
containing $\pi_{\mu}=(e\,;\,0,1)$. 
%
%
\begin{lem} \label{lem:deg}
For each $\eta \in \QLS(\mu)$, there exists a unique 
$\pi_{\eta} \in \SLS_{0}(\mu)$ such that 
$\cl(\pi_{\eta})=\eta$ and $\kappa(\pi_{\eta}) = \kappa(\eta) \in \WJu$. 
\end{lem}

\begin{rem} \label{rem:deg}
For $\eta = (w_{1},\,\dots,\,w_{s} \,;\, 
a_{0},\,a_{1},\,\dots,\,a_{s}) \in \QLS(\mu)$, 
define $\xi_{1},\,\dots,\,\xi_{s-1},\,\xi_{s} \in Q^{\vee}$ 
by the following recursive formula (from $u=s$ to $u=1$): 
%
%
\begin{equation} \label{eq:bxi}
%
%
\xi_{s}=0, \qquad 
\xi_{u}=\xi_{u+1} + \wt(w_{u+1} \Rightarrow w_{u})
\quad \text{for $1 \le u \le s-1$};
%
\end{equation}
for the definition of 
$\wt(w_{u+1} \Rightarrow w_{u})$, see \S\ref{subsec:QBG}. 
Then we know from \cite[Proposition~2.32]{NNS3} that 
%
%
\begin{equation} \label{eq:pieta}
\pi_{\eta} = 
(w_{1}\PJ(t_{\xi_1}),\,\dots,\,w_{s-1}\PJ(t_{\xi_{s-1}}),\,w_{s} \,;\, 
a_{0},\,a_{1},\,\dots,\,a_{s});
\end{equation}
the weight $\wt(\pi_{\eta}) \in P_{\af}^{0} = P \oplus \BZ \delta$ of $\pi_{\eta}$ 
can be written as: $\wt(\pi_{\eta}) = \wt(\eta) + \Deg(\eta)\delta$, where 
$\Deg:\QLS(\mu) \rightarrow \BZ_{\le 0}$ denotes the (tail) degree function 
(see \cite[Corollary~4.8]{LNSSS2}) given by:
%
%
\begin{equation} \label{eq:deg}
\Deg (\eta) = - \sum_{u=1}^{s-1} a_{u} \pair{\mu}{\wt(w_{u+1} \Rightarrow w_{u})}. 
\end{equation}
\end{rem}

For $\eta = (w_{1},\,\dots,\,w_{s} \,;\, 
a_{0},\,a_{1},\,\dots,\,a_{s}) \in \QLS(\mu)$ and $v \in W$, 
define $\kap{\eta}{v} \in W$ 
by the following recursive formula 
(cf. \eqref{eq:hax} and Proposition~\ref{prop:maxs}): 
%
%
\begin{equation} \label{eq:haw}
\begin{cases}
\ha{w}_{0}:=v, & \\[2mm]
\ha{w}_{u}:=\tbmax{w_{u}}{\J}{\ha{w}_{u-1}} & \text{for $1 \le u \le s$}, \\[2mm]
\kap{\eta}{v}:=\ha{w}_{s};
\end{cases}
\end{equation}
we call $\kap{\eta}{v} \in W$ 
the final direction of $\eta$ with respect to $v$. 
We set
%
%
\begin{equation} \label{eq:zeta}
\zeta(\eta,v):=
\wt ( \ha{w}_{1} \Rightarrow v ) + 
\sum_{u=1}^{s-1} \wt (\ha{w}_{u+1} \Rightarrow \ha{w}_{u}).
\end{equation}
%
%
\begin{rem} \label{rem:pieta}
Keep the notation and setting of Remark~\ref{rem:deg}. 
We see from Lemma~\ref{lem:wtS} that 
$[\zeta(\eta,v) - \wt ( \ha{w}_{1} \Rightarrow v )]^{\J} = 
[\xi_{1}]^{\J}$. Hence, by Lemma~\ref{lem:PiJ}\,(2), we have
$\iota(\pi_{\eta}) = w_{1} 
\PJ(t_{\zeta(\eta,v) - \wt ( \ha{w}_{1} \Rightarrow v )})$. 
\end{rem}
%
%
\subsection{Reformulation of the main result in terms of quantum LS paths.}
\label{subsec:cor}
Let $\lambda,\,\mu \in P^{+}$ be such that 
$\lambda-\mu \in P^{+}$, and 
define $\Jl,\,\Jm,\,\J_{\lambda-\mu} \subset I$ as in \eqref{eq:J}. 
The following is a corollary of Theorem~\ref{thm:main}. 
%
%
\begin{cor} \label{cor:main}
For $x \in W$, 
the following equality holds\,{\rm:}
%
%
\begin{equation} \label{eq:mainc}
\begin{split}
& \gch V_{x}^{-}(\lambda-\mu) \\
& \qquad = 
\sum_{v \in W}\,
\sum_{
  \begin{subarray}{c}
  \eta \in \QLS(\mu) \\[1mm]
  \kap{\eta}{v} = x 
  \end{subarray}}
(-1)^{\ell(v)-\ell(x)}
\underbrace{q^{-\Deg(\eta)}\be^{-\wt(\eta)}}_{=\be^{-\wt(\pi_{\eta})}}
\gch V_{vt_{\zeta(\eta,v)}}^{-}(\lambda). 
\end{split}
\end{equation}
\end{cor}

We will give a proof of this corollary in \S\ref{subsec:prf-cor}. 
%
%
\subsection{Case of minuscule weights.}
\label{subsec:minu}

Here we apply the formula \eqref{eq:main} in Theorem~\ref{thm:main}
to the case that $\mu \in P^{+}$ is a minuscule weight, i.e., 
$\pair{\mu}{\alpha^{\vee}} \in \bigl\{0,1\bigr\}$ for all $\alpha \in \Delta^{+}$; 
we know that $\mu$ is just a fundamental weight $\vpi_{r}$ satisfying the 
condition that $\pair{\vpi_{r}}{\theta^{\vee}}=1$. 
We deduce from this condition that 
the subgraph $\SBSa{\vpi_{r}}$ has no edge for any rational number $0 < a < 1$. 
Hence it follows from the definition of semi-infinite LS paths that 
\begin{equation} \label{eq:SLS-min}
\SLS(\vpi_{r}) = \bigl\{ (z \,;\, 0,1) \mid z \in \WSa{\J_{\vpi_{r}}} \bigr\}. 
\end{equation}
Let $x \in W_{\af}$, and let $y \in W_{\af}$ be such that $y \sige x$. 
We see that $\pi = (z \,;\, 0,1) \in \SLS(\vpi_{r})$ satisfies 
$\iota(\pi) \sile \PS{\J_{\vpi_{r}}}(y)$ and $\kap{\pi}{y}=x$ 
(note that $\PS{\J_{\vpi_{r}}}(\kap{\pi}{y}) = \kappa(\pi) = \iota(\pi) = z$ in this case) 
if and only if $\pi = (\PS{\J_{\vpi_{r}}}(x) \,;\, 0,1)$ and 
$\max \Lile{y}{\PS{\J_{\vpi_{r}}}(x)}=x$; we have 
$\wt (\pi) = x \vpi_{r}$ in this case. Therefore, we obtain 
%
%
\begin{equation} \label{eq:main-min}
\begin{split}
& \frac{1}{1-q^{-\lambda_{r}}}
  \gch V_{x}^{-}(\lambda-\vpi_{r}) \\[3mm]
& \hspace*{5mm} = 
 \sum_{ \begin{subarray}{c}
   y \in W_{\af},\, y \sige x \\[1mm] 
   \max \Lile{y}{\PS{\J_{\vpi_{r}}}(x)}=x \end{subarray}}\,
  (-1)^{\sell(y)-\sell(x)}
  q^{-\nul(x\vpi_{r})} \be^{-\fin(x\vpi_{r})} \gch V_{y}^{-}(\lambda).
\end{split}
\end{equation}

Similarly, let us apply the formula \eqref{eq:mainc} in Corollary~\ref{cor:main}
to the case that $\mu = \vpi_{r}$ is a minuscule weight; 
for examples in type $A_{2}$, see Appendix~\ref{sec:example}. 
By the same reasoning as for $\SLS(\vpi_{r})$ above, we see that 
\begin{equation} \label{eq:QLS-min}
\QLS(\vpi_{r}) = \bigl\{ (u \,;\, 0,1) \mid u \in \WSu{\J_{\vpi_{r}}} \bigr\}. 
\end{equation}
Let $x,\,v \in W$. 
We deduce that $\eta = (u \,;\, 0,1) \in \QLS(\vpi_{r})$ satisfies 
$\kap{\eta}{v} = x$ (note that $\mcr{\kap{\eta}{v}}^{\J_{\vpi_{r}}} = 
\kappa(\eta) = \iota(\eta)=u$ in this case) if and only if 
$\eta = (\mcr{x}^{\J_{\vpi_{r}}} \,;\, 0,1)$ and 
$\tbmax{x}{\J_{\vpi_{r}}}{v} = x$; observe that $\wt (\eta) = x \vpi_{r}$, 
$\Deg (\eta) = 0$, and $\zeta(\eta,v) = \wt (x \Rightarrow v)$ in this case. 
Therefore, we obtain 
%
%
\begin{equation} \label{eq:mainc-min}
\gch V_{x}^{-}(\lambda-\vpi_{r}) =
\sum_{
  \begin{subarray}{c}
  v \in W \\[1mm]
  \tbmax{x}{\J_{\vpi_{r}}}{v} = x
  \end{subarray}}
(-1)^{\ell(v)-\ell(x)} \be^{-x\vpi_{r}}
\gch V_{v t_{\wt(x \Rightarrow v)}}^{-}(\lambda). 
\end{equation}
%
%
\section{Standard monomial theory for semi-infinite LS paths.}
\label{sec:SMT}
%
%
\subsection{Connected components of $\SLS(\lambda)$.}
\label{subsec:conn}

Let $\lambda \in P^{+}$, and write it as 
$\lambda = \sum_{i \in I} \lambda_{i} \vpi_{i}$, with $\lambda_{i} \in \BZ_{\ge 0}$; 
note that $\J=\Jl=\bigl\{i \in I \mid \lambda_{i} = 0\bigr\}$ (see \eqref{eq:J}). 
We define $\Par(\lambda)$ to be the set of $I$-tuples of partitions 
$\bchi = (\chi^{(i)})_{i \in I}$ such that $\chi^{(i)}$ is a partition of 
length (strictly) less than $\lambda_{i}$ for each $i \in I$; 
a partition of length less than $0$ is understood 
to be the empty partition $\emptyset$. 
Also, for $\bchi = (\chi^{(i)})_{i \in I} \in \Par(\lambda)$, we set 
$|\bchi|:=\sum_{i \in I} |\chi^{(i)}|$, where for a partition 
$\chi = (\chi_{1} \ge \chi_{2} \ge \cdots \ge \chi_{l} \ge 0)$, 
we set $|\chi| := \chi_{1}+\cdots+\chi_{l}$. 

Here we recall from \cite[Sect.~7]{INS} 
the parametrization of the set $\Conn(\SLS(\lambda))$ of 
connected components of $\SLS(\lambda)$ 
in terms of the set $\Par(\lambda)$. We set
$\Turn(\lambda):=
 \bigl\{k/\lambda_{i} \mid i \in I \setminus \J \text{ and }
 0 \le k \le \lambda_{i}\bigr\}$. 
By \cite[Proposition~7.1.2]{INS}, 
each connected component of $\SLS(\lambda)$ 
contains a unique element of the form: 
%
%
\begin{equation} \label{eq:ext}
\bigl( \PJ(t_{\xi_{1}}),\,\dots,\,\PJ(t_{\xi_{s-1}}),\,e \,;\, 
  a_{0},\,a_{1},\,\dots,\,a_{s-1},\,a_{s} \bigr), 
\end{equation}
where $s \ge 1$, 
$\xi_{1},\,\dots,\,\xi_{s-1}$ are elements of $Q^{\vee}_{I \setminus \J}$
such that $\xi_{1} > \cdots > \xi_{s-1} > 0=:\xi_{s}$ (for the notation, 
see \S\ref{subsec:liealg} and \S\ref{subsec:SiBG}), and 
$a_{u} \in \Turn(\lambda)$ for all $0 \le u \le s$. 
For each element of the form \eqref{eq:ext} 
(or equivalently, each connected component of $\SLS(\lambda)$), 
we define an element $\bchi = (\chi^{(i)})_{i \in I} \in \Par(\lambda)$ as follows.
First, let $i \in I \setminus \J$; note that $\lambda_{i} \ge 1$.  
For each $1 \le k \le \lambda_{i}$, take $0 \le u \le s$ in such a way that 
$a_{u}$ is contained in the interval $\bigl( (k-1)/\lambda_{i},\,k/\lambda_{i} \bigr]$. 
Then we define the $k$-th entry $\chi^{(i)}_{k}$ of the partition $\chi^{(i)}$ 
to be $\pair{\vpi_{i}}{\xi_{u}}$, the coefficient of $\alpha_{i}^{\vee}$ in $\xi_{u}$; 
we know from (the proof of) \cite[Proposition~7.2.1]{INS} that 
$\chi^{(i)}_{k}$ does not depend on the choice of $u$ above. 
Since $\xi_{1} > \cdots > \xi_{s-1} > 0=\xi_{s}$, we see that
$\chi^{(i)}_{1} \ge \cdots \ge 
\chi^{(i)}_{\lambda_{i}-1} \ge \chi^{(i)}_{\lambda_{i}}=0$. 
Hence, for each $i \in I \setminus \J$, 
we obtain a partition $\chi^{(i)}$ of length less than $\lambda_{i}$. 
For $i \in \J$, we set $\chi^{(i)}:=\emptyset$. 
Thus we obtain an element 
$\bchi = (\chi^{(i)})_{i \in I} \in \Par(\lambda)$, and hence 
a map from $\Conn(\SLS(\lambda))$ to $\Par(\lambda)$. 
Moreover, we know from \cite[Proposition~7.2.1]{INS} that 
this map is bijective; 
we denote by $\pi_{\bchi} \in \SLS(\lambda)$ 
the element of the form \eqref{eq:ext} 
corresponding to $\bchi \in \Par(\lambda)$ under this bijection.
For $\bchi \in \Par(\lambda)$, we denote by 
$\SLS_{\bchi}(\lambda)$ the connected component of $\SLS(\lambda)$ 
containing $\pi_{\bchi}$. 

%
%
\begin{rem} \label{rem:init}
Let $\bchi = (\chi^{(i)})_{i \in I} \in \Par(\lambda)$, 
with $\chi^{(i)} = (\chi^{(i)}_{1} \ge \cdots )$ for $i \in I$; 
note that $\chi^{(i)}_{1} = 0$ if $\chi^{(i)} = \emptyset$. 
We set
\begin{equation} \label{eq:xi}
\ip{\bchi} := \sum_{i \in I} \chi^{(i)}_{1}\alpha_{i}^{\vee} \in Q^{\vee};
\end{equation}
we see from the definition that $\iota(\pi_{\bchi}) = \PJ(t_{\ip{\bchi}})$. 
\end{rem}

%
\subsection{Affine Weyl group action.}
\label{subsec:Weyl}

Let $\CB$ be a regular crystal in the sense of \cite[Sect.~2.2]{Kas02} 
(or, a normal crystal in the sense of \cite[p.\,389]{HK});
for example, $\SLS(\lambda)$ for $\lambda \in P^{+}$ is a regular crystal
by Theorem~\ref{thm:isom}, and hence so is 
$\SLS(\lambda) \otimes \SLS(\mu)$ for $\lambda,\,\mu \in P^{+}$.
We know from \cite[Sect.~7]{Kas94} that 
the affine Weyl group $W_{\af}$ acts on $\CB$ as follows: 
for $b \in \CB$ and $i \in I_{\af}$, 
%
%
\begin{equation} \label{eq:W-act}
s_{i} \cdot b := 
\begin{cases}
f_{i}^{n}b & \text{if $n:=\pair{\wt(b)}{\alpha_{i}^{\vee}} \ge 0$}, \\[1.5mm]
e_{i}^{-n}b & \text{if $n:=\pair{\wt(b)}{\alpha_{i}^{\vee}} \le 0$}. 
\end{cases}
\end{equation}
The following lemma is shown by induction on 
the (ordinary) length $\ell(x)$ of $x$ and 
the tensor product rule for crystals 
(see also \cite[Lemma~7.2]{KNS}). 

%
\begin{lem} \label{lem:Weyl} \mbox{}
\begin{enu}
\item Let $\lambda \in P^{+}$, and take $\J=\Jl$ as in \eqref{eq:J}.  
If $\pi \in \SLS(\lambda)$ is of the form \eqref{eq:ext}, 
then for $x \in W_{\af}$, 
%
%
\begin{equation} \label{eq:tx0}
x \cdot \pi = 
\bigl(\PJ(xt_{\xi_1}),\dots,\PJ(xt_{\xi_{s-1}}),\PJ(x)\,;\,a_{0},a_{1},\dots,a_{s-1},a_{s}\bigr). 
\end{equation}

\item Let $\lambda,\,\mu \in P^{+}$. 
Let $\bsg \in \Par(\lambda)$, $\bchi \in \Par(\mu)$, and 
$\xi,\zeta \in Q^{\vee}$. Then, for $x \in W_{\af}$, 
%
%
\begin{equation} \label{eq:Weyl}
x \cdot \bigl( (t_{\xi} \cdot \pi_{\bsg}) \otimes (t_{\zeta} \cdot \pi_{\bchi}) \bigr)
= (xt_{\xi} \cdot \pi_{\bsg}) \otimes (xt_{\zeta} \cdot \pi_{\bchi}). 
\end{equation}
\end{enu}
\end{lem}
%
%
\subsection{Standard monomial theory.}
\label{subsec:SMT}

Let $\lambda=\sum_{i \in I}\lambda_{i}\vpi_{i},\,
\mu=\sum_{i \in I}\mu_{i}\vpi_{i} \in P^{+}$, and 
define $\Jl,\,\Jm,\,\J_{\lambda+\mu} \subset I$ 
as in \eqref{eq:J}. Following \cite[Proposition~3.4]{KNS}, 
we define $\SM(\lambda+\mu)$ to be the subset 
of $\SLS(\lambda) \otimes \SLS(\mu)$ consisting of 
those elements $\pi \otimes \eta$ satisfying the 
condition that there exists $y \in W_{\af}$ such that 
$\kappa(\eta) \sige \PS{\Jm}(y)$ and 
$\kappa(\pi) \sige \PS{\Jl}(\io{\eta}{y})$; 
for the definition of $\io{\eta}{y} \in W_{\af}$, see \eqref{eq:tix}.
We know from \cite[Theorem~3.1]{KNS} that 
$\SM(\lambda+\mu)$ is a subcrystal of 
$\SLS(\lambda) \otimes \SLS(\mu)$, and 
it is isomorphic as a crystal to $\SLS(\lambda+\mu)$. 
Now, we  briefly recall from the proof of \cite[Theorem~3.1]{KNS}
the description of the isomorphism from $\SLS(\lambda+\mu)$ to 
$\SM(\lambda+\mu) \subset \SLS(\lambda) \otimes \SLS(\mu)$ 
(which we denote by $\Phi_{\lambda\mu}$). 
Let $\Par(\lambda,\mu)$ be the subset of those elements $(\bsg, \bchi, \xi) \in 
\Par(\lambda) \times \Par(\mu) \times \QSv{I \setminus (\Jl \cup \Jm)}$ 
satisfying the condition that 
$c_{i} \ge \chi^{(i)}_{1}$ for all $I \setminus (\Jl \cup \Jm)$, 
where $\bchi=(\chi^{(i)})_{i \in I}$ with 
$\chi^{(i)} = (\chi^{(i)}_{1} \ge \cdots \ge \chi^{(i)}_{\mu_i-1} \ge 0)$ for $i \in I$, 
and $\xi=\sum_{ i \in I \setminus (\Jl \cup \Jm) } c_{i}\alpha_{i}^{\vee}$. 
By \cite[Proposition~7.8]{KNS}, 
there exists a bijection from $\Par(\lambda,\mu)$ 
to the set of connected components of $\SM(\lambda+\mu)$, 
which sends $(\bsg, \bchi, \xi) \in \Par(\lambda,\mu)$ to 
the connected component of $\SM(\lambda+\mu)$ containing 
the element $(t_{\xi} \cdot \pi_{\bsg}) \otimes \pi_{\bchi}$. 
Here, we define a map 
$\Theta:\Par(\lambda,\mu) \rightarrow \Par(\lambda+\mu)$ 
as follows. Let $(\bsg,\,\bchi,\,\xi) \in \Par(\lambda,\mu)$, and 
write $\bsg \in \Par(\lambda)$, $\bchi \in \Par(\mu)$, 
$\xi \in \QSv{I \setminus (\Jl \cup \Jm)}$ as: 
\begin{equation} \label{eq:omg1}
\begin{cases}
\bsg = (\sigma^{(i)})_{i \in I}, \quad \text{with }
  \sigma^{(i)}=(\sigma^{(i)}_{1} \ge \cdots \ge \sigma^{(i)}_{\lambda_i-1} \ge 0) \quad
  \text{for $i \in I$}, \\[2mm]
\bchi = (\chi^{(i)})_{i \in I}, \quad \text{with }
  \chi^{(i)}=(\chi^{(i)}_{1} \ge \cdots \ge \chi^{(i)}_{\mu_i-1} \ge 0) \quad
  \text{for $i \in I$}, \\[2mm]
\xi=\sum_{i \in I \setminus (\Jl \cup \Jm)} 
  c_{i}\alpha_{i}^{\vee};\ \text{recall that $c_{i} \ge \chi^{(i)}_{1}$ 
  for all $i \in I \setminus (\Jl \cup \Jm)$}.
\end{cases}
\end{equation}
For each $i \in I$, we set 
\begin{equation} \label{eq:omg2}
\omega^{(i)} := 
\bigl(
 \overbrace{\sigma^{(i)}_{1}+c_{i} \ge \cdots \ge \sigma^{(i)}_{\lambda_i-1}+c_{i} \ge c_{i}}^{%
 \text{Remove these parts if $i \in \Jl$}}
 \ge \underbrace{\chi^{(i)}_{1} \ge \cdots \ge \chi^{(i)}_{\mu_i-1}}_{%
 \begin{subarray}{c}
 \text{Remove these parts} \\
 \text{and set $c_{i}=0$ if $i \in \Jm$}
 \end{subarray}}\bigr), 
\end{equation}
which is a partition of length less than $\lambda_{i}+\mu_{i}$. 
Define $\Theta(\bsg,\,\bchi,\,\xi):=(\omega^{(i)})_{i \in I} \in \Par(\lambda+\mu)$; 
we can deduce that this map $\Theta$ is bijective. 
We know from \cite[Sect.~7]{KNS} that 
there exists an isomorphism 
\begin{equation} \label{eq:Phi}
\Phi_{\lambda\mu}:\SLS(\lambda+\mu) 
\stackrel{\sim}{\rightarrow} \SM(\lambda+\mu) \quad 
(\hookrightarrow \SLS(\lambda) \otimes \SLS(\mu))
\end{equation}
of crystals, which sends $\pi_{\bom}$ to 
$(t_{\xi} \cdot \pi_{\bsg}) \otimes \pi_{\bchi}$
if $\Theta^{-1}(\bom)=(\bsg,\,\bchi,\,\xi)$ for $\bom \in \Par(\lambda+\mu)$. 
%
%
\subsection{Standard monomial theory for Demazure crystals.}
\label{subsec:lem-SMT}

Let $\lambda=\sum_{i \in I}\lambda_{i}\vpi_{i},\,
\mu=\sum_{i \in I}\mu_{i}\vpi_{i} \in P^{+}$, and 
define $\Jl,\,\Jm,\,\J_{\lambda+\mu} \subset I$ as in \eqref{eq:J}.
For $\pi \otimes \eta \in \SLS(\lambda) \otimes \SLS(\mu)$, we set 
%
%
\begin{equation} \label{eq:dual3}
(\pi \otimes \eta)^{\ast}:=\eta^{\ast} \otimes \pi^{\ast} 
\in \SLS(\mu^{\ast}) \otimes \SLS(\lambda^{\ast});
\end{equation}
for the definition of the dual paths $\eta^{\ast}$ and $\pi^{\ast}$, 
see \eqref{eq:dual}. Then we have 
%
%
\begin{equation} \label{eq:dual4}
\begin{cases}
\wt( (\pi \otimes \eta)^{\ast}) = - \wt (\pi \otimes \eta), \\[1mm]
\bigl( e_{i}(\pi \otimes \eta) \bigr)^{\ast} = 
f_{i}\bigl((\pi \otimes \eta)^{\ast}\bigr), \quad 
\bigl( f_{i}(\pi \otimes \eta) \bigr)^{\ast} = 
e_{i}\bigl((\pi \otimes \eta)^{\ast}\bigr)
\end{cases}
\end{equation}
for $\pi \otimes \eta \in \SLS(\lambda) \otimes \SLS(\mu)$ and $i \in I_{\af}$; 
cf. \eqref{eq:dual2}. Also, for $y \in W_{\af}$, we set
%
%
\begin{equation} \label{eq:smtdem}
\SM_{\sige y}(\lambda+\mu) : = \Phi_{\lambda\mu}(\SLS_{\sige y}(\lambda+\mu)), \qquad 
\SM_{\sile y}(\lambda+\mu) : = \Phi_{\lambda\mu}(\SLS_{\sile y}(\lambda+\mu)).
\end{equation}
%
%
\begin{thm} \label{thm:SMT}
Keep the notation and setting above. Let $y \in W_{\af}$. 
\begin{enu}
\item An element $\pi \otimes \eta \in \SLS(\lambda) \otimes \SLS(\mu)$
is contained in $\SM_{\sige y}(\lambda+\mu)$
if and only if $\kappa(\eta) \sige \PS{\Jm}(y)$ and 
$\kappa(\pi) \sige \PS{\Jl}(\io{\eta}{y})$. 

\item An element $\pi \otimes \eta \in \SLS(\lambda) \otimes \SLS(\mu)$ 
is contained in  $\SM_{\sile y}(\lambda+\mu)$ if and only if 
$\iota(\pi) \sile \PS{\Jl}(y)$ and 
$\iota(\eta) \sile \PS{\Jm}(\kap{\pi}{y})$. 
\end{enu}
\end{thm}

\begin{proof}
Part (1) follows from \cite[Theorem~3.5]{KNS}. 
Let us prove part (2). 
By using \cite[Lemma~7.2 and Remark~7.3]{KNS} 
and \eqref{eq:dual2}, \eqref{eq:dual4}, 
we deduce that the following diagram is commutative:
\begin{equation*}
\begin{CD}
\SLS(\lambda+\mu) 
 & @>{\Phi_{\lambda\mu}}>> \SLS(\lambda) \otimes \SLS(\mu) \\
@V{\ast}VV & @VV{\ast}V \\
\SLS(\mu^{\ast}+\lambda^{\ast}) 
 & @>{\Phi_{\mu^{\ast}\lambda^{\ast}}}>> 
 \SLS(\mu^{\ast}) \otimes \SLS(\lambda^{\ast}),
\end{CD}
\end{equation*}
where $\lambda^{\ast}=-\lng\lambda$ and $\mu^{\ast}=-\lng\mu$.
Also, it follows from Lemma~\ref{lem:lng1} that 
$\bigl( \SLS_{\sile y}(\lambda+\mu) \bigr)^{\ast} = 
\SLS_{\sige y\lng}(\mu^{\ast} + \lambda^{\ast})$. 
Therefore, we obtain 
\begin{equation*}
\SM_{\sile x}(\lambda+\mu)=
\bigl(
 \underbrace{\SM_{\sige x\lng}(\mu^{\ast}+\lambda^{\ast})}_{%
 \subset \SLS(\mu^{\ast}) \otimes \SLS(\lambda^{\ast})}
\bigr)^{\ast}.
\end{equation*}
Hence part (2) follows from part (1) and Lemma~\ref{lem:io-kap}. 
This proves the theorem. 
\end{proof}
%
%
\begin{prop} \label{prop:smt}
Let $\psi \in \SLS(\lambda+\mu)$, 
and write $\Phi_{\lambda\mu}(\psi) \in \SLS(\lambda) \otimes \SLS(\mu)$ as
$\Phi_{\lambda\mu}(\psi) = \pi \otimes \eta$, 
with $\pi \in \SLS(\lambda)$ and $\eta \in \SLS(\mu)$. 
Let $y \in W_{\af}$. 
\begin{enu}
\item If $\psi \in \SLS_{\sige y}(\lambda+\mu)$, then 
$\io{\psi}{y} = \io{\pi}{\io{\eta}{y}}$. 

\item If $\psi \in \SLS_{\sile y}(\lambda+\mu)$, then 
$\kap{\psi}{y} = \kap{\eta}{\kap{\pi}{y}}$. 
\end{enu}
\end{prop}

In order to prove this proposition, 
we need some technical lemmas.
%
%
\begin{lem} \label{lem:iota1}
Let $\nu \in P^{+}$ and $y \in W_{\af}$. 
Let $\psi \in \SLS_{\sige y}(\nu)$ and 
$i \in I$ be such that $\psi':=f_{i}\psi \ne \bzero${\rm;}
note that $\psi' \in \SLS_{\sige y}(\nu)$ by Remark~\ref{rem:stable}. 
If $\ve_{i}(\psi) \ge 1$, then $\io{\psi'}{y} = \io{\psi}{y}$.
If $\ve_{i}(\psi)=0$, then 
$\io{\psi}{y}^{-1} \alpha_{i} \in \Delta^{+}+\BZ\delta$. 
Moreover, $\io{\psi'}{y}$ is identical to $\io{\psi}{y}$ or 
$s_{i}\io{\psi}{y}$. 
\end{lem}

\begin{proof}
We set $\J=\J_{\nu} \subset I$ as in \eqref{eq:J}. 
Write $\psi \in \SLS_{\sige y}(\nu)$ as 
$\psi = (x_{1},\dots,x_{s};a_{0},a_{1},\dots,a_{s})$, 
and define $y=\ti{x}_{s+1}$, $\ti{x}_{s}$, $\dots$, $\ti{x}_{2}$, 
$\ti{x}_{1}=\io{\psi}{y}$ by the same formula as \eqref{eq:tix}.
Assume that $\psi'=f_{i}\psi$ is of the form:
%
%
\begin{equation} \label{eq:psi}
\begin{split}
& \psi'=f_{i}\psi = 
  (x_{1},\dots,x_{p},s_{i}x_{p+1},\dots,s_{i}x_{q},x_{q},x_{q+1},\dots,x_{s}; \\
& \hspace{50mm} 
  a_{0},a_{1},\dots,a_{p}=t_{0},\dots,a_{q-1},t_{1},a_{q},a_{q+1},\dots,a_{s})
\end{split}
\end{equation}
for some $0 \le p < q \le s$ (see \eqref{eq:t-f} and \eqref{eq:fpi}); remark that 
$\pair{x_{u}\nu}{\alpha_{i}^{\vee}} > 0$ for all $p+1 \le u \le q$ 
(see the comment after \eqref{eq:t-f}), which implies that 
%
%
\begin{equation} \label{eq:c1-1}
x_{u}^{-1}\alpha_{i} \in 
(\Delta^{+} \setminus \DeJ^{+}) + \BZ\delta, 
\quad \text{and hence} \quad 
\ti{x}_{u}^{-1}\alpha_{i} \in \Delta^{+} + \BZ\delta
\end{equation}
for all $p+1 \le u \le q$. 
We set 
$(y_{1},\dots,\,y_{q}):=(x_{1},\dots,x_{p},s_{i}x_{p+1},\dots,s_{i}x_{q})$, 
and define $\ti{y}_{q},\,\dots,\,\ti{y}_{1}$ by
\begin{equation*}
\begin{cases}
\ti{y}_{q}:=\min \Lige{\ti{x}_{q}}{y_{q}}, & \\[2mm]
\ti{y}_{u}:=\min \Lige{\ti{y}_{u+1}}{y_{u}} & \text{for $1 \le u \le q-1$};
\end{cases}
\end{equation*}
we have $\io{\psi'}{y}=\ti{y}_{1}$. 
Since $\pair{x_{q}\nu}{\alpha_{i}^{\vee}} > 0$, i.e.,
$x_{q}^{-1}\alpha_{i} \in (\Delta^{+} \setminus \Delta^{+}_{\J})+\BZ \delta$, 
and since $\ti{x}_{q}^{-1}\alpha_{i} \in \Delta^{+} + \BZ\delta$ by \eqref{eq:c1-1}, 
we see by Lemma~\ref{lem:dia}\,(1) and (2) that 
$\ti{y}_{q}=\min \Lige{\ti{x}_{q}}{y_{q}} = 
 \min \Lige{\ti{x}_{q}}{s_{i}x_{q}}$ is identical to 
 $\min \Lige{s_{i}\ti{x}_{q}}{s_{i}x_{q}} = 
 s_{i} \bigl( \min \Lige{\ti{x}_{q}}{x_{q}} \bigr) = 
 s_{i}\ti{x}_{q}$. 
Also, we deduce by Lemma~\ref{lem:dia}\,(1), 
together with \eqref{eq:c1-1}, 
that $\ti{y}_{u}=s_{i}\ti{x}_{u}$ for $p+1 \le u \le q-1$. 

Assume first that $\ve_{i}(\psi) \ge 1$; 
in order to prove that 
$\io{\psi'}{y} = \ti{y}_{1} = \ti{x}_{1} = \io{\psi}{y}$, 
it suffices to show that 
\begin{equation} \label{eq:c1-0}
\ti{y}_{u} = \ti{x}_{u} \quad 
 \text{for some $1 \le u \le p$}.
\end{equation}
Since $\ve_{i}(\psi) \ge 1$, 
we see by the definition of the root operator $e_{i}$ (see \eqref{eq:t-e}) that 
there exists $1 \le u \le p$ such that 
$\pair{x_{u}\nu}{\alpha_{i}^{\vee}} < 0$. 
We set $k:=\max \bigl\{1 \le u \le p \mid 
\pair{x_{u}\nu}{\alpha_{i}^{\vee}} < 0 \bigr\}$; 
note that 
\begin{equation} \label{eq:c1-2a}
x_{k}^{-1}\alpha_{i} \in 
(\Delta^{-} \setminus \DeJ^{-})+\BZ\delta.
\end{equation}
By the definition of the root operator $f_{i}$ (see \eqref{eq:t-f}), 
we see that $\pair{x_{u}\nu}{\alpha_{i}^{\vee}} =0$ 
for all $k+1 \le u \le p$. Suppose, for a contradiction, that 
$\ti{x}_{u}^{-1}\alpha_{i} \in \Delta^{-}+\BZ\delta$ 
for some $k+1 \le u \le p$, and set $m:=\max \bigl\{k+1 \le u \le p \mid 
\ti{x}_{u}^{-1}\alpha_{i} \in \Delta^{-}+\BZ\delta \bigr\}$; 
remark that $\ti{x}_{m+1}^{-1}\alpha_{i} \in \Delta^{+}+\BZ\delta$ 
(for the case $m=p$, see \eqref{eq:c1-1}). Also, 
we have $x_{m}^{-1}\alpha_{i} \in \Delta_{J}+\BZ\delta$ 
since $\pair{x_{m}\nu}{\alpha_{i}^{\vee}} =0$, and 
$\ti{x}_{m}^{-1}\alpha_{i} \in \Delta^{-}+\BZ\delta$ by the definition of $m$. 
However, this is a contradiction by Lemma~\ref{lem:dia}\,(3) 
(applied to $\ti{x}_{m}=\min \Lige{\ti{x}_{m+1}}{x_{m}}$). 
Therefore, we conclude that 
\begin{equation} \label{eq:c1-2}
\ti{x}_{u}^{-1}\alpha_{i} \in \Delta^{+}+\BZ\delta \quad 
\text{for all $k+1 \le u \le p$}.
\end{equation}
Here, it follows from Lemma~\ref{lem:dia}\,(3) that 
$\ti{y}_{p} = \min \Lige{\ti{y}_{p+1}}{y_{p}} = 
\min \Lige{s_{i}\ti{x}_{p+1}}{x_{p}}$ is identical to 
$\ti{x}_{p}$ or $s_{i}\ti{x}_{p}$. 
If $\ti{y}_{p} = \ti{x}_{p}$, then 
\eqref{eq:c1-0} holds for $u=p$. 
Hence we may assume that $\ti{y}_{p} = s_{i}\ti{x}_{p}$. 
In this case, it follows again from Lemma~\ref{lem:dia}\,(3) that 
$\ti{y}_{p-1} = \min \Lige{\ti{y}_{p}}{y_{p-1}} = 
\min \Lige{s_{i}\ti{x}_{p}}{x_{p-1}}$ is identical to 
$\ti{x}_{p-1}$ or $s_{i}\ti{x}_{p-1}$. 
If $\ti{y}_{p-1} = \ti{x}_{p-1}$, then \eqref{eq:c1-0} holds for $u=p-1$. 
Hence we may assume that $\ti{y}_{p-1} = s_{i}\ti{x}_{p-1}$. 
By repeating this argument, we may assume that 
%
%
\begin{equation} \label{eq:c1-3}
\ti{y}_{u} = s_{i}\ti{x}_{u} \quad 
 \text{for all $k+1 \le u \le p$}.
\end{equation}
Then, by Lemma~\ref{lem:dia}\,(2) and \eqref{eq:c1-2a}, \eqref{eq:c1-2}, 
$\ti{y}_{k} = \min \Lige{\ti{y}_{k+1}}{y_{k}} = 
\min \Lige{s_{i}\ti{x}_{k+1}}{x_{k}}$ is identical to $\ti{x}_{k}$. 
Thus we obtain $\io{\psi'}{x} = \ti{y}_{1} = \ti{x}_{1} = \io{\psi}{x}$. 

Assume next that $\ve_{i}(\psi) = 0$. 
Suppose, for a contradiction, that 
$\io{\psi}{y}^{-1}\alpha_{i} = \ti{x}_{1}^{-1}\alpha_{i} \in \Delta^{-}+\BZ\delta$, 
and set $l:=\min \bigl\{ 2 \le u \le p+1 \mid 
\ti{x}_{u}^{-1}\alpha_{i} \in \Delta^{+}+\BZ\delta \bigr\}$ 
(recall \eqref{eq:c1-1}); note that 
$\ti{x}_{u}^{-1}\alpha_{i} \in \Delta^{-}+\BZ\delta$
for all $1 \le u \le l-1$, which implies that 
$\pair{x_{u}\nu}{\alpha_{i}^{\vee}} =
 \pair{\ti{x}_{u}\nu}{\alpha_{i}^{\vee}} \le 0$
for all $1 \le u \le l-1$. 
Since $\ve_{i}(\psi) = 0$, we deduce by 
the definition of the root operator $e_{i}$ that 
$\pair{x_{u}\nu}{\alpha_{i}^{\vee}} = 0$
for all $1 \le u \le l-1$, and hence 
$x_{u}^{-1}\alpha_{i} \in \DeJ+\BZ\delta$ 
for all $1 \le u \le l-1$. 
In particular, we obtain 
$x_{l-1}^{-1}\alpha_{i} \in \DeJ+\BZ\delta$ and 
$\ti{x}_{l-1}^{-1}\alpha_{i} \in \Delta^{-}+\BZ\delta$. 
However, since $\ti{x}_{l}^{-1}\alpha_{i} \in \Delta^{+}+\BZ\delta$, 
this is a contradiction by Lemma~\ref{lem:dia}\,(3) 
(applied to $\ti{x}_{l-1}=\min \Lige{\ti{x}_{l}}{x_{l-1}}$). 
Thus we obtain $\io{\psi}{y}^{-1}\alpha_{i} \in \Delta^{+}+\BZ\delta$. 

Now, recall that $\psi'=f_{i}\psi$ is of form \eqref{eq:psi}. 
Since $\ve_{i}(\psi) = 0$ by our assumption, 
we see by the definition of the root operators $e_{i}$ and $f_{i}$ that 
\begin{equation} \label{eq:c1-4}
\text{$H^{\psi}_{i}(t) \ge 0$ 
for $t \in [a_{0},a_{p}] = [0,t_{0}]$, with 
$H^{\psi}_{i}(0) = H^{\psi}_{i}(t_{0}) = 0$}. 
\end{equation}
If there exists $1 \le u \le p$ such that 
$\pair{x_{u}\nu}{\alpha_{i}^{\vee}} < 0$, 
then we deduce by the same argument as for \eqref{eq:c1-0} that 
$\io{\psi'}{x} = \ti{y}_{1} = \ti{x}_{1} = \io{\psi}{x}$. 
Hence we may assume that 
$\pair{x_{u}\nu}{\alpha_{i}^{\vee}} \ge 0$ 
for all $1 \le u \le p$. In this case, we have 
$\pair{x_{u}\nu}{\alpha_{i}^{\vee}} = 0$ 
for all $1 \le u \le p$ by \eqref{eq:c1-4}. 
By the same argument as for \eqref{eq:c1-2}, we deduce that 
$\ti{x}_{u}^{-1}\alpha_{i} \in \Delta^{+}+\BZ\delta$
for all $1 \le u \le p$.
Furthermore, by the same argument as for \eqref{eq:c1-3}, 
we see that $\io{\psi'}{x}$ is identical to $\io{\psi}{x}$ or 
$s_{i}\io{\psi}{x}$. This proves the lemma. 
\end{proof}

%
%
\begin{lem} \label{lem:iota2}
Let $\lambda,\mu \in P^{+}$, and $y \in W_{\af}$. 
Let $\pi \otimes \eta \in \SM_{\sige y}(\lambda+\mu)$ 
and $i \in I$ be such that $\pi' \otimes \eta' := 
f_{i}(\pi \otimes \eta) \ne \bzero$, 
where $\pi' \in \SLS(\lambda)$ and $\eta' \in \SLS(\mu)$\,{\rm;}
notice that $\pi' \otimes \eta' \in \SM_{\sige y}(\lambda+\mu)$ 
by Remark~\ref{rem:stable}. 
If $\ve_{i}(\pi \otimes \eta) \ge 1$, then 
$\io{\pi'}{\io{\eta'}{y}} = \io{\pi}{\io{\eta}{y}}$.
If $\ve_{i}(\pi \otimes \eta)=0$, then 
$\io{\pi}{\io{\eta}{y}}^{-1} \alpha_{i} \in \Delta^{+}+\BZ\delta$. 
Moreover, $\io{\pi'}{\io{\eta'}{y}}$ is identical to 
$\io{\pi}{\io{\eta}{y}}$ or $s_{i}\io{\pi}{\io{\eta}{y}}$. 
\end{lem}
\begin{proof}
If $\ve_{i}(\pi \otimes \eta) = 0$, then 
we see by the tensor product rule for crystals that 
$\ve_{i}(\pi) = 0$. Since
$\pi \in \SLS_{\sige \io{\eta}{y}}(\lambda)$ 
by Theorem~\ref{thm:SMT}\,(1), 
we see by Lemma~\ref{lem:iota1} that 
$\io{\pi}{\io{\eta}{y}}^{-1} \alpha_{i} \in \Delta^{+}+\BZ\delta$ 
in this case. 

Assume first that 
$f_{i}(\pi \otimes \eta) = f_{i}\pi \otimes \eta$, 
i.e., $\pi'=f_{i}\pi$ and $\eta'=\eta$; 
note that $\vp_{i}(\pi) > \ve_{i}(\eta)$ 
by the tensor product rule for crystals, 
and that $\io{\eta'}{y} = \io{\eta}{y}$. 
If $\ve_{i}(\pi \otimes \eta) \ge 1$, then 
we see by the tensor product rule for crystals 
and the inequality $\vp_{i}(\pi) > \ve_{i}(\eta)$ that 
$\ve_{i}(\pi) = \ve_{i}(\pi \otimes \eta) \ge 1$. 
Hence it follows from Lemma~\ref{lem:iota1}, 
applied to $\pi \in \SLS_{\sige \io{\eta}{y}}(\lambda)$, that 
$\io{\pi'}{\io{\eta'}{y}} = 
\io{\pi'}{\io{\eta}{y}} = 
\io{\pi}{\io{\eta}{y}}$.
If $\ve_{i}(\pi \otimes \eta) = 0$, 
then we have $\ve_{i}(\pi) = 0$, as seen above. 
Therefore, we see by Lemma~\ref{lem:iota1}, 
applied to $\pi \in \SLS_{\sige \io{\eta}{y}}(\lambda)$, that 
$\io{\pi'}{\io{\eta'}{y}} = 
 \io{\pi'}{\io{\eta}{y}}$ is identical to 
 $\io{\pi}{\io{\eta}{y}}$ or $s_{i}\io{\pi}{\io{\eta}{y}}$. 

Assume next that 
$f_{i}(\pi \otimes \eta) = \pi \otimes f_{i}\eta$, 
i.e., $\pi'=\pi$ and $\eta'=f_{i}\eta$. 
We see by Lemma~\ref{lem:iota1} that 
$\io{\eta'}{y}$ is identical to 
$\io{\eta}{y}$ or $s_{i}\io{\eta}{y}$. 
If $\io{\eta'}{y} = \io{\eta}{y}$, then 
it is obvious that $\io{\pi'}{\io{\eta'}{y}} = 
\io{\pi}{\io{\eta}{y}}$. 
Assume now that $\io{\eta'}{y} = s_{i}\io{\eta}{y}$. 
We see by Lemma~\ref{lem:iota1} that 
$\io{\eta'}{y} = s_{i}\io{\eta}{y}$ 
only if $\ve_{i}(\eta) = 0$; note that 
$\io{\eta}{y}^{-1}\alpha_{i} \in \Delta^{+}+\BZ\delta$. 
Write $\pi \in \SLS_{\sige \io{\eta}{y}}(\nu)$ as 
$\pi = (x_{1},\dots,x_{s};a_{0},a_{1},\dots,a_{s})$,
and define 
\begin{align*}
\io{\eta}{y} & =\ti{x}_{s+1},\ti{x}_{s},\dots,\ti{x}_{2},\ti{x}_{1}=\io{\pi}{\io{\eta}{y}}, \\
\io{\eta'}{y} = s_{i}\io{\eta}{y} & =\ti{y}_{s+1},\ti{y}_{s},\dots,\ti{y}_{2},\ti{y}_{1}=
  \io{\pi}{\io{\eta'}{y}},
\end{align*}
by the same formula as \eqref{eq:tix}. 
If $\ve_{i}(\pi \otimes \eta) \ge 1$, then 
we see by the tensor product rule for crystals and 
the equality $\ve_{i}(\eta) = 0$ that $\ve_{i}(\pi) \ge 1$. 
Hence, by the same argument as for \eqref{eq:c1-0} (with $p$ replaced by $s$), 
we deduce that $\ti{y}_{u} = \ti{x}_{u}$ for some $1 \le u \le s$, 
which implies that $\io{\pi'}{\io{\eta'}{y}} = 
\io{\pi}{\io{\eta}{y}}$. Also, if $\ve_{i}(\pi \otimes \eta) = 0$, 
then we have $\ve_{i}(\pi) = 0$, as seen above. 
By the same argument as in the last paragraph of 
the proof of Lemma~\ref{lem:iota1}, we deduce that 
$\io{\pi'}{\io{\eta'}{y}}$ is identical to 
 $\io{\pi}{\io{\eta}{y}}$ or $s_{i}\io{\pi}{\io{\eta}{y}}$. 
This proves the lemma. 
\end{proof}
%
%
\begin{lem} \label{lem:n0}
Let $\lambda,\mu \in P^{+}$, and $v \in W$. 
If we set $v':=\tbmin{\lng}{\Jm}{v}$, then 
\begin{equation} \label{eq:n0}
\tbmin{\lng}{\Jlm}{v} = \tbmin{\lng}{\Jl}{v'}.
\end{equation}
\end{lem}

\begin{proof}
We set $w:=\tbmin{\lng}{\Jlm}{v}$ and 
$w':=\tbmin{\lng}{\Jl}{v'}$. We prove the assertion 
by descending induction on $\ell(v)$. 
If $v = \lng$, then we see by Remark~\ref{rem:tbmin} that 
$w = \lng = v' = w'$. Assume now that $\ell(v) < \ell(\lng)$, 
and take $i \in I$ such that $\ell(s_{i}v) = \ell(v)+1$, 
or equivalently, $v^{-1}\alpha_{i} \in \Delta^{+}$. 
We set
\begin{equation*}
w_1:=\tbmin{\lng}{\Jlm}{s_{i}v}, \quad 
v_1':=\tbmin{\lng}{\Jm}{s_{i}v}, \quad
w_1':=\tbmin{\lng}{\Jl}{v_1'}; 
\end{equation*}
by our induction hypothesis, we have $w_1 = w_1'$. 

\paragraph{Case 1.}
%
Assume that $\lng^{-1}\alpha_{i} \in \Delta^{-} \setminus \DeS{\Jlm}^{-}$, 
or equivalently, $i^{\ast} \notin \Jlm = \Jl \cap \Jm$. 
We deduce from \cite[Lemma~3.6\,(2)]{NNS3} that $w_{1} = w$. 
If $\lng^{-1}\alpha_{i} \in \Delta^{-} \setminus \DeS{\Jm}^{-}$, 
or equivalently, if $i^{\ast} \notin \Jm$, then 
we deduce from \cite[Lemma~3.6\,(2)]{NNS3} that $v_1' = v'$, 
and hence $w_{1}'=w'$. Hence, by our induction hypothesis, 
we obtain $w = w'$.
Thus we may assume that $\lng^{-1}\alpha_{i} \in \DeS{\Jm}$, 
or equivalently, $i^{\ast} \in \Jm$. 
In this case, it follows from \cite[Lemma~3.6\,(3)]{NNS3} that 
$(v')^{-1}\alpha_{i} \in \Delta^{+}$, and 
$v_{1}'$ is identical to $v'$ or $s_{i}v'$. 
If $v_1' = v'$, then we obtain $w = w'$ 
in exactly the same way as above. 
Assume now that $v_1' = s_{i}v'$. 
Since $i^{\ast} \notin \Jlm = \Jl \cap \Jm$ and 
since $i^{\ast} \in \Jm$, we see that $i^{\ast} \notin \Jl$, 
which implies that $\lng^{-1}\alpha_{i} \in 
\Delta^{-} \setminus \DeS{\Jl}^{-}$. 
Hence it follows from \cite[Lemma~3.6\,(2)]{NNS3} that 
\begin{equation*}
w_1'=\tbmin{\lng}{\Jl}{v_1'} = \tbmin{\lng}{\Jl}{s_{i}v'} = 
\tbmin{\lng}{\Jl}{v'} = w'.
\end{equation*}
Therefore, by our induction hypothesis, we obtain $w = w'$.

\paragraph{Case 2.}
%
Assume that $\lng^{-1}\alpha_{i} \in \DeS{\Jlm}$, 
or equivalently, $i^{\ast} \in \Jlm = \Jl \cap \Jm$. 
In this case, it follows from \cite[Lemma~3.6\,(3)]{NNS3} that 
$w^{-1}\alpha_{i} \in \Delta^{+}$, and 
$w_{1}$ is identical to $w$ or $s_{i}w$. 
Similarly, since $\lng^{-1}\alpha_{i} \in \DeS{\Jlm} \subset \DeS{\Jm}$ 
by the assumption, we have $(v')^{-1}\alpha_{i} \in \Delta^{+}$, and 
$v_1'$ is identical to $v'$ or $s_{i}v'$. 
Since $(v')^{-1}\alpha_{i} \in \Delta^{+}$, and since 
$\lng^{-1}\alpha_{i} \in \DeS{\Jlm} \subset \DeS{\Jl}$ 
by the assumption, it follows from \cite[Lemma~3.6\,(3)]{NNS3} that 
$(w')^{-1}\alpha_{i} \in \Delta^{+}$. 
In addition, if $v_1' = s_{i}v'$, then 
it follows from \cite[Lemma~3.6\,(3)]{NNS3} that 
$w_1'$ is identical to $w'$ or $s_{i}w'$; 
if $v_1' = v'$, then $w_1' = w'$. 
In both cases, $w_1'$ is identical to $w'$ or $s_{i}w'$. 
Since $w_1 = w_1'$ by our induction hypothesis, 
we deduce that $w$ is identical to $w'$ or $s_{i}w'$. 
Since $w^{-1}\alpha_{i} \in \Delta^{+}$ and 
$(w')^{-1}\alpha_{i} \in \Delta^{+}$ as seen above, 
we obtain $w=w'$, as desired. This proves the lemma. 
\end{proof}

\begin{proof}[Proof of Proposition~\ref{prop:smt}]
We give a proof only for part (1), since we can show part (2) by 
taking the dual paths $\psi^{\ast}$ and 
$(\pi \otimes \eta)^{\ast} = \eta^{\ast} \otimes \pi^{\ast}$, 
and then by applying Lemma~\ref{lem:io-kap} and part (1) to them. 
First, we claim that there exists $i_1,\dots,i_{n} \in I_{\af}$ such that 
\begin{equation*}
f_{i_{1}}f_{i_{2}} \cdots f_{i_{n-1}}f_{i_{n}}\psi = 
(\lng t_{\gamma}) \cdot \pi_{\bom}
\end{equation*}
for some $\gamma \in Q^{\vee}$ and $\bom \in \Par(\lambda+\mu)$; 
for the definition of $\pi_{\bom}$, see \S\ref{subsec:conn}. 
Indeed, by \cite[Lemma~5.4.1]{NS16}, 
there exist $j_1,\dots,j_{m} \in I_{\af}$ such that 
$f_{j_{1}}f_{j_{2}} \cdots f_{j_{m-1}}f_{j_{m}}\psi = 
t_{\gamma} \cdot \pi_{\bom}$ 
for some $\gamma \in Q^{\vee}$ and $\bom \in \Par(\lambda+\mu)$. 
We deduce from definition \eqref{eq:W-act} 
that the action of $\lng \in W$ on the (extremal) element 
$t_{\gamma} \cdot \pi_{\bom}$ is given only by the root operators $f_{i}$, 
$i \in I$. Hence we have verified the claim above 
(see also \cite[Lemma~3.11]{NNS3}). 
Here it follows from Remark~\ref{rem:stable} that 
$f_{i_{k}}f_{i_{k+1}} \cdots f_{i_{n-1}}f_{i_{n}}\psi \in 
\SLS_{\sige y}(\lambda+\mu)$ for all $1 \le k \le n+1$.
Now we proceed by induction on $n$. 
Assume first that $n=0$, that is, $\psi = (\lng t_{\gamma}) \cdot \pi_{\bom}$. 
We write $\Theta^{-1}(\bom)$ as 
$\Theta^{-1}(\bom) = (\bsg,\bchi,\xi) \in \Par(\lambda,\mu)$; 
see \eqref{eq:omg1} and \eqref{eq:omg2}. 
By Lemma~\ref{lem:Weyl}\,(2), 
\begin{equation*}
\pi \otimes \eta = \Phi_{\lambda\mu}(\psi) = \Phi_{\lambda\mu}((\lng t_{\gamma}) \cdot \pi_{\bom}) = 
((\lng t_{\xi + \gamma}) \cdot \pi_{\bsg}) \otimes 
((\lng t_{\gamma}) \cdot \pi_{\bchi}). 
\end{equation*}
Write $y$ as $y = v t_{\zeta}$, 
with $v \in W$ and $\zeta \in Q^{\vee}$. 
By Lemma~\ref{lem:Weyl}\,(1), Corollary~\ref{cor:mins}, 
and Remark~\ref{rem:tbmin}, we deduce that 
\begin{equation*}
\io{\psi}{y} = \tbmin{\lng}{\Jlm}{v} \cdot t_{ [\ip{\bom}+\gamma]^{\Jlm}+[\zeta]_{\Jlm} }; 
\end{equation*}
for $\ip{\bom} \in Q^{\vee}$, see Remark~\ref{rem:init}. 
Similarly, we deduce that 
\begin{equation*}
\io{\eta}{y} = \underbrace{\tbmin{\lng}{\Jm}{v}}_{=:v'} \cdot 
t_{[\ip{\bchi}+\gamma]^{\Jm}+[\zeta]_{\Jm}}, 
\end{equation*}
and then that
\begin{equation*}
\io{\pi}{\io{\eta}{y}} = 
\tbmin{\lng}{\Jl}{v'} \cdot 
t_{ [\xi + \gamma + \ip{\bsg}]^{\Jl} + [[\ip{\bchi}+\gamma]^{\Jm}+[\zeta]_{\Jm}]_{\Jl} }.
\end{equation*}
It follows from Lemma~\ref{lem:n0} that 
$\tbmin{\lng}{\Jlm}{v} = \tbmin{\lng}{\Jl}{v'}$. 
Also, we see by definitions \eqref{eq:omg1} and \eqref{eq:omg2} that 
\begin{equation*}
[\ip{\bom}+\gamma]^{\Jlm}+[\zeta]_{\Jlm} = 
[\xi + \gamma + \ip{\bsg}]^{\Jl} + [[\ip{\bchi}+\gamma]^{\Jm}+[\zeta]_{\Jm}]_{\Jl}. 
\end{equation*}
From these, we obtain $\io{\psi}{y} = \io{\pi}{\io{\eta}{y}}$ 
in the case $n=0$, as desired. 

Assume now that $n > 0$; for simplicity of notation, 
we set $\psi':=f_{i_{n}}\psi$ and $i:=i_{n}$. 
Also, write $\Phi_{\lambda\mu}(\psi') = f_{i}(\pi \otimes \eta) 
\in \SLS(\lambda) \otimes \SLS(\mu)$ as
$\Phi_{\lambda\mu}(\psi') = \pi' \otimes \eta'$, 
with $\pi' \in \SLS(\lambda)$ and $\eta' \in \SLS(\mu)$. 
By our induction hypothesis, we have
$\io{\psi'}{y} = \io{\pi'}{\io{\eta'}{y}}$. 
If $\ve_{i}(\psi) = \ve_{i}(\pi \otimes \eta) \ge 1$, 
then it follows from Lemmas~\ref{lem:iota1} and \ref{lem:iota2}, 
together with our induction hypothesis, that 
$\io{\psi}{y}=\io{\psi'}{y}=
\io{\pi'}{\io{\eta'}{y}} = \io{\pi}{\io{\eta}{y}}$. 
Assume that 
$\ve_{i}(\psi) = \ve_{i}(\pi \otimes \eta) = 0$. 
Then we see again by Lemmas~\ref{lem:iota1} and \ref{lem:iota2}, 
together with our induction hypothesis, that 
$\io{\psi}{y}$ is identical to $\io{\pi}{\io{\eta}{y}}$ or 
$s_{i}\io{\pi}{\io{\eta}{y}}$, and that 
$\io{\psi}{y}^{-1}\alpha_{i} \in \Delta^{+}$ and 
$\io{\pi}{\io{\eta}{y}}^{-1}\alpha_{i} \in \Delta^{+}$; 
if $\io{\psi}{y} = s_{i}\io{\pi}{\io{\eta}{y}}$, 
then either $\io{\psi}{y}^{-1}\alpha_{i}$ or 
$\io{\pi}{\io{\eta}{y}}^{-1}\alpha_{i}$ is a negative root, 
which is a contradiction. 
Therefore, we obtain $\io{\psi}{y}=\io{\pi}{\io{\eta}{y}}$, as desired. 
This completes the proof of Proposition~\ref{prop:smt}. 
\end{proof}
%
%
\begin{cor}[{cf. \cite[Lemma~3.12]{NNS3}}] \label{cor:smt}
Let $\psi \in \SLS(\lambda+\mu)$, 
and write $\Phi_{\lambda\mu}(\psi) \in \SLS(\lambda) \otimes \SLS(\mu)$ as
$\Phi_{\lambda\mu}(\psi) = \pi \otimes \eta$, 
with $\pi \in \SLS(\lambda)$ and $\eta \in \SLS(\mu)$. Then, 
\begin{equation}
\iota(\pi) = \PS{\Jl}(\iota(\psi)) \quad \text{and} \quad 
\kappa(\eta) = \PS{\Jm}(\kappa(\psi)).
\end{equation}
\end{cor}

\begin{proof}
We first remark that $\Jlm \subset \Jl$ and 
$\Jlm \subset \Jm$, which implies that 
$\PS{\Jl} \circ \PS{\Jlm} = \PS{\Jl}$ and 
$\PS{\Jm} \circ \PS{\Jlm} = \PS{\Jm}$, respectively. 
We take $y \in W_{\af}$ such that 
$\psi \in \SLS_{\sige y}(\lambda+\mu)$. 
By the definitions, we have 
$\PS{\Jlm}(\io{\psi}{y}) = \iota(\psi)$ and 
$\PS{\Jl}(\io{\pi}{\io{\eta}{y}}) = \iota(\pi)$.
Since $\io{\psi}{y} = \io{\pi}{\io{\eta}{y}}$ 
by Proposition~\ref{prop:smt}\,(1), we see that 
\begin{align*}
\PS{\Jl}\bigl( \iota(\psi) \bigr) & = 
\PS{\Jl}\bigl( \PS{\Jlm}(\io{\psi}{y}) \bigr) = 
\PS{\Jl}\bigl( \PS{\Jlm}(\io{\pi}{\io{\eta}{y}}) \bigr) \\ 
& = \PS{\Jl}\bigl( \io{\pi}{\io{\eta}{y}} \bigr) = \iota(\pi), 
\end{align*}
as desired. Similarly, we can prove that 
$\kappa(\eta) = \PS{\Jm}(\kappa(\psi))$. 
This proves the corollary. 
\end{proof}
%
%
\section{Proof of Theorem~\ref{thm:main}: part 1.}
\label{sec:prf1}
%
%
\subsection{Formula for graded characters.}
\label{subsec:lem1}

Let us take an arbitrary $r \in I$; note that $\Jr = I \setminus \{r\}$ (see \eqref{eq:J}). 
It is easily checked that 
the map $\WSa{\Jr} \rightarrow W_{\af}\vpi_r$, $x \mapsto x\vpi_r$, is bijective. 
The next lemma follows from \cite[Proposition~4.2.1]{INS} and \cite[Lemma~2.1.5]{NS06}. 
%
%
\begin{lem} \label{lem:vpi}
If $\pi \in \SLS(\vpi_{r})$ satisfies $\kappa(\pi) = \PS{\Jr}(t_{\xi})$ 
for some $\xi \in Q^{\vee}$, then $\pi=(\PS{\Jr}(t_{\xi})\,;\,0,\,1)$. 
\end{lem}

Let us fix $\lambda \in P^{+}$. Recall from \eqref{eq:Phi} 
the following isomorphism of crystals: 
\begin{equation*}
\Phi_{\lambda\vpi_{r}}: 
\SLS(\lambda+\vpi_{r}) \stackrel{\sim}{\rightarrow} 
\SM(\lambda+\vpi_{r}) \quad (\hookrightarrow 
\SLS(\lambda) \otimes \SLS(\vpi_{r})).
\end{equation*}
%
%
\begin{lem} \label{lem:img}
Let $\xi \in Q^{\vee}$, and set
\begin{equation} \label{eq:img1}
\BB:=\bigl\{ \psi \in \SLS_{\sige t_{\xi}}(\lambda+\vpi_{r}) \mid 
\PS{\Jr}(\kappa(\psi)) = \PS{\Jr}(t_{\xi}) \bigr\}. 
\end{equation}
Then, 
\begin{equation} \label{eq:img2}
\Phi_{\lambda\vpi_r}(\BB) = 
\SLS_{\sige t_{\xi}}(\lambda) \otimes 
\bigl\{ (\PS{\Jr}(t_{\xi})\,;\,0,\,1) \bigr\}. 
\end{equation}
\end{lem}
\begin{proof}
First, we prove the inclusion $\subset$ in \eqref{eq:img2}. 
Let $\psi \in \BB$, and write $\Phi_{\lambda\vpi_r}(\psi)$ as 
$\Phi_{\lambda\vpi_r}(\psi) = \pi \otimes \eta$, 
with $\pi \in \SLS(\lambda)$ and $\eta \in \SLS(\vpi_r)$. 
Then we see from Corollary~\ref{cor:smt} that 
$\kappa(\eta) = \PS{\Jr}(\kappa(\psi)) = \PS{\Jr}(t_{\xi})$. 
Hence it follows from Lemma~\ref{lem:vpi} that 
$\eta = (\PS{\Jr}(t_{\xi})\,;\,0,\,1)$. 
Also, since $\psi \in \SLS_{\sige t_{\xi}}(\lambda+\vpi_{r})$, 
we see from Theorem~\ref{thm:SMT}\,(1) that 
$\kappa(\pi) \sige \PS{\Jl}(\io{\eta}{t_{\xi}})$; 
it is obvious that $\io{\eta}{t_{\xi}} = t_{\xi}$ 
since $\eta = (\PS{\Jr}(t_{\xi})\,;\,0,\,1)$ as shown above. 
Hence we obtain $\pi \in \SLS_{\sige t_{\xi}}(\lambda)$. 
Thus we have proved the inclusion $\subset$. 

Next, we prove the opposite inclusion $\supset$ in \eqref{eq:img2}. 
Let $\pi \in \SLS_{\sige t_{\xi}}(\lambda)$, and 
set $\eta = (\PS{\Jr}(t_{\xi})\,;\,0,\,1)$; recall that 
$\io{\eta}{t_{\xi}} = t_{\xi}$. 
Therefore, by Theorem~\ref{thm:SMT}\,(1), we have 
$\pi \otimes \eta \in \SM_{\sige t_{\xi}}(\lambda+\vpi_r) = 
\Phi_{\lambda\vpi_r}(\SLS_{\sige t_{\xi}}(\lambda+\vpi_r))$. 
Let $\psi$ be the unique element of 
$\SLS_{\sige t_{\xi}}(\lambda+\vpi_r)$ such that 
$\pi \otimes \eta = \Phi_{\lambda\vpi_r}(\psi)$.
Then, by Corollary~\ref{cor:smt}, 
$\PS{\Jr}(\kappa(\psi)) = \kappa(\eta) = \PS{\Jr}(t_{\xi})$, 
which implies that $\psi \in \BB$, as desired. 
This proves the lemma. 
\end{proof}
%
%
\begin{prop} \label{prop:s1}
Let $\lambda \in P^{+}$ and $r \in I$. 
For $\xi \in Q^{\vee}$, the following equality holds\,{\rm:}
\begin{equation}
\gch V_{t_{\xi}}^{-}(\lambda+\vpi_{r}) = 
 \be^{t_{\xi}\vpi_{r}}\gch V_{t_{\xi}}^{-}(\lambda) + 
 \gch V_{s_{r}t_{\xi}}^{-} (\lambda+\vpi_{r}).
\end{equation}
\end{prop}

\begin{proof}
Define $\BB$ as in \eqref{eq:img1}. 
By Lemma~\ref{lem:img} and \eqref{eq:gch1}, 
it suffices to prove that 
\begin{equation} \label{eq:s1-1}
\SLS_{\sige t_{\xi}}(\lambda+\vpi_{r}) = 
\BB \sqcup \SLS_{\sige s_{r}t_{\xi}}(\lambda+\vpi_{r}).
\end{equation}
First, suppose, for a contradiction, that 
$\BB \cap \SLS_{\sige s_{r}t_{\xi}}(\lambda+\vpi_{r}) \ne \emptyset$. 
Let $\psi \in \BB \cap \SLS_{\sige s_{r}t_{\xi}}(\lambda+\vpi_{r})$. 
Since $\psi \in \SLS_{\sige s_{r}t_{\xi}}(\lambda+\vpi_{r})$, 
we have $\kappa(\psi) \sige \PS{\J_{\lambda+\vpi_r}}(s_{r}t_{\xi})$. 
Since $\psi \in \BB$, we have 
$\PS{\Jr}(\kappa(\psi)) = \PS{\Jr}(t_{\xi})$. 
Also, it is checked by Lemma~\ref{lem:PiJ}, 
together with $\J_{\lambda+\vpi_r} \subset \Jr$, that 
$\PS{\Jr}(\PS{\J_{\lambda+\vpi_r}}(s_{r}t_{\xi})) = s_{r}\PS{\Jr}(t_{\xi})$. 
Therefore, by Lemma~\ref{lem:611}, we see that 
\begin{equation*}
\PS{\Jr}(t_{\xi}) = \PS{\Jr}(\kappa(\psi)) \sige 
\PS{\Jr}(\PS{\J_{\lambda+\vpi_r}}(s_{r}t_{\xi})) = s_{r}\PS{\Jr}(t_{\xi}).
\end{equation*}
However, since $\pair{\PS{\Jr}(t_{\xi})\vpi_r}{\alpha_{r}^{\vee}} = 
\pair{t_{\xi}\vpi_r}{\alpha_{r}^{\vee}} = 1 > 0$, 
it follows from Lemma~\ref{lem:si} that 
$s_{r}\PS{\Jr}(t_{\xi}) \sig \PS{\Jr}(t_{\xi})$, which is a contradiction. 
Thus we obtain 
$\BB \cap \SLS_{\sige s_{r}t_{\xi}}(\lambda+\vpi_{r}) = \emptyset$. 

Next we prove the equality in \eqref{eq:s1-1}. 
The inclusion $\supset$ is obvious; note that $s_{r}t_{\xi} \sig t_{\xi}$. 
Let us prove the opposite inclusion $\subset$. 
Let $\psi \in \SLS_{\sige t_{\xi}}(\lambda+\vpi_{r})$, and 
assume that $\psi \notin \BB$ to conclude that 
$\psi \in \SLS_{\sige s_{r}t_{\xi}}(\lambda+\vpi_{r})$. 
We write $\kappa(\psi) \in \WSa{\J_{\lambda+\vpi_r}}$ as 
$\kappa(\psi)=w \PS{\J_{\lambda+\vpi_r}}(t_{\gamma})$ 
with $w \in W^{\J_{\lambda+\vpi_r}}$ and $\gamma \in Q^{\vee}$; 
note that $[\gamma]^{\J_{\lambda+\vpi_r}} \ge [\xi]^{\J_{\lambda+\vpi_r}}$ 
by Lemma~\ref{lem:SiB}\,(1). It is easily checked by Lemma~\ref{lem:PiJ} that 
$\PS{\Jr}(\kappa(\psi))=\PS{\Jr}(w \PS{\J_{\lambda+\vpi_r}}(t_{\gamma})) = 
\mcr{w}^{\Jr}\PS{\Jr}(t_{\gamma})$; since $\psi \notin \BB$, 
it follows that $\mcr{w}^{\Jr} \ne e$ or $\gamma - \xi \notin \QSv{\Jr}$. 
If $\mcr{w}^{\Jr} \ne e$, then a reduced expression for $w$ contains $s_{r}$, 
which implies that $w \sige s_{r}$ 
(note that $s_{r} \in \WSu{\J_{\lambda+\vpi_r}}$). 
Therefore, by Lemma~\ref{lem:SiB}\,(3), 
$\kappa(\psi)=w \PS{\J_{\lambda+\vpi_r}}(t_{\gamma}) \sige 
s_{r}\PS{\J_{\lambda+\vpi_r}}(t_{\gamma})$. 
Also, since $[\gamma]^{\J_{\lambda+\vpi_r}} \ge [\xi]^{\J_{\lambda+\vpi_r}}$ 
as seen above, we have 
$s_{r}\PS{\J_{\lambda+\vpi_r}}(t_{\gamma}) \sige 
 s_{r}\PS{\J_{\lambda+\vpi_r}}(t_{\xi})$ by Lemma~\ref{lem:SiB}\,(2); 
note that $s_{r}\PS{\J_{\lambda+\vpi_r}}(t_{\xi}) = 
 \PS{\J_{\lambda+\vpi_r}}(s_{r}t_{\xi})$ by Lemma~\ref{lem:PiJ}. 
Combining these inequalities, 
we obtain $\kappa(\psi) \sige \PS{\J_{\lambda+\vpi_r}}(s_{r}t_{\xi})$, 
and hence $\psi \in \SLS_{\sige s_{r}t_{\xi}}(\lambda+\vpi_{r})$. 
Assume now that $\gamma - \xi \notin \QSv{\Jr}$, or equivalently, 
$\pair{\vpi_{r}}{\gamma - \xi} \ne 0$. 
Since $[\gamma]^{\J_{\lambda+\vpi_r}} \ge [\xi]^{\J_{\lambda+\vpi_r}}$ as seen above, 
and since $r \in I \setminus \J_{\lambda+\vpi_r}$, 
it follows that $\pair{\vpi_{r}}{\gamma - \xi} > 0$, 
which implies that $[\gamma]^{\J_{\lambda+\vpi_r}} \ge 
[\xi+ \alpha_{r}^{\vee}]^{\J_{\lambda+\vpi_r}}$. 
Therefore, by Lemma~\ref{lem:SiB}\,(2), 
$\kappa(\psi) 
   = w \PS{\J_{\lambda+\vpi_r}}(t_{\gamma}) 
   \sige w \PS{\J_{\lambda+\vpi_r}}(t_{\xi + \alpha_{r}^{\vee}})$. 
Also, since $w \sige e$, it follows by Lemma~\ref{lem:SiB}\,(3) that 
$w \PS{\J_{\lambda+\vpi_r}}(t_{\xi + \alpha_{r}^{\vee}}) \sige 
   \PS{\J_{\lambda+\vpi_r}}(t_{\xi + \alpha_{r}^{\vee}})$. 
Finally, notice that $s_{r}t_{\xi} \sil t_{\xi+\alpha_{r}^{\vee}}$ 
since $s_{r}t_{\xi} \edge{-\alpha_r+\delta} t_{\xi+\alpha_{r}^{\vee}}$ in $\SB$. 
It follows from Lemma~\ref{lem:611} that 
$\PS{\J_{\lambda+\vpi_r}}(t_{\xi + \alpha_{r}^{\vee}}) \sige 
 \PS{\J_{\lambda+\vpi_r}}(s_{r}t_{\xi})$. 
Combining these inequalities, 
we obtain $\kappa(\psi) \sige \PS{\J_{\lambda+\vpi_r}}(s_{r}t_{\xi})$, 
and hence $\psi \in \SLS_{\sige s_{r}t_{\xi}}(\lambda+\vpi_{r})$. 
This completes the proof of the proposition. 
\end{proof}
%
%
\subsection{First step in the proof of Theorem~\ref{thm:main}.}
\label{subsec:s1}

We prove \eqref{eq:main} in the case that 
$\mu=\vpi_{r}$ for $r \in I$, and 
$x = t_{\xi}$ for $\xi \in Q^{\vee}$.  
In this case, the right-hand side of \eqref{eq:main} 
is identical to 
%
%
\begin{equation} \label{eq:c1a}
 \sum_{ \begin{subarray}{c}
   y \in W_{\af} \\[1mm]
   y \sige t_{\xi} \end{subarray}}\,
 \sum_{ \begin{subarray}{c}
  \pi \in \SLS(\vpi_{r}) \\[1mm] 
  \iota(\pi) \sile \PS{\Jr}(y),\, \kap{\pi}{y}=t_{\xi}
  \end{subarray}}  (-1)^{\sell(y)}
  \be^{-\wt(\pi)} \gch V_{y}^{-}(\lambda).
\end{equation}
Let $y \in W_{\af}$ be such that $y \sige t_{\xi}$. 
If $\pi \in \SLS(\vpi_{r})$ satisfies $\kap{\pi}{y}=t_{\xi} \in W_{\af}$, 
then we see that $\kappa(\pi)=\PS{\Jr}(t_{\xi})$. 
Hence it follows from Lemma~\ref{lem:vpi}
that $\pi = (\PS{\Jr}(t_{\xi}) \,;\,0,1)$.
Therefore, for $y \in W_{\af}$ such that $y \sige t_{\xi}$, 
\begin{equation} \label{eq:s1a}
\underbrace{%
\bigl\{ \pi \in \SLS(\vpi_{r}) \mid 
  \iota(\pi) \sile \PS{\Jr}(y),\,\kap{\pi}{y}=t_{\xi} \bigr\} = 
\bigl\{ (\PS{\Jr}(t_{\xi})\,;\,0,1) \bigr\}}_{%
\text{this equality holds iff $\max \Lile{y}{\PS{\Jr}(t_{\xi})} = t_{\xi}$}%
}
  \quad \text{or} \quad \emptyset. 
\end{equation}
%
%
\begin{lem} \label{lem:e}
Let $y \in W_{\af}$ be such that $y \sige t_{\xi}$, and 
write it as $y = vt_{\zeta}$ 
with $v \in W$ and $\zeta \in Q^{\vee,+}$. Then, 
\begin{equation*}
\max \Lile{y}{\PS{\Jr}(t_{\xi})} = t_{\xi} \quad \iff \quad 
 v \in \{e,\,s_{r}\} \ \text{\rm and} \ \zeta - \xi \in \BZ_{\ge 0}\alpha_{r}^{\vee}.
\end{equation*}
\end{lem}

\begin{proof}
We write $\max \Lile{y}{\PS{\Jr}(t_{\xi})}$ as: 
$\max \Lile{y}{\PS{\Jr}(t_{\xi})} = wt_{\gamma}$ 
with $w \in W$ and $\gamma \in Q^{\vee}$. 
We see from Proposition~\ref{prop:maxs} that
\begin{equation*}
w = \tbmax{e}{\Jr}{v}, \qquad 
\gamma=[\xi]^{\Jr} + [\zeta - \wt(w \Rightarrow v)]_{\Jr}.
\end{equation*}

We first prove the implication $\Rightarrow$. Let us assume that $w = e$ and $\gamma=\xi$. 
Because $v$ is greater than or equal to $e$ in the (ordinary) Bruhat order on $W$, 
there exists a shortest directed path from $e$ to $v$ in $\QB$ 
whose directed edges are all Bruhat edges. Therefore, we have
$\wt(w \Rightarrow v) = \wt(e \Rightarrow v) = 0$. 
Hence we obtain $\xi = [\xi]^{\Jr} + [\zeta]_{\Jr}$, which implies that 
$\zeta - \xi \in \BZ_{\ge 0} \alpha_{r}^{\vee}$. 
Also, suppose, for a contradiction, that 
$s_{j}$ appears in a reduced expression for $v$ 
for some $j \in \Jr = I \setminus \{r\}$. 
Then, $v$ is greater than or equal to $s_{j}$ in the Bruhat order on $W$, 
Therefore, there exists a shortest directed path 
from $e$ to $v$ in $\QB$ passing through $s_{j}$ 
(whose directed edges are all Bruhat edges); 
in particular, $\ell(e \Rightarrow v) > \ell(s_j \Rightarrow v)$. 
However, since $\tbmax{e}{\Jr}{v} = e$ by our assumption, and since 
$s_{j} \in e\WS{\Jr}$, it follows from the definition of $\dtb{v}$ 
that $\ell(s_{j} \Rightarrow v) = \ell(s_{j} \Rightarrow e) + \ell(e \Rightarrow v)$;
in particular, $\ell(s_{j} \Rightarrow v) < \ell(e \Rightarrow v)$, 
which is a contradiction. 
Hence we conclude that $v \in \{e, s_{r}\}$. 

We next prove the implication $\Leftarrow$.
Assume that $v = e$. Then it is obvious that $w = e$, 
and hence $\wt(w \Rightarrow v) = \wt(e \Rightarrow e) = 0$. 
Hence we obtain $\gamma = [\xi]^{\Jr} + [\zeta]_{\Jr}$. 
Since $\zeta - \xi \in \BZ_{\ge 0}\alpha_{r}^{\vee}$ by the assumption, 
we deduce that $\gamma = \xi$. 
Assume now that $v = s_{r}$. By the definition of $\dtb{v}$, 
we have $\ell(e \Rightarrow s_{r}) = 
\ell(e \Rightarrow w) + \ell(w \Rightarrow s_{r})$. 
Since $\ell(e \Rightarrow s_{r}) = 1$, we see that 
$\ell(w \Rightarrow s_{r}) = 0$ or $1$. 
Since $w \in e\WS{\Jr}$ and $s_{r} \notin e\WS{\Jr}$, 
it follows that $w \ne s_{r}$, which implies that 
$\ell(w \Rightarrow s_{r}) = 1$. 
Therefore, we obtain $\ell(e \Rightarrow w) = 0$, i.e., $w = e$, 
and hence the same argument as above shows that 
$\gamma = \xi$. This proves the lemma. 
\end{proof}

By Lemma~\ref{lem:e} and \eqref{eq:s1a}, 
we see that the right-hand side of \eqref{eq:c1a} 
(and hence that of \eqref{eq:main}) is identical to: 
\begin{align}
& \sum_{m \in \BZ_{\ge 0}}
  \sum_{v \in \{e,\,s_{r} \}}
  (-1)^{\sell(vt_{\xi+m\alpha_{r}^{\vee}})}
  \be^{-t_{\xi}\vpi_{r}} \gch V_{vt_{\xi+m\alpha_{r}^{\vee}}}^{-}(\lambda) \nonumber \\[3mm]
& \hspace*{10mm} = 
  \sum_{m \in \BZ_{\ge 0}} 
  q^{-m\lambda_{r}}\be^{-t_{\xi}\vpi_{r}}
  \bigl( \gch V_{ t_{\xi} }^{-}(\lambda) -
         \gch V_{ s_{r}t_{\xi} }^{-}(\lambda) \bigr)
  \quad \text{by Proposition~\ref{prop:gch-tx}}. \label{eq:e3}
\end{align}
By Proposition~\ref{prop:s1}, we have 
$\gch V_{ t_{\xi} }^{-}(\lambda) -
         \gch V_{ s_{r}t_{\xi} }^{-}(\lambda) = \be^{t_{\xi}\vpi_{r}} 
 \gch V_{ t_{\xi} }^{-}(\lambda-\vpi_{r})$. 
Substituting this equality into \eqref{eq:e3}, 
we deduce that the right-hand side of \eqref{eq:e3} 
is identical to: 
\begin{equation*}
\sum_{m \in \BZ_{\ge 0}} 
  q^{-m\lambda_{r}} \gch V_{ t_{\xi} }^{-}(\lambda-\vpi_{r}) = 
  \frac{1}{1-q^{-\lambda_{r}}} \gch V_{ t_{\xi} }^{-}(\lambda-\vpi_{r}). 
\end{equation*}
This proves Theorem~\ref{thm:main} 
in the case that $\mu = \vpi_{r}$ and $x=t_{\xi}$. 
%
%
\section{String property and Demazure operators.}
\label{sec:dem}
%
%
\subsection{Recursion formula for graded characters in terms of Demazure operators.}
\label{subsec:rec}

For each $i \in I_{\af}$, 
we define a $\BC(q)$-linear operator $\sD_{i}$ on $\BC(q)[P]$ by
\begin{equation*}
\sD_{i} \be^{\nu} := 
\frac{ \be^{\nu - \rho} - \be^{s_{i}(\nu - \rho)} }{ 1-\be^{\alpha_{i}} } \be^{\rho} = 
\frac{ \be^{\nu}-\be^{\alpha_{i}}\be^{s_{i}\nu} }{ 1-\be^{\alpha_{i}} } \qquad 
\text{for $\nu \in P$, where $q=\be^{\delta}$}; 
\end{equation*}
note that $\sD_{i}^{2}=\sD_{i}$, and that
\begin{equation} \label{eq:Di}
\sD_{i} \be^{\nu} = 
\begin{cases}
\be^{\nu}(1+\be^{\alpha_{i}}+\be^{2\alpha_{i}}+ \cdots +\be^{-\pair{\nu}{\alpha_{i}^{\vee}}\alpha_{i}}) 
  & \text{if $\pair{\nu}{\alpha_{i}^{\vee}} \le 0$}, \\[1.5mm]
0 & \text{if $\pair{\nu}{\alpha_{i}^{\vee}}=1$}, \\[1.5mm]
-\be^{\nu}(\be^{- \alpha_{i}}+ \be^{- 2\alpha_{i}}+\cdots +\be^{(-\pair{\nu}{\alpha_{i}^{\vee}}+1)\alpha_{i}})
  & \text{if $\pair{\nu}{\alpha_{i}^{\vee}} \ge 2$}.
\end{cases}
\end{equation}
Also, we define a $\BC(q)$-linear operator $\sT_{i}$ on $\BC(q)[P]$ by
\begin{equation*}
\sT_{i}:=\sD_{i}-1, \quad \text{that is}, \quad 
\sT_{i}\be^{\nu}=\frac{\be^{\alpha_i}(\be^{\nu}-\be^{s_i\nu})}{1-\be^{\alpha_i}};
\end{equation*}
note that $\sT_{i}^{2}=-\sT_{i}$ and $\sT_{i}\sD_{i}=\sD_{i}\sT_{i}=0$.
We can easily verify the following lemma. 
%
%
\begin{lem}[Leibniz rule] \label{lem:lei}
For $\nu_1,\,\nu_2 \in P$ and $i \in I_{\af}$, it holds that 
%
%
\begin{equation} \label{eq:L1}
\sD_{i}(\be^{\nu_1}\be^{\nu_2}) = (\sD_{i} \be^{\nu_1+\rho})\be^{\nu_2-\rho} + 
\be^{s_{i}\nu_1} (\sD_{i}\be^{\nu_2}), 
\end{equation}
%
%
\begin{equation} \label{eq:L2}
\sT_{i}(\be^{\nu_1}\be^{\nu_2}) = (\sT_{i}\be^{\nu_1})\be^{\nu_2} + 
\be^{s_{i}\nu_1} (\sT_{i}\be^{\nu_2}).
\end{equation}
\end{lem}

Now, we take an arbitrary $\mu \in P^{+}$, and define $\Jm \subset I$ as in \eqref{eq:J}. 
Fix $i \in I_{\af}$. A subset $\sS \subset \SLS(\mu)$ is called an $i$-string 
if $\sS$ is of the form 
$\sS=\bigl\{\pi,f_{i}\pi,\,\dots,\,f_{i}^{\vp_{i}(\pi)-1}\pi,\,
  f_{i}^{\vp_{i}(\pi)}\pi \bigr\}$
for some $\pi \in \SLS(\mu)$ such that $e_{i}\pi = \bzero$; 
in this case, we call $\pi_{\RH}=\pi_{\RH}^{\sS}:=\pi$ and 
$\pi_{\RL}=\pi_{\RL}^{\sS}:=f_{i}^{\vp_{i}(\pi)}\pi$ 
the $i$-highest element and the $i$-lowest element in $\sS$, respectively.
Note that $\pi_{\RH}=\pi_{\RL}$ if and only if $\#\sS = 1$. 
\begin{rem} \label{rem:str}
The crystal $\SLS(\mu)$ decomposes into a disjoint union of 
(infinitely many) $i$-strings for each $i \in I_{\af}$. 
\end{rem}
The next lemma follows from the definition of the root operator $f_{i}$ 
(see also \cite[Lemma~2.4.5 and its proof]{NS16}; cf. \cite[5.3 Lemma]{Lit94}). 
%
%
\begin{lem} \label{lem:str0}
Keep the notation and setting above. 
We set $w:=\kappa(\pi_{\RH}) \in \WSa{\Jm}$ and $N:=\vp_{i}(\pi_{\RH}) \in \BZ_{\ge 0}${\rm;}
note that $f_{i}^{N}\pi_{\RH}=\pi_{\RL}$. 
\begin{enu}
\item If $N \ge 1$, then 
$\kappa(f_{i}^{k}\pi_{\RH})=w$ for all $0 \le k \le N-1$. 

\item If $\pair{w\mu}{\alpha_{i}^{\vee}} \le 0$, then 
$\kappa(f_{i}^{k}\pi_{\RH}) = w$ for all $0 \le k \le N$. 

\item If $\pair{w\mu}{\alpha_{i}^{\vee}} > 0$, then 
$N \ge 1$, and $\kappa(f_{i}^{N}\pi_{\RH}) = s_{i}w \sig w$.
\end{enu}
\end{lem}
The following corollary is 
an immediate consequence of Lemma~\ref{lem:str0}. 
%
%
\begin{cor}[{cf. \cite[5.4 Lemma]{Lit94}}] \label{cor:str0}
Let $i \in I_{\af}$ and $y \in W_{\af}$. 
For each $i$-string $\sS \subset \SLS(\mu)$, 
the intersection $\SLS_{\sige y}(\mu) \cap \sS$ is identical to 
$\emptyset$, $\sS$, or $\{ \pi_{\RL}^{\sS} \}$. 
\end{cor}

The proof of the next lemma is straightforward. 
%
%
\begin{lem} \label{lem:dem}
Let $i \in I_{\af}$, and let $\sS \subset \SLS(\mu)$ be an $i$-string, 
with $\pi_{\RH}=\pi_{\RH}^{\sS}$ and $\pi_{\RL}=\pi_{\RL}^{\sS}$ 
the $i$-highest element and the $i$-lowest element in $\sS$, respectively. Then, 
\begin{equation} \label{eq:sdem1}
\sD_{i} \left(\sum_{\pi \in \sS} \be^{\wt(\pi)}\right) 
  = \sum_{\pi \in \sS} \be^{\wt(\pi)}, \qquad 
\sT_{i} \left(\sum_{\pi \in \sS} \be^{\wt(\pi)}\right) = 0, 
\end{equation}
\begin{equation} \label{eq:sdem2}
\sD_{i} \be^{\wt(\pi_{\RL})} = 
 \sum_{\pi \in \sS} \be^{\wt(\pi)}, \qquad
\sT_{i} \be^{\wt(\pi_{\RL})} = 
 \sum_{\pi \in \sS \setminus \{\pi_{\RL}\}} \be^{\wt(\pi)}, 
\end{equation}
\begin{equation} \label{eq:sdem3}
\sD_{i} \be^{\wt(\pi_{\RH})} = -
 \sum_{\pi \in \sS \setminus \{\pi_{\RH},\pi_{\RL}\}} \be^{\wt(\pi)}, \qquad
\sT_{i} \be^{\wt(\pi_{\RH})} = - 
 \sum_{\pi \in \sS \setminus \{\pi_{\RL}\}} \be^{\wt(\pi)}.
\end{equation}
\end{lem}

Let $y \in W_{\af}$ and $i \in I_{\af}$ be such that $s_{i}y \sil y$, 
or equivalently, $y^{-1}\alpha_{i} \in \Delta^{-} + \BZ \delta$ (see Lemma~\ref{lem:si}); 
note that $\pair{y\mu}{\alpha_{i}^{\vee}} \le 0$, and 
that $\pair{y\mu}{\alpha_{i}^{\vee}} = 0$  if and only if 
$\PS{\Jm}(s_{i}y)=\PS{\Jm}(y)$.
It follows from \cite[Proposition~5.3.2]{NS16} that
the set $\SLS_{\sige s_{i}y}(\mu) \cup \{\bzero\}$ 
is stable under the action of the root operator $e_{i}$, and that
\begin{equation} \label{eq:532}
\SLS_{\sige s_{i}y}(\mu) = 
 \bigl\{ e_{i}^{k}\pi \mid \pi \in \SLS_{\sige y}(\mu), \, 
 0 \le k \le \ve_{i}(\pi) \bigr\}. 
\end{equation}
Using Corollary~\ref{cor:str0} and \eqref{eq:532}, \eqref{eq:sdem2}, 
together with Remark~\ref{rem:str}, 
we can easily show the following proposition (see also 
\cite[Sect.~5.5]{Lit94} and \cite[Proposition~9.2.3]{KasF}). 
%
%
\begin{prop} \label{prop:dem}
Let $y \in W_{\af}$ and $i \in I_{\af}$ be such that $s_{i}y \sil y$. 
Then, 
%
%
\begin{equation} \label{eq:dem1}
\gch V_{s_{i}y}^{-}(\mu) = \sD_{i} \gch V_{y}^{-}(\mu).
\end{equation}
\end{prop}
%
%
\begin{rem} \label{rem:dem-op}
Since $\sD_{i}^{2}=\sD_{i}$, we deduce that 
for $y \in W_{\af}$ such that $s_{i}y \sig y$, 
%
%
\begin{equation} \label{eq:dem2}
\sD_{i} \gch V_{y}^{-}(\mu) = \gch V_{y}^{-}(\mu).
\end{equation}
Also, by \eqref{eq:dem1} and \eqref{eq:dem2}, we have
%
%
\begin{equation} \label{eq:dem3}
\sT_{i} \gch V_{y}^{-}(\mu) = 
\begin{cases}
\gch V_{s_{i}y}^{-}(\mu) - \gch V_{y}^{-}(\mu) 
 & \text{if $s_{i}y \sil y$}, \\[2mm]
0 & \text{if $s_{i}y \sig y$}. 
\end{cases}
\end{equation}
\end{rem}
%
%
\subsection{String property for Demazure-like subsets of $\SLS(\mu)$.}
\label{subsec:str}

We take an arbitrary $\mu \in P^{+}$, and define 
$\Jm \subset I$ as in \eqref{eq:J}. For $u,v \in W_{\af}$, we set
\begin{equation}
\sls{u}{v}{\mu}:=
\bigl\{ \pi \in \SLS(\mu) \mid 
\kappa(\pi) \sige \PS{\Jm}(v),\,\io{\pi}{v}=u\bigr\}. 
\end{equation}
Fix $i \in I_{\af}$. For an $i$-string $\sS \subset \SLS(\mu)$, 
we set
\begin{equation}
\sS_{u,v}:=\sls{u}{v}{\mu} \cap \sS.
\end{equation}
%
%
\begin{prop} \label{prop:string}
Keep the notation and setting above. 
In addition, assume that $s_{i}u \sig u$ and $s_{i}v \sig v$. 
Then we have the following table\,{\rm:}
\begin{equation} \label{eq:table}
{\renewcommand{\arraystretch}{1.5}
\begin{array}{|c||c||c|c||c|c||c|} \hline
&
& \multicolumn{2}{c||}{\text{\rm Case 1}} 
& \multicolumn{2}{c||}{\text{\rm Case 2}} 
& \text{\rm Case 3} \\ \hline\hline
& \#\sS & \multicolumn{2}{c||}{\ge 2} & \multicolumn{2}{c||}{\ge 1} & \ge 1 \\ \hline
\text{\rm (i)} & \sS_{u,v} 
& \multicolumn{2}{c||}{\sS} 
& \multicolumn{2}{c||}{\{ \pi_{\RH} \}} 
& \emptyset \\ \hline
\text{\rm (ii)} & \sS_{s_{i}u,v}
& \multicolumn{2}{c||}{\emptyset} 
& \multicolumn{2}{c||}{\sS \setminus \{\pi_{\RH}\}} 
& \emptyset \\ \hline\hline
\text{\rm (iii)} & \sS_{u,s_{i}v}
& \{ \pi_{\RL} \} & \sS  
& \emptyset & \{ \pi_{\RH} \}
& \emptyset \\ \hline
\text{\rm (iv)} & \sS_{s_{i}u,s_{i}v}
& \emptyset & \emptyset
& \{\pi_{\RL}\} & \sS \setminus \{ \pi_{\RH} \}
& \emptyset \\ \hline\hline
&
& \text{\rm Case 1.1} & \text{\rm Case 1.2}
& \text{\rm Case 2.1} & \text{\rm Case 2.2}
& \text{\rm Case 3} \\ \hline
\end{array}
}
\end{equation}
\end{prop}

\begin{proof}
We set $w:=\kappa(\pi_{\RH})$ and $N:=\vp_{i}(\pi_{\RH}) = \ve_{i}(\pi_{\RL}) \ge 0$. 
Since $s_{i}v \sig v$ by the assumption, 
we see by Corollary~\ref{cor:str0} and \eqref{eq:532} that 
\begin{equation} \label{eq:cs1a}
\SLS_{\sige v}(\mu) \cap \sS = \sS \quad \text{or} \quad \emptyset. 
\end{equation}
Also, we see by \eqref{eq:532} that 
$\SLS_{\sige v}(\mu) \supset \SLS_{\sige s_{i}v}(\mu)$. 
Therefore, 
\begin{equation} \label{eq:cs1b}
\SLS_{\sige v}(\mu) \cap \sS = \emptyset \quad \Rightarrow \quad 
\SLS_{\sige s_{i}v}(\mu) \cap \sS = \emptyset. 
\end{equation}
Assume that $\SLS_{\sige v}(\mu) \cap \sS = \sS$. 
If $\pair{w\mu}{\alpha_{i}^{\vee}} \le 0$ 
(resp., if $\pair{w\mu}{\alpha_{i}^{\vee}} > 0$ and $w \sige \PS{\Jm}(s_{i}v)$), 
then we deduce from 
Lemma~\ref{lem:str0} and Lemma~\ref{lem:dmd}\,(1) 
(resp., from Lemma~\ref{lem:str0}) 
that $\SLS_{\sige s_{i}v}(\mu) \cap \sS = \sS$, that is,
\begin{equation} \label{eq:cs1c}
\left.
\begin{array}{c}
\pair{w\mu}{\alpha_{i}^{\vee}} \le 0, \quad \text{or} \\[2mm]
\pair{w\mu}{\alpha_{i}^{\vee}} > 0 \text{ and }
w \sige \PS{\Jm}(s_{i}v)
\end{array} \right\} 
\quad \Rightarrow \quad
\SLS_{\sige s_{i}v}(\mu) \cap \sS = \sS. 
\end{equation}
Also, it follows from Lemma~\ref{lem:str0} and Lemma~\ref{lem:dmd}\,(3) 
that
\begin{equation} \label{eq:cs1d}
\pair{w\mu}{\alpha_{i}^{\vee}} > 0
\text{ and } w \not\sige \PS{\Jm}(s_{i}v)
\quad \Rightarrow \quad
\SLS_{\sige s_{i}v}(\mu) \cap \sS = \bigl\{ \pi_{\RL} \bigr\}.
\end{equation}
%
%
\begin{claim} \label{c:str2}
Let $y \in W_{\af}$ be such that 
$\SLS_{\sige y}(\mu) \cap \sS = \sS$. 
\begin{enu}
\item If $y^{-1}\alpha_{i} \in \Delta^{+}+\BZ \delta$, 
then $\io{\pi_{\RH}}{y}^{-1}\alpha_{i} \in \Delta^{+} + \BZ \delta$. 

\item If $N \ge 1$, then $\io{\pi_{\RH}}{y}^{-1}\alpha_{i} \in \Delta^{+} + \BZ \delta$ and 
$\io{f_{i}\pi_{\RH}}{y} \in \bigl\{ \io{\pi_{\RH}}{y},\,s_{i}\io{\pi_{\RH}}{y} \bigr\}$. 

\item If $N \ge 2$, then
\begin{equation} \label{eq:a0}
\io{f_{i}\pi_{\RH}}{y}=\io{f_{i}^{2}\pi_{\RH}}{y}= \cdots =
\io{f_{i}^{N-1}\pi_{\RH}}{y}=\io{\underbrace{f_{i}^{N}\pi_{\RH}}_{=\pi_{\RL}}}{y}.
\end{equation}
\end{enu}
\end{claim}
\noindent
{\it Proof of Claim~\ref{c:str2}.}
Since $\ve_{i}(\pi_{\RH})=0$, part (2) follows from Lemma~\ref{lem:iota1}, 
applied to the case that $\nu = \mu$ and $\psi=\pi_{\RH}$. 
Also, part (3) follows from Lemma~\ref{lem:iota1}, 
applied to the case that $\nu = \mu$ and $\psi=f_{i}^{k}\pi_{\RH}$ 
for each $1 \le k \le N-1$; note that $\ve_{i}(f_{i}^{k}\pi_{\RH}) \ge 1$ if $k \ge 1$. 
Hence it remains to show part (1) in the case that $N=0$. 
Let $\Lambda \in P^{+}$ be an arbitrary regular and dominant weight, 
that is, $\pair{\Lambda}{\alpha_{i}^{\vee}} > 0$ for all $i \in I$, 
or equivalently, $\J_{\Lambda}=\emptyset$. 
We set $\eta:=(y\,;\,0,\,1) \in \SLS(\Lambda)$; 
we see that $\io{\eta}{y}=y$, and hence 
$\pi_{\RH} \otimes \eta \in \SM_{\sige y}(\mu+\Lambda) 
\subset \SLS(\mu) \otimes \SLS(\Lambda)$ by Theorem~\ref{thm:SMT}\,(1).
Since $y^{-1}\alpha_{i} \in \Delta^{+}+\BZ \delta$ by the assumption, 
we see that $\ve_{i}(\pi_{\RH} \otimes \eta) = 0$ and 
$\vp_{i}(\pi_{\RH} \otimes \eta) \ge 1$ 
by the tensor product rule for crystals. 
Set $\psi:=\Phi_{\mu\Lambda}^{-1}(\pi_{\RH} \otimes \eta) \in 
\SLS_{\sige y}(\mu+\Lambda)$. Since
$\ve_{i}(\psi) = 0$ and $\vp_{i}(\psi) \ge 1$, 
it follows from Lemma~\ref{lem:iota1} that 
$\io{\psi}{y}^{-1}\alpha_{i} \in \Delta^{+}+\BZ\delta$. 
By Proposition~\ref{prop:smt} and the fact that $\io{\eta}{y}=y$ seen above, 
we deduce that $\io{\psi}{y} = \io{\pi_{\RH}}{\io{\eta}{y}} = 
\io{\pi_{\RH}}{y}$. Hence we obtain
$\io{\pi_{\RH}}{y}^{-1}\alpha_{i} \in \Delta^{+} + \BZ \delta$, 
as desired. \bqed

%
\begin{claim} \label{c:str3}
Assume that $\SLS_{\sige v}(\mu) \cap \sS = \sS${\rm;}
note that $\pi_{\RL} \in \SLS_{\sige s_{i}v}(\mu)$ by \eqref{eq:cs1c} and \eqref{eq:cs1d}. 
If $N \ge 1$, then $\io{\pi_{\RL}}{s_{i}v} = \io{\pi_{\RL}}{v}$, 
and hence
\begin{equation} \label{eq:cs3}
\begin{array}{r}
\overbrace{\io{f_{i}\pi_{\RH}}{v}=\io{f_{i}^{2}\pi_{\RH}}{v}= \cdots =
\io{f_{i}^{N-1}\pi_{\RH}}{v}}^{\text{\rm if $N \ge 2$; see \eqref{eq:a0}}}
=\io{\overbrace{f_{i}^{N}\pi_{\RH}}^{=\pi_{\RL}}}{v} \\[1mm]
\rotatebox{90}{$=$} \hspace*{9mm} \\
\underbrace{%
\io{f_{i}\pi_{\RH}}{s_{i}v}=\io{f_{i}^{2}\pi_{\RH}}{s_{i}v}= \cdots =
\io{f_{i}^{N-1}\pi_{\RH}}{s_{i}v}}_{%
\text{\rm if $N \ge 2$ and $\SLS_{\sige s_{i}v}(\mu) \cap \sS = \sS$; 
see \eqref{eq:cs1c} and \eqref{eq:a0}}}
=\io{\underbrace{f_{i}^{N}\pi_{\RH}}_{=\pi_{\RL}}}{s_{i}v}.
\end{array}
\end{equation}
\end{claim}

\noindent
{\it Proof of Claim~\ref{c:str3}.}
Take an arbitrary regular and dominant weight $\Lambda \in P^{+}$; 
note that $\pair{v\Lambda}{\alpha_{i}^{\vee}} > 0$ 
since $s_{i}v \sig v$ by the assumption. 
Set $\pi_{v}:=(v\,;\,0,\,1) \in \SLS(\Lambda)$.  
Since $s_{i}v \sig v$, 
we see by Theorem~\ref{thm:SMT}\,(1) that 
$\pi_{\RL} \otimes \pi_{v} \in 
\SM_{\sige v}(\mu+\Lambda) \subset \SLS(\mu) \otimes \SLS(\Lambda)$;
we set $\psi:=\Phi_{\mu\Lambda}^{-1}(\pi_{\RL} \otimes \pi_{v}) 
\in \SLS_{\sige v}(\mu+\Lambda)$. Recall that $N=\ve_{i}(\pi_{\RL}) \ge 1$ by the assumption, 
which implies that $\ve_{i}(\psi) = \ve_{i}(\pi_{\RL} \otimes \pi_{v}) \ge 1$ 
by the tensor product rule for crystals. Also, we have 
$\vp_{i}(\pi_{v}) = \pair{v\Lambda}{\alpha_{i}^{\vee}} > 0$, 
which implies that $\vp_{i}(\psi) \ge 1$, and hence $f_{i}\psi \ne \bzero$. 
Therefore, by Lemma~\ref{lem:iota1}, we have 
\begin{equation} \label{cs3-1}
\io{\psi}{v} = \io{f_{i}\psi}{v}.
\end{equation}
Here it follows from Proposition~\ref{prop:smt}\,(1) that
$\io{\psi}{v} = \io{\pi_{\RL}}{\io{\pi_{v}}{v}}$; 
it is easily checked that $\io{\pi_{v}}{v} = v$. 
Hence we obtain 
\begin{equation} \label{cs3-2}
\io{\psi}{v} = \io{\pi_{\RL}}{v}.
\end{equation}
Observe that 
$\Phi_{\mu\Lambda}(f_{i}\psi)=f_{i}(\pi_{\RL} \otimes \pi_{v}) = 
\pi_{\RL} \otimes (f_{i}\pi_{v})$ by the tensor product rule for crystals,
where $f_{i}\pi_{v} = (s_{i}v,v\,;\,0,a,1)$ for some $0 < a < 1$. 
It follows from Proposition~\ref{prop:smt}\,(1) that
$\io{f_{i}\psi}{v} = \io{\pi_{\RL}}{\io{f_{i}\pi_{v}}{v}}$; 
it is easily checked that $\io{f_{i}\pi_{v}}{v} = s_{i}v$. 
Hence we obtain
\begin{equation} \label{cs3-3}
\io{f_{i}\psi}{v} = \io{\pi_{\RL}}{s_{i}v}.
\end{equation}
From \eqref{cs3-1}, \eqref{cs3-2}, and \eqref{cs3-3}, 
we conclude that $\io{\pi_{\RL}}{v} = \io{\pi_{\RL}}{s_{i}v}$, as desired. \bqed
%
%
\begin{claim} \label{c:str4}
Assume that $\SLS_{\sige v}(\mu) \cap \sS = \sS$. 
If $N = 0$, then $\io{\pi_{\RL}}{s_{i}v}$ is identical to 
$\io{\pi_{\RL}}{v}$ or $s_{i}\io{\pi_{\RL}}{v}$.  
\end{claim}

\noindent
{\it Proof of Claim~\ref{c:str4}.}
The proof is similar to that of Claim~\ref{c:str3}; 
instead of \eqref{cs3-1}, we can show by Lemma~\ref{lem:iota1} that 
\begin{equation} \label{cs4}
\io{f_{i}\psi}{v} \in \bigl\{ \io{\psi}{v},\,s_{i}\io{\psi}{v} \bigr\}.
\end{equation}
Also, equations \eqref{cs3-2} and \eqref{cs3-3} hold also in this case. 
Substituting \eqref{cs3-2} and \eqref{cs3-3} into \eqref{cs4}, 
we obtain $\io{\pi_{\RL}}{s_{i}v} \in 
\bigl\{ \io{\pi_{\RL}}{v},\,s_{i}\io{\pi_{\RL}}{v} \bigr\}$, 
as desired. \bqed

\vspace{3mm}

In order to show that 
$\sS_{u,v} = \sS$, $\bigl\{\pi_{\RH}\bigr\}$, or $\emptyset$, 
let us assume that $\sS_{u,v} \ne \emptyset$; 
by \eqref{eq:cs1a}, we have $\SLS_{\sige v}(\mu) \cap \sS = \sS$. 
If $\#\sS = 1$ (or equivalently, $N=0$), 
then it is obvious that $\sS_{u,v} = \sS = \bigl\{ \pi_{\RH} \bigr\}$. 
Assume now that $\#\sS \ge 2$ (or equivalently, $N \ge 1$).
Since $s_{i}u \sig u$ by the assumption, 
we see from Claim~\ref{c:str2} that 
either \eqref{eq:a1} or \eqref{eq:a2} below holds: 
\begin{equation} \label{eq:a1}
\underbrace{
\io{\pi_{\RH}}{v} = 
\io{f_{i}\pi_{\RH}}{v}= \cdots = \io{f_{i}^{N}\pi_{\RH}}{v}}_{=u}; 
\end{equation}
\begin{equation} \label{eq:a2}
\underbrace{\io{\pi_{\RH}}{v}}_{=u} \sil
\underbrace{
\io{f_{i}\pi_{\RH}}{v}= \cdots = \io{f_{i}^{N}\pi_{\RH}}{v}}_{%
=s_{i}u}. 
\end{equation}
If \eqref{eq:a1} (resp., \eqref{eq:a2}) holds, 
then $\sS_{u,v} = \sS$ (resp., $\sS_{u,v}=\bigl\{ \pi_{\RH} \bigr\}$). 
Thus we have proved that 
$\sS_{u,v} = \sS$, $\bigl\{\pi_{\RH}\bigr\}$, or $\emptyset$, as desired. 
Hence we have the following possibilities (Cases 1, 2a, 2b, and 3).

\paragraph{Case 1.}
Assume that $\#\sS \ge 2$ and $\sS_{u,v}=\sS$. 
Since \eqref{eq:a1} holds in this case, 
it is obvious that $\sS_{s_{i}u,v} = \emptyset$.
Recall that $w = \kappa(\pi_{\RH})$. 

\paragraph{Subcase 1.1.}
Assume that $\pair{w\mu}{\alpha_{i}^{\vee}} > 0$ and 
$w \not\sige \PS{\Jm}(s_{i}v)$. 
It follows from \eqref{eq:cs1d} that 
$\SLS_{\sige s_{i}v}(\mu) \cap \sS = \bigl\{\pi_{\RL}\bigr\}$. 
By Claim~\ref{c:str3} and \eqref{eq:a1}, 
we have $\io{\pi_{\RL}}{s_{i}v} = u$. 
Hence we conclude that $\sS_{u,s_{i}v} = \bigl\{\pi_{\RL}\bigr\}$ and 
$\sS_{s_{i}u,s_{i}v} = \emptyset$. 

\paragraph{Subcase 1.2.}
Assume that $\pair{w\mu}{\alpha_{i}^{\vee}} \le 0$, or 
that $\pair{w\mu}{\alpha_{i}^{\vee}} > 0$ and 
$w \sige \PS{\Jm}(s_{i}v)$. 
It follows from \eqref{eq:cs1c} that 
$\SLS_{\sige s_{i}v}(\mu) \cap \sS = \sS$. 
By Claim~\ref{c:str2}\,(2), \eqref{eq:cs3}, \eqref{eq:a1}, and 
the assumption $s_{i}u \sig u$, 
we deduce that 
\begin{equation} \label{eq:a3}
\underbrace{
\io{\pi_{\RH}}{s_{i}v} = 
\io{f_{i}\pi_{\RH}}{s_{i}v}= \cdots =\io{f_{i}^{N}\pi_{\RH}}{s_{i}v}}_{=u}. 
\end{equation}
Therefore, we conclude that $\sS_{u,s_{i}v} = \sS$ and 
$\sS_{s_{i}u,s_{i}v} = \emptyset$. 

\paragraph{Case 2a.}
Assume that $\#\sS \ge 2$ and $\sS_{u,v}=\bigl\{ \pi_{\RH} \bigr\}$. 
Since \eqref{eq:a2} holds in this case, 
it is obvious that $\sS_{s_{i}u,v} = \sS \setminus \bigl\{ \pi_{\RH} \bigr\}$. 

\paragraph{Subcase 2a.1.}
Assume that $\pair{w\mu}{\alpha_{i}^{\vee}} > 0$ and 
$w \not\sige \PS{\Jm}(s_{i}v)$. 
It follows from \eqref{eq:cs1d} that 
$\SLS_{\sige s_{i}v}(\mu) \cap \sS = \bigl\{\pi_{\RL}\bigr\}$. 
By Claim~\ref{c:str3} and \eqref{eq:a2}, 
we have $\io{\pi_{\RL}}{s_{i}v} = s_{i}u$. 
Hence we conclude that $\sS_{u,s_{i}v} = \emptyset$ 
and $\sS_{s_{i}u,s_{i}v} = \bigl\{\pi_{\RL}\bigr\}$. 

\paragraph{Subcase 2a.2.}
Assume that $\pair{w\mu}{\alpha_{i}^{\vee}} \le 0$, or 
that $\pair{w\mu}{\alpha_{i}^{\vee}} > 0$ and 
$w \sige \PS{\Jm}(s_{i}v)$. 
It follows from \eqref{eq:cs1c} that 
$\SLS_{\sige s_{i}v}(\mu) \cap \sS = \sS$. 
By Claim~\ref{c:str2}\,(2), \eqref{eq:cs3}, \eqref{eq:a2}, and 
the assumption $s_{i}u \sig u$, we deduce that 
\begin{equation} \label{eq:a4}
\underbrace{\io{\pi_{\RH}}{s_{i}v}}_{=u} \sil
\underbrace{
\io{f_{i}\pi_{\RH}}{s_{i}v}= \cdots = \io{f_{i}^{N}\pi_{\RH}}{s_{i}v}}_{%
=s_{i}u}. 
\end{equation}
Therefore, we conclude that 
$\sS_{u,s_{i}v} = \bigl\{\pi_{\RH} \bigr\}$ and 
$\sS_{s_{i}u,s_{i}v} = \sS \setminus \bigl\{\pi_{\RH} \bigr\}$. 

\paragraph{Case 2b.}
Assume that $\#\sS = 1$ 
(or equivalently, $N=0$; 
in this case, $\sS=\bigl\{\pi_{\RH}\bigr\}=\bigl\{\pi_{\RL}\bigr\}$) 
and $\sS_{u,v}=\bigl\{ \pi_{\RH} \bigr\} = \sS$. 
It is obvious that $\sS_{s_{i}u,v} = \emptyset = 
\sS \setminus \bigl\{ \pi_{\RH} \bigr\}$. 

Since $\sS=\bigl\{\pi_{\RH}\bigr\}=\bigl\{\pi_{\RL}\bigr\}$, 
we have $\SLS_{\sige s_{i}v}(\mu) \cap \sS = \sS$ by \eqref{eq:cs1c} and \eqref{eq:cs1d}. 
Also, it follows from Claim~\ref{c:str4} that 
$\io{\pi_{\RL}}{s_{i}v}$ is identical to 
$\io{\pi_{\RL}}{v}$ or $s_{i}\io{\pi_{\RL}}{v}$; 
in this case, 
$\pi_{\RL}=\pi_{\RH}$ and $\io{\pi_{\RH}}{v} = u$. 
Therefore, $\io{\pi_{\RH}}{s_{i}v} \in 
\bigl\{ u, \, s_{i}u \bigr\}$. 

\paragraph{Subcase 2b.1.}
If $\io{\pi_{\RH}}{s_{i}v} = s_{i}u$, then 
$\sS_{u,s_{i}v}=\emptyset$ and 
$\sS_{s_{i}u,s_{i}v}=\bigl\{ \pi_{\RH} \bigr\} = \bigl\{ \pi_{\RL} \bigr\}$. 

\paragraph{Subcase 2b.2.}
If $\io{\pi_{\RH}}{s_{i}v} = u$, then 
$\sS_{u,s_{i}v}=\bigl\{ \pi_{\RH} \bigr\}$ and 
$\sS_{s_{i}u,s_{i}v}=\emptyset=\sS \setminus \bigl\{ \pi_{\RH} \bigr\}$. 

\paragraph{Case 3.}
Assume that $\sS_{u,v}=\emptyset$; 
we show that $\sS_{s_{i}u,v}=\sS_{u,s_{i}v}=\sS_{s_{i}u,s_{i}v}=\emptyset$. 
By \eqref{eq:cs1a} and \eqref{eq:cs1b}, we may (and do) assume that 
$\SLS_{\sige v}(\mu) \cap \sS = \sS$. 
We set $z:=\io{\pi_{\RH}}{v}$; remark that $z \ne u$ since $\sS_{u,v}=\emptyset$. 
We see from Claim~\ref{c:str2} that 
$z^{-1}\alpha_{i} \in \Delta^{+}+\BZ\delta$, and 
$\io{f_{i}^{k}\pi_{\RH}}{v} \in \bigl\{z,\,s_{i}z\bigr\}$ 
for all $0 \le k \le N$. 
Therefore, if at least one of $\sS_{s_{i}u,v}$, $\sS_{u,s_{i}v}$, 
and $\sS_{s_{i}u,s_{i}v}$ is nonempty, then $z$ is identical to 
$u$ or $s_{i}u$. Since $z^{-1}\alpha_{i} \in \Delta^{+}+\BZ\delta$, 
and $s_{i}u \sig u$ by the assumption, we must have $z = u$, 
which is a contradiction. 

This completes the proof of Proposition~\ref{prop:string}. 
\end{proof}
%
%
\section{Proof of Theorem~\ref{thm:main}: part 2.}
\label{sec:prf2}
%
%
\subsection{Recursion formula for the right-hand side of \eqref{eq:main}.}
\label{subsec:rec2}

Let us take an arbitrary $\mu \in P^{+}$. 
Recall from \eqref{eq:dual} and \eqref{eq:dual2} 
the definition and some properties of 
the dual path $\pi^{\ast} \in \SLS(-\lng\mu)$ for $\pi \in \SLS(\mu)$. 
We deduce by Lemmas~\ref{lem:lng1} and \ref{lem:io-kap} that 
for $x,y \in W_{\af}$ such that $y \sige x$,
%
%
\begin{equation} \label{eq:ba}
 \sum_{ \begin{subarray}{c}
  \pi \in \SLS(\mu) \\[1mm] \iota(\pi) \sile \PS{\J_{\mu}}(y),\,\kap{\pi}{y}=x 
  \end{subarray}} 
  \be^{-\wt(\pi)} 
=  
 \sum_{ \begin{subarray}{c}
  \eta \in \SLS(-\lng\mu) \\[1mm] 
  \kappa(\eta) \sige \PS{\J_{-\lng\mu}}(y\lng),\,\,
  \io{\eta}{y\lng}=x\lng 
  \end{subarray}} 
  \be^{\wt(\eta)} =: \ba_{\mu}(x,y). 
\end{equation}
The next lemma follows from Proposition~\ref{prop:string} 
(applied to the case that $u=x\lng$ and $v=y\lng$) and Lemma~\ref{lem:dem}, 
together with Remark~\ref{rem:str}. 
%
%
\begin{lem} \label{lem:ba}
Let $x,y \in W_{\af}$ be such that $y \sige x$, and $s_{i}x \sil x$, $s_{i}y \sil y$. 
We have 
\begin{equation} \label{eq:r1}
\ba_{\mu}(s_{i}x,y) = -s_{i} \sT_{i}\ba_{\mu}(x,y), 
\end{equation}
\begin{equation} \label{eq:r2}
\ba_{\mu}(s_{i}x,s_{i}y) = -s_{i} \sT_{i}\ba_{\mu}(x,s_{i}y) + 
s_{i}\bigl( \ba_{\mu}(x,y) - \ba_{\mu}(x,s_{i}y) \bigr), 
\end{equation}
\begin{equation} \label{eq:r3}
s_{i}\ba_{\mu}(s_{i}x,y) + s_{i}\ba_{\mu}(x,y) = 
\ba_{\mu}(s_{i}x,y) + \ba_{\mu}(x,y),
\end{equation}
\begin{equation} \label{eq:r4}
s_{i}\ba_{\mu}(s_{i}x,s_{i}y) + s_{i}\ba_{\mu}(x,y) = 
\ba_{\mu}(s_{i}x,s_{i}y) + \ba_{\mu}(x,y), 
\end{equation}
\begin{equation} \label{eq:r5}
s_{i}\ba_{\mu}(s_{i}x,y) - s_{i}\ba_{\mu}(s_{i}x,s_{i}y) = 
\ba_{\mu}(s_{i}x,y) - \ba_{\mu}(s_{i}x,s_{i}y). 
\end{equation}
\end{lem}

For simplicity of notation, we set
\begin{equation*}
\begin{split}
& (W_{\af})_{\sige x} := \bigl\{ y \in W_{\af} \mid y \sige x \bigr\}
  \quad \text{for $x \in W_{\af}$}, \\
& \bv_{\lambda}(x,y):=(-1)^{\sell(y)-\sell(x)} \gch V_{y}^{-}(\lambda)
  \quad \text{for $\lambda \in P^{+}$ and $x,y \in W_{\af}$}. 
\end{split}
\end{equation*}
Also, for $\lambda,\mu \in P^{+}$ such that $\lambda-\mu \in P^{+}$, 
we denote by $\bF(x)=\bF_{\lambda\mu}(x)$ 
the right-hand side of \eqref{eq:main}; 
we see that
%
%
\begin{equation} \label{eq:main2}
\bF(x) = 
 \sum_{ y \in (W_{\af})_{\sige x} }
  \ba_{\mu}(x,y) \bv_{\lambda}(x,y).
\end{equation}
%
%
\begin{prop}[{cf.\,\eqref{eq:dem2} and \eqref{eq:dem3}}] \label{prop:rec2}
Let $\lambda,\mu \in P^{+}$ be such that $\lambda-\mu \in P^{+}$. 
For $x \in W_{\af}$ and $i \in I_{\af}$ such that $s_{i}x \sil x$, 
it holds that $\sT_{i}\bF_{\lambda\mu}(x) = \bF_{\lambda\mu}(s_{i}x)-\bF_{\lambda\mu}(x)$, 
and hence $\bF_{\lambda\mu}(s_{i}x) = \sD_{i}\bF_{\lambda\mu}(x)$. 
\end{prop}

\begin{proof}
We set
\begin{equation*}
(W_{\af})_{\sige x}^{+} :=
 \bigl\{ y \in (W_{\af})_{\sige x} \mid  s_{i}y \sig y \bigr\}, \qquad
(W_{\af})_{\sige x}^{-} :=
 \bigl\{ y \in (W_{\af})_{\sige x} \mid  s_{i}y \sil y \bigr\};
\end{equation*}
we have
\begin{equation*}
\bF(x)= \bF^{+}(x) + \bF^{-}(x), 
\quad \text{where} \quad
\bF^{\pm}(x):=
\sum_{ y \in (W_{\af})_{\sige x}^{\pm} }
 \ba_{\mu}(x,y)\bv_{\lambda}(x,y).
\end{equation*}
We compute as follows: 
\begin{align*}
\sT_{i}\bF^{+}(x)
& \stackrel{\eqref{eq:L2}}{=} 
  \sum_{ y \in (W_{\af})_{\sige x}^{+} }
   (\sT_{i}\ba_{\mu}(x,y)) \bv_{\lambda}(x,y) +
  \sum_{ y \in (W_{\af})_{\sige x}^{+} }
   (s_{i}\ba_{\mu}(x,y)) 
   (\underbrace{\sT_{i}\bv_{\lambda}(x,y)}_{=0 \text{ by \eqref{eq:dem3}}}) \\[2mm]
& \stackrel{\eqref{eq:r2}}{=} 
  \sum_{ y \in (W_{\af})_{\sige x}^{+} }
  (\ba_{\mu}(x,s_{i}y)-\ba_{\mu}(x,y)-s_{i}\ba_{\mu}(s_{i}x,y))\bv_{\lambda}(x,y) \\[2mm]
& = - \sum_{ y \in (W_{\af})_{\sige x}^{+} } 
      \ba_{\mu}(x,y)\bv_{\lambda}(x,y) + 
      \sum_{ y \in (W_{\af})_{\sige x}^{+} }
      (\ba_{\mu}(x,s_{i}y)-s_{i}\ba_{\mu}(s_{i}x,y))\bv_{\lambda}(x,y), 
\end{align*}
and
\begin{align*}
\sT_{i}\bF^{-}(x)
& \stackrel{\eqref{eq:L2}}{=} 
  \sum_{ y \in (W_{\af})_{\sige x}^{-} }
   (\sT_{i}\ba_{\mu}(x,y)) \bv_{\lambda}(x,y) +
  \sum_{ y \in (W_{\af})_{\sige x}^{-} }
   (s_{i}\ba_{\mu}(x,y)) (\sT_{i}\bv_{\lambda}(x,y)) \\[2mm]
& = \sum_{ y \in (W_{\af})_{\sige x}^{-} } 
   (\underbrace{-s_{i}\ba_{\mu}(s_{i}x,y)}_{\text{by \eqref{eq:r1}}}) \bv_{\lambda}(x,y) 
   - \sum_{ y \in (W_{\af})_{\sige x}^{-} } 
   (s_{i}\ba_{\mu}(x,y)) (\underbrace{\bv_{\lambda}(x,s_{i}y) + \bv_{\lambda}(x,y)}_{\text{by \eqref{eq:dem3}}}) \\[2mm]
& = -\sum_{ y \in (W_{\af})_{\sige x}^{-} } 
   (s_{i}\ba_{\mu}(s_{i}x,y) + s_{i}\ba_{\mu}(x,y)) \bv_{\lambda}(x,y) 
   - \sum_{ y \in (W_{\af})_{\sige x}^{-} } 
    (s_{i}\ba_{\mu}(x,y)) \bv_{\lambda}(x,s_{i}y) \\[2mm]
& = -\sum_{ y \in (W_{\af})_{\sige x}^{-} } 
   (\underbrace{\ba_{\mu}(s_{i}x,y) + \ba_{\mu}(x,y)}_{\text{by \eqref{eq:r3}}}) \bv_{\lambda}(x,y) 
    - \sum_{ y \in (W_{\af})_{\sige x}^{-} } 
    (s_{i}\ba_{\mu}(x,y)) \bv_{\lambda}(x,s_{i}y) \\[2mm]
& = - \sum_{ y \in (W_{\af})_{\sige x}^{-} } \ba_{\mu}(s_{i}x,y)\bv_{\lambda}(x,y) 
    - \sum_{ y \in (W_{\af})_{\sige x}^{-} } \ba_{\mu}(x,y) \bv_{\lambda}(x,y) \\
& \hspace*{80mm} 
    - \sum_{ y \in (W_{\af})_{\sige x}^{-} } 
    (s_{i}\ba_{\mu}(x,y)) \bv_{\lambda}(x,s_{i}y). 
\end{align*}
Therefore, we obtain 
\begin{align}
\sT_{i}\bF(x) &= \sT_{i}\bF^{+}(x) + \sT_{i}\bF^{-}(x) \nonumber \\
& =  - \bF(x) + 
      \sum_{ y \in (W_{\af})_{\sige x}^{+} }
      (\ba_{\mu}(x,s_{i}y)-s_{i}\ba_{\mu}(s_{i}x,y))\bv_{\lambda}(x,y) \nonumber \\
& \quad 
     - \sum_{ y \in (W_{\af})_{\sige x}^{-} } \ba_{\mu}(s_{i}x,y)\bv_{\lambda}(x,y) 
     - \sum_{ y \in (W_{\af})_{\sige x}^{-} } (s_{i}\ba_{\mu}(x,y)) \bv_{\lambda}(x,s_{i}y). \label{eq:rm1}
\end{align}
Since $s_{i}x \sil x$ by the assumption, 
we deduce by Lemma~\ref{lem:dmd}\,(1) that 
$(W_{\af})_{\sige s_{i}x}^{-} = (W_{\af})_{\sige x}^{-}$. 
Hence it follows that 
\begin{equation*}
- \sum_{ y \in (W_{\af})_{\sige x}^{-} } \ba_{\mu}(s_{i}x,y)\bv_{\lambda}(x,y) = 
\sum_{ y \in (W_{\af})_{\sige s_{i}x}^{-} } \ba_{\mu}(s_{i}x,y)\bv_{\lambda}(s_{i}x,y).
\end{equation*}
Also, we see by Lemma~\ref{lem:dmd}\,(3) 
that $y \in (W_{\af})_{\sige x}^{-}$ if and only if 
$s_{i}y \in (W_{\af})_{\sige s_{i}x}^{+}$. Hence we deduce that
\begin{equation*}
\sum_{ y \in (W_{\af})_{\sige x}^{-} } (s_{i}\ba_{\mu}(x,y)) \bv_{\lambda}(x,s_{i}y)
= \sum_{ y \in (W_{\af})_{\sige s_{i}x}^{+} } (s_{i}\ba_{\mu}(x,s_{i}y)) \bv_{\lambda}(x,y). 
\end{equation*}
Therefore, the right-hand side of \eqref{eq:rm1} plus $\bF(x)$ is identical to
\begin{align}
& \sum_{ y \in (W_{\af})_{\sige s_{i}x}^{-} } \ba_{\mu}(s_{i}x,y)\bv_{\lambda}(s_{i}x,y) \nonumber \\
& \qquad + 
\underbrace{%
  \sum_{ y \in (W_{\af})_{\sige x}^{+} }
      (\ba_{\mu}(x,s_{i}y)-s_{i}\ba_{\mu}(s_{i}x,y))\bv_{\lambda}(x,y)
  - \sum_{ y \in (W_{\af})_{\sige s_{i}x}^{+} } 
      (s_{i}\ba_{\mu}(x,s_{i}y)) \bv_{\lambda}(x,y)}_{=:\bG(x)}. \label{eq:G0}
\end{align}
Because $(W_{\af})_{\sige x}^{+} \subset (W_{\af})_{\sige s_{i}x}^{+}$ 
by the assumption that $s_{i}x \sil x$, we see that
%
%
\begin{align}
\bG(x) & = \sum_{ y \in (W_{\af})_{\sige x}^{+} }
      (\overbrace{\ba_{\mu}(x,s_{i}y)-
       s_{i}\ba_{\mu}(s_{i}x,y)-
       s_{i}\ba_{\mu}(x,s_{i}y)}^{%
       \text{$=-\ba_{\mu}(s_{i}x,y)$ by \eqref{eq:r4}}})\bv_{\lambda}(x,y) \nonumber \\
& \hspace{50mm}
  - \sum_{ \begin{subarray}{c} y \in (W_{\af})_{\sige s_{i}x}^{+} \\[1mm] y \not\sige x \end{subarray} } 
  (s_{i}\ba_{\mu}(x,s_{i}y)) \bv_{\lambda}(x,y). \label{eq:G}
\end{align}
Here we claim that 
for $y \in (W_{\af})_{\sige s_{i}x}^{+}$ such that $y \not\sige x$, 
%
%
\begin{equation} \label{eq:rm3}
s_{i}\ba_{\mu}(x,s_{i}y) = \ba_{\mu}(s_{i}x,y). 
\end{equation}
In order to show this equality, 
it suffices to prove the following claim.
%
%
\begin{claim} \label{c:r1}
Let $y \in (W_{\af})_{\sige s_{i}x}^{+}$ be such that $y \not\sige x$. 
Let $\sS \subset \SLS(-\lng\mu)$ be an $i$-string such that 
\begin{equation*}
\sS_{x\lng, s_{i}y\lng}:=
\sS \cap \sls{x\lng}{s_{i}y\lng}{-\lng\mu} \ne \emptyset.
\end{equation*}
Then we have
\begin{equation} \label{eq:cr1a}
\sS_{x\lng, s_{i}y\lng}=
\bigl\{ \pi_{\RH}^{\sS} \bigr\} \quad \text{\rm and} \quad
\sS_{s_{i}x\lng, y\lng}=\bigl\{ \pi_{\RL}^{\sS} \bigr\}. 
\end{equation}
\end{claim}

\noindent
{\it Proof of Claim~\ref{c:r1}.} 
Suppose, for a contradiction, that 
$\sS_{x\lng, y\lng} \ne \emptyset$. 
If $\pi \in \sS_{x\lng, y\lng}$, then 
$x\lng = \io{\pi}{y\lng} \sige y\lng$, 
and hence $y \sige x$ by Lemma~\ref{lem:lng1}. 
However, this contradicts the assumption that $y \not\sige x$. 
Hence we have $\sS_{x\lng, y\lng} = \emptyset$. 
Since $s_{i}x\lng \sig x\lng$ and $y\lng \sig s_{i}y\lng$ by the assumption, 
the set $\sS_{x\lng, y\lng}$ is determined 
by row (iii) in table \eqref{eq:table}. 
Since $\sS_{x\lng, y\lng} = \emptyset$ as seen above, 
we obtain \eqref{eq:cr1a} by rows (i) and (iv) in table \eqref{eq:table}.
This proves the claim (and hence \eqref{eq:rm3}). \bqed

\vspace{3mm}

Substituting \eqref{eq:rm3} into \eqref{eq:G}, 
we see that
\begin{align*}
\bG(x) & = 
  - \sum_{ y \in (W_{\af})_{\sige x}^{+} } \ba_{\mu}(s_{i}x,y)\bv_{\lambda}(x,y) 
  - \sum_{ \begin{subarray}{c} y \in (W_{\af})_{\sige s_{i}x}^{+} \\[1mm] y \not\sige x \end{subarray} } 
  \ba_{\mu}(s_{i}x,y) \bv_{\lambda}(x,y) \\[2mm]
& = 
  \sum_{ y \in (W_{\af})_{\sige x}^{+} } \ba_{\mu}(s_{i}x,y)\bv_{\lambda}(s_{i}x,y) +
  \sum_{ \begin{subarray}{c} y \in (W_{\af})_{\sige s_{i}x}^{+} \\[1mm] y \not\sige x \end{subarray} } 
  \ba_{\mu}(s_{i}x,y) \bv_{\lambda}(s_{i}x,y) \\[2mm]
& = 
  \sum_{ y \in (W_{\af})_{\sige s_{i}x}^{+} } \ba_{\mu}(s_{i}x,y)\bv_{\lambda}(s_{i}x,y). 
\end{align*}
Hence the right-hand side of \eqref{eq:rm1} plus 
$\bF(x)$ (see \eqref{eq:G0}) is identical to
\begin{equation*}
  \sum_{ y \in (W_{\af})_{\sige s_{i}x}^{-} } \ba_{\mu}(s_{i}x,y)\bv_{\lambda}(s_{i}x,y) + 
  \sum_{ y \in (W_{\af})_{\sige s_{i}x}^{+} } \ba_{\mu}(s_{i}x,y)\bv_{\lambda}(s_{i}x,y) = \bF(s_{i}x).  
\end{equation*}
Substituting this equality into \eqref{eq:rm1}, we conclude that 
$\sT_{i}\bF(x) = \bF(s_{i}x) - \bF(x)$, as desired. 
This completes the proof of Proposition~\ref{prop:rec2}. 
\end{proof}
%
%
\subsection{Second step in the proof of Theorem~\ref{thm:main}.}
\label{subsec:s2}

We prove \eqref{eq:main} in the case that 
$\mu=\vpi_{r}$ for $r \in I$ (and $x \in W_{\af}$ is general). 
Let $\lambda = \sum_{i \in I}\lambda_{i} \vpi_{i} \in P^{+}$ be such that 
$\lambda-\vpi_{r} \in P^{+}$, and let $x \in W_{\af}$. 
We deduce from \cite[Lemma~1.4]{AK} (see also 
\cite[(1a) and (2a) in the proof of Lemma~5.4.1]{NS16}) that 
there exist $i_{1},\,\dots,\,i_{n} \in I_{\af}$ 
and $\xi \in Q^{\vee}$ such that 
\begin{equation*}
x=s_{i_n}s_{i_{n-1}} \cdots s_{i_2}s_{i_1}t_{\xi} \sil 
s_{i_{n-1}} \cdots s_{i_2}s_{i_1}t_{\xi} \sil \cdots \cdots \sil 
s_{i_2}s_{i_1}t_{\xi} \sil s_{i_1}t_{\xi} \sil t_{\xi}.
\end{equation*}
By Propositions~\ref{prop:dem} and \ref{prop:rec2}, we see that
\begin{equation*}
\begin{split}
\gch V_{x}^{-}(\lambda-\vpi_{r}) & = 
\sD_{i_n}\sD_{i_{n-1}} \cdots \sD_{i_2}\sD_{i_1}
\gch V_{t_{\xi}}^{-}(\lambda-\vpi_{r}), \\
\bF_{\lambda\vpi_{r}} (x) & = 
\sD_{i_n}\sD_{i_{n-1}} \cdots \sD_{i_2}\sD_{i_1}
\bF_{\lambda\vpi_{r}} (t_{\xi}), 
\end{split}
\end{equation*}
respectively. Also, in \S\ref{subsec:s1}, we proved that 
\begin{equation*}
\frac{1}{1-q^{-\lambda_{r}}} \gch V_{t_{\xi}}^{-}(\lambda-\vpi_{r}) = 
\bF_{\lambda\vpi_{r}} (t_{\xi}). 
\end{equation*}
From these equalities, we obtain 
\begin{equation*}
\frac{1}{1-q^{-\lambda_{r}}} \gch V_{x}^{-}(\lambda-\vpi_{r}) = 
\bF_{\lambda\vpi_{r}} (x), 
\end{equation*} 
as desired. 
%
%
\section{Proof of Theorem~\ref{thm:main}: part 3.}
\label{sec:prf3}
%
%
\subsection{Final step in the proof of Theorem~\ref{thm:main}.}
\label{subsec:s3}
%
Keep the notation and setting of Theorem~\ref{thm:main}. 
We prove \eqref{eq:main} 
by induction on $|\mu|:=\sum_{i \in I}\mu_{i} \in \BZ_{\ge 0}$. 
If $|\mu|=1$, that is, if $\mu=\vpi_{r}$ for some $r \in I$, 
then we proved \eqref{eq:main} in \S\ref{subsec:s2}. 
Assume now that $|\mu| > 1$, and take $r \in I$ such that $\mu_{r} \ge 1$. 
We set $\nu:=\mu-\vpi_{r} \in P_{+}$. We compute as follows: 
\begin{align*}
& \prod_{i \in I} \prod_{k = \lambda_{i}-\mu_{i}+1}^{\lambda_{i}} \frac{1}{1-q^{-k}}
  \gch V_{x}^{-}(\lambda-\mu) = 
  \prod_{i \in I} \prod_{k = \lambda_{i}-\mu_{i}+1}^{\lambda_{i}} \frac{1}{1-q^{-k}}
  \gch V_{x}^{-}(\lambda-\vpi_{r}-\nu) \\[3mm]
& \quad = \frac{1}{1-q^{-\lambda_{r}}}
 \sum_{ \begin{subarray}{c}
   y \in W_{\af} \\
   y \sige x \end{subarray}}\,
 \sum_{ \begin{subarray}{c}
  \pi \in \SLS(\nu) \\
  \iota(\pi) \sile \PS{\J_{\nu}}(y),\, \kap{\pi}{y}=x 
  \end{subarray}}  (-1)^{\sell(y)-\sell(x)}
  \be^{-\wt(\pi)} \gch V_{y}^{-}(\lambda-\vpi_{r}) \\
& \hspace*{100mm} \text{by our induction hypothesis} \\[3mm]
& \quad = \sum_{ \begin{subarray}{c}
   y \in W_{\af} \\
   y \sige x \end{subarray}}\,
 \sum_{ \begin{subarray}{c}
  \pi \in \SLS(\nu) \\ 
  \iota(\pi) \sile \PS{\J_{\nu}}(y),\, \kap{\pi}{y}=x 
  \end{subarray}}  (-1)^{\sell(y)-\sell(x)}
  \be^{-\wt(\pi)} \times \\[2mm]
& \hspace*{40mm}
 \sum_{ \begin{subarray}{c}
   z \in W_{\af} \\
   z \sige y \end{subarray}}\,
 \sum_{ \begin{subarray}{c}
  \eta \in \SLS(\vpi_{r}) \\
  \iota(\eta) \sile \PS{\J_{\vpi_{r}}}(z),\, \kap{\eta}{z}=y 
  \end{subarray}}  (-1)^{\sell(z)-\sell(y)}
  \be^{-\wt(\eta)}
  \gch V_{z}^{-}(\lambda) \\[3mm]
& \hspace*{100mm} \text{by the formula shown in \S\ref{subsec:s2}} \\[3mm]
& \quad =
 \sum_{ \begin{subarray}{c}
   z,y \in W_{\af} \\
   z \sige y \sige x \end{subarray}}\ 
 \sum_{ \begin{subarray}{c}
  \eta \in \SLS(\vpi_{r}) \\
  \iota(\eta) \sile \PS{\J_{\vpi_r}}(z),\, \kap{\eta}{z}=y 
  \end{subarray}}\ 
 \sum_{ \begin{subarray}{c}
  \pi \in \SLS(\nu) \\
  \iota(\pi) \sile \PS{\J_{\nu}}(y),\, \kap{\pi}{y}=x 
  \end{subarray}}  (-1)^{\sell(z)-\sell(x)}
  \be^{-(\wt(\pi)+\wt(\eta))} 
  \gch V_{z}^{-}(\lambda) \\
& \quad =:\bH(x).
\end{align*}
We see by Theorem~\ref{thm:SMT}\,(2) that 
for each $z \in W_{\af}$ with $z \sige x$, 
the set
\begin{align*}
& \bigsqcup_{ z \sige y \sige x }
\Bigl(
\bigl\{\eta \in \SLS(\vpi_{r}) \mid 
       \iota(\eta) \sile \PS{\J_{\vpi_r}}(z),\,\kap{\eta}{z}=y\bigr\} \times \\[-3mm]
& \hspace*{70mm}
\bigl\{\pi \in \SLS(\nu) \mid 
  \iota(\pi) \sile \PS{\J_{\nu}}(y),\, \kap{\pi}{y}=x \bigr\}\Bigr)
\end{align*}
is in bijection with the set 
$\bigl\{ \eta \otimes \pi \in 
\SM_{\sile z}(\vpi_{r}+\nu) \mid 
\kap{\pi}{\kap{\eta}{z}}=x \bigr\}$
by the map $(\eta,\pi) \mapsto \eta \otimes \pi$.
Hence we have
\begin{equation*}
\bH(x) = \sum_{ \begin{subarray}{c}
   z \in W_{\af} \\[1mm]
   z \sige x \end{subarray}}\ 
 \sum_{ \begin{subarray}{c}
  \eta \otimes \pi \in \SM_{\sile z}(\vpi_{r}+\nu) \\[1mm]
  \kap{\pi}{\kap{\eta}{z}}=x
  \end{subarray}} (-1)^{\sell(z)-\sell(x)}
 \be^{-(\wt(\pi)+\wt(\eta))} 
 \gch V_{z}^{-}(\lambda). 
\end{equation*}
Moreover, it follows from Proposition~\ref{prop:smt}\,(2) 
that the set $\bigl\{ \eta \otimes \pi \in 
\SM_{\sile z}(\vpi_{r}+\nu) \mid 
\kap{\pi}{\kap{\eta}{z}}=x \bigr\}$ is in bijection with the set 
$\bigl\{ \psi \in \SLS_{\sile z}(\vpi_{r}+\nu) \mid 
\kap{\psi}{z}=x\bigr\}$ by the map 
$\Phi_{\vpi_r\nu}: \SLS(\mu) = 
\SLS(\vpi_r + \nu) \stackrel{\sim}{\rightarrow} 
\SM(\vpi_r + \nu) 
\hookrightarrow \SLS(\vpi_{r}) \otimes \SLS(\nu)$.
Therefore, we conclude that
\begin{equation*}
\bH(x) = \sum_{ \begin{subarray}{c}
   z \in W_{\af} \\[1mm]
   z \sige x \end{subarray}}\ 
 \sum_{ \begin{subarray}{c}
  \psi \in \SLS(\mu) \\[1mm] 
  \iota(\psi) \sile \Pi^{J_{\mu}}(z),\, \kap{\psi}{z}=x 
  \end{subarray}} (-1)^{\sell(z)-\sell(x)}
  \be^{-\wt(\psi)} 
  \gch V_{z}^{-}(\lambda). 
\end{equation*}
This completes the proof of Theorem~\ref{thm:main}.
%
%
\subsection{Proof of Corollary~\ref{cor:main}.}
\label{subsec:prf-cor}

We take and fix $x \in W$. 
Let $\BX$ be the subset of $\SLS(\mu) \times W_{\af}$ 
consisting of those elements $(\pi,y)$ satisfying the conditions that 
$\iota(\pi) \sile \PS{\Jm}(y)$ and $\kap{\pi}{y}=x$; 
note that $y \sige \kap{\pi}{y}=x$ if $(\pi,y) \in \BX$. 
Then the right-hand side of \eqref{eq:main} can be rewritten as:
\begin{equation} \label{eq:maina}
\sum_{(\pi,y) \in \BX} (-1)^{\sell(y)-\sell(x)}
  \be^{-\wt(\pi)} 
  \gch V_{y}^{-}(\lambda). 
\end{equation}
Also, let $\BY$ be the subset of $\QLS(\mu) \times W$ 
consisting of those elements $(\eta,v)$ satisfying 
the condition that $\kap{\eta}{v}=x$.
We define the map $\cl:\SLS(\mu) \times W_{\af} \twoheadrightarrow 
\QLS(\mu) \times W$ by $\cl(\pi,y):=(\cl(\pi),\cl(y))$ for 
$(\pi,y) \in \SLS(\mu) \times W_{\af}$ 
(for the map $\cl:W_{\af} \twoheadrightarrow W$, see \S\ref{subsec:QLS}). 
We claim that $\cl(\BX)=\BY$.

Let $(\eta,v) \in \QLS(\mu) \times W$, and let 
$(\pi,y) \in \SLS(\mu) \times W_{\af}$. 
Write $\pi \in \SLS(\mu)$ as: 
\begin{equation} \label{eq:pi1}
\begin{cases}
\pi=(x_{1},\dots,x_{s}\,;\,a_{0},a_{1},\dots,a_{s}), \quad \text{with} \\[2mm]
x_{u}=w_{u}\PS{\Jm}(t_{\xi_u}) \quad 
 \text{for $w_{u} \in \WSu{\Jm}$ and $\xi_u \in Q^{\vee}$}, \quad 1 \le u \le s. 
\end{cases}
\end{equation}
Assume that $\cl(\pi,y) = (\eta,v)$. 
If we take $1 \le u_1 < \cdots < u_r = s$ in such a way that 
\begin{equation*}
\begin{split}
  w_{1} = \cdots = w_{u_1} & \ne 
  w_{u_1+1} = \cdots = w_{u_2} \ne 
  w_{u_2+1} = \cdots  \\
& \cdots \ne w_{u_{r-1}+1} = \cdots = w_{u_r} \ (=w_{s}), 
\end{split}
\end{equation*}
then $\eta \in \QLS(\mu)$ is of the form 
$\eta=(w_{u_1},\dots,w_{u_r};a_{0},a_{u_{1}},\dots,a_{u_{r}})$. 
Since $\cl(y)=v$, we see that $y \in W_{\af}$ is of the form 
$y = vt_{\zeta}$ for some $\zeta \in Q^{\vee}$. 
If $\iota(\pi) \sile \PS{\Jm}(y)$, 
then we deduce by Proposition~\ref{prop:maxs} that 
\begin{equation} \label{eq:cor1}
\kap{\pi}{y} = \kap{\eta}{v} \cdot t_{[\xi_s]^{\Jm} + [\zeta-\zeta(\eta,v)]_{\Jm}}. 
\end{equation}
Moreover, if in addition $\kap{\pi}{y} = x \in W$, or equivalently, if $(\pi,y) \in \BX$, 
then we see by \eqref{eq:cor1} that 
$\kap{\eta}{v} = x$ since $\kap{\eta}{v} \in W$. 
Thus we obtain $(\eta,v) \in \BY$, 
which implies that $\cl(\BX) \subset \BY$.

By exactly the same argument as for Lemma~\ref{lem:deg} 
(that is, as for \cite[Lemma~6.2.3]{NS16}; see also \cite[Lemma~2.3.2]{LNSSS15}), 
we can show that for each $\eta \in \QLS(\mu)$ and $\bchi \in \Par(\mu)$, 
there exists a unique element $\pi_{\bchi,\eta} \in \SLS_{\bchi}(\mu)$ 
such that $\cl(\pi_{\bchi,\eta}) = \eta$ and 
$\kappa(\pi_{\bchi,\eta}) = \kappa(\eta) \in \WSu{\Jm}$; 
for the definitions of $\Par(\mu)$ and $\SLS_{\bchi}(\mu)$, 
see \S\ref{subsec:conn}. 
Notice that if $\bchi = (\emptyset)_{i \in I}$, 
then $\pi_{\bchi,\eta}$ is identical to 
$\pi_{\eta}$ in Lemma~\ref{lem:deg}.
We deduce by \cite[Lemma~7.1.4]{INS} that 
if $\iota(\pi_{\eta}) = w \PS{\Jm}(t_{\xi})$ 
for some $w \in \WSu{\Jm}$ and $\xi \in Q^{\vee}$, 
then $\iota(\pi_{\bchi,\eta}) = w \PS{\Jm}(t_{\xi+\ip{\bchi}})$ 
(for $\ip{\bchi}$, see Remark~\ref{rem:init}). 
Also, observe that 
$\wt (\pi_{\bchi,\eta}) = \wt (\pi_{\eta}) - |\bchi|\delta$. 

Fix $(\eta,v) \in \BY$. 
We claim that for $(\pi,y) \in \SLS(\mu) \times W_{\af}$, 
%
%
\begin{equation} \label{eq:iff}
\begin{split}
& \text{$\cl(\pi,y) = (\eta,v)$ and $(\pi,y) \in \BX$} \iff \\
& \hspace*{20mm}
\begin{cases}
\pi = \pi_{\bchi, \eta} & \text{for some $\bchi \in \Par(\mu)$}, \\[1mm]
y = vt_{\zeta(\eta,v)+\gamma} & \text{for some $\gamma \in Q^{\vee}_{I \setminus \Jm}$ such that 
$\gamma \ge \ip{\bchi}$};
\end{cases}
\end{split}
\end{equation}
notice that the opposite inclusion $\cl(\BX) \supset \BY$ 
follows from this claim.
First, we show the implication $\Rightarrow$. 
Write $\pi \in \SLS(\mu)$ in the form \eqref{eq:pi1}.
If we write $y \in W_{\af}$ as $y = vt_{\zeta}$ with $\zeta \in Q^{\vee}$ 
(recall that $\cl(y)=v$), then we see by \eqref{eq:cor1} that 
\begin{equation} \label{eq:cor2}
x = \kap{\pi}{y} =  
\kap{\eta}{v} \cdot t_{[\xi_s]^{\Jm} + [\zeta-\zeta(\eta,v)]_{\Jm}} = 
x \cdot t_{[\xi_s]^{\Jm} + [\zeta-\zeta(\eta,v)]_{\Jm}};
\end{equation}
recall that $(\pi,y) \in \BX$ and $(\eta,v) \in \BY$. 
Therefore, we deduce that $[\xi_s]^{\Jm} + [\zeta-\zeta(\eta,v)]_{\Jm} = 0$, 
and hence $[\xi_s]^{\Jm} = 0$ and $[\zeta-\zeta(\eta,v)]_{\Jm} = 0$. 
By the equality $[\xi_s]^{\Jm} = 0$ and Lemma~\ref{lem:PiJ}\,(3), 
we have $\kappa(\pi) = x_{s} = w_{s}\PS{\Jm}(t_{\xi_s}) = w_{s} \in \WSu{\Jm}$. 
Since $\cl(\pi) = \eta$ by our assumption, we see 
by the definition that $\pi = \pi_{\bchi, \eta}$ 
if $\pi \in \SLS_{\bchi}(\mu)$ for $\bchi \in \Par(\mu)$ 
(see the previous paragraph).
Since $\iota(\pi_{\eta}) = 
w_{1} \PS{\Jm}(t_{\zeta(\eta,v)-\wt(\ha{w}_{1} \Rightarrow v)})$ 
by Remark~\ref{rem:pieta}, it follows that 
\begin{equation} \label{eq:pi2}
\iota(\pi) = \iota(\pi_{\bchi, \eta}) = 
w_{1} \PS{\Jm}(t_{\zeta(\eta,v)-\wt(\ha{w}_{1} \Rightarrow v)+\ip{\bchi}}).
\end{equation}
Since $(\pi,y) \in \BX$, we have $\PS{\Jm}(y) = \PS{\Jm}(vt_{\zeta}) \sige \iota(\pi)$.
Hence it follows from Lemmas~\ref{lem:siwt} and \ref{lem:wtS} 
(recall that $\ha{w}_{1} \in w_{1} \WS{\Jm}$) that
\begin{equation*}
[\zeta]^{\Jm} \ge 
[\zeta(\eta,v)-\wt(\ha{w}_{1} \Rightarrow v)+\ip{\bchi} + \wt(w_{1} \Rightarrow v)]^{\Jm} = 
[\zeta(\eta,v)+\ip{\bchi}]^{\Jm};
\end{equation*}
recall that $\ip{\bchi} \in Q^{\vee}_{I \setminus \Jm}$. 
Combining this inequality with the equality $[\zeta-\zeta(\eta,v)]_{\Jm} = 0$ shown above, 
we obtain $\zeta = \zeta(\eta,v) + \gamma$ for 
some $\gamma \in Q^{\vee}_{I \setminus \Jm}$ such that 
$\gamma \ge \ip{\bchi}$. Thus we have proved the implication $\Rightarrow$. 
Next, we show the implication $\Leftarrow$.
It is obvious that $\cl(\pi) = \cl(\pi_{\bchi,\eta}) = \eta$ and $\cl(y) = v$. 
Also, we have 
$\iota(\pi) = \iota(\pi_{\bchi, \eta}) = 
 w_{1} \PS{\Jm}(t_{\zeta(\eta,v)-\wt(\ha{w}_{1} \Rightarrow v)+\ip{\bchi}})$ 
as shown in \eqref{eq:pi2}, 
where $w_{1} = \iota(\eta)$ and $\ha{w}_{1} = \tbmax{w_{1}}{\Jm}{v}$. Because
\begin{equation*}
[\zeta(\eta,v) + \gamma]^{\Jm} \ge 
[\zeta(\eta,v)+\ip{\bchi}]^{\Jm} = 
[\zeta(\eta,v)-\wt(\ha{w}_{1} \Rightarrow v)+\ip{\bchi} + \wt(w_1 \Rightarrow \mcr{v}^{\Jm})]^{\Jm}
\end{equation*}
by Lemma~\ref{lem:wtS}, it follows from Lemma~\ref{lem:siwt} that 
\begin{equation*}
\begin{split}
\PS{\Jm}(y) & = \PS{\Jm}(vt_{\zeta(\eta,v) + \gamma}) 
              = \mcr{v}^{\Jm}\PS{\Jm}(t_{\zeta(\eta,v) + \gamma}) \\
& \sige w_{1} \PS{\Jm}(t_{\zeta(\eta,v)-\wt(\ha{w}_{1} \Rightarrow v)+\ip{\bchi}}) = 
 \iota(\pi_{\bchi, \eta}) = \iota(\pi).
\end{split}
\end{equation*}
Finally, by the same argument as for \eqref{eq:cor2} (applied to 
the case that $\xi_{s}=0$ and $\zeta = \zeta(\eta,v)+\gamma$), 
we deduce that $\kap{\pi}{y} =  
\kap{\eta}{v} \cdot t_{[0]^{\Jm} + [(\zeta(\eta,v)+\gamma)-\zeta(\eta,v)]_{\Jm}} = x$.
Thus we have shown the implication $\Leftarrow$, 
thereby completing the proof of \eqref{eq:iff}. 

By \eqref{eq:iff}, we can rewrite \eqref{eq:maina} 
(which is identical to the right-hand side of \eqref{eq:main}) as: 
\begin{equation*}
\begin{split}
& 
\sum_{v \in W} \, 
\sum_{
 \begin{subarray}{c}
 \eta \in \QLS(\mu) \\
 \kap{\eta}{v}=x
 \end{subarray}} \,
\sum_{\bchi \in \Par(\mu)} \, 
\sum_{
 \begin{subarray}{c}
 \gamma \in Q^{\vee}_{I \setminus \Jm} \\
 \gamma \ge \ip{\bchi}
 \end{subarray}} 
(-1)^{\sell(vt_{\zeta(\eta,v)+\gamma})-\sell(x)}
  \be^{-\wt(\pi_{\bchi,\eta})} 
  \gch V_{vt_{\zeta(\eta,v)+\gamma}}^{-}(\lambda) \\[3mm]
& = 
\sum_{v \in W} \, 
\sum_{
 \begin{subarray}{c}
 \eta \in \QLS(\mu) \\
 \kap{\eta}{v}=x
 \end{subarray}} \,
\sum_{\bchi \in \Par(\mu)} \, 
\sum_{
 \begin{subarray}{c}
 \gamma \in Q^{\vee}_{I \setminus \Jm} \\
 \gamma \ge \ip{\bchi}
 \end{subarray}} 
(-1)^{\ell(v)-\ell(x)}
  \be^{-\wt(\pi_{\eta})+|\bchi|\delta} 
  \underbrace{q^{-\pair{\lambda}{\gamma}}\gch V_{vt_{\zeta(\eta,v)}}^{-}(\lambda)}_{
  \text{by Proposition~\ref{prop:gch-tx}}} \\[3mm]
& = 
\sum_{\bchi \in \Par(\mu)} \, 
\sum_{
 \begin{subarray}{c}
 \gamma \in Q^{\vee}_{I \setminus \Jm} \\
 \gamma \ge \ip{\bchi}
 \end{subarray}} q^{ |\bchi|-\pair{\lambda}{\gamma} } 
\underbrace{%
\sum_{v \in W} \, 
\sum_{
 \begin{subarray}{c}
 \eta \in \QLS(\mu) \\
 \kap{\eta}{v}=x
 \end{subarray}} 
(-1)^{\ell(v)-\ell(x)}
  \be^{-\wt(\pi_{\eta})} 
  \gch V_{vt_{\zeta(\eta,v)}}^{-}(\lambda)}_{%
\text{the right-hand side of \eqref{eq:mainc}}}.
\end{split}
\end{equation*}
Therefore, in order to prove \eqref{eq:mainc} in Corollary~\ref{cor:main}, 
it suffices to show that 
\begin{equation} \label{eq:gen}
\sum_{\bchi \in \Par(\mu)} \, 
\sum_{
 \begin{subarray}{c}
 \gamma \in Q^{\vee}_{I \setminus \Jm} \\
 \gamma \ge \ip{\bchi}
 \end{subarray}} q^{ |\bchi|-\pair{\lambda}{\gamma} } 
=
\prod_{i \in I} \prod_{k = \lambda_{i}-\mu_{i}+1}^{\lambda_{i}} \frac{1}{1-q^{-k}}. 
\end{equation}
For each $i \in I \setminus \Jm$, 
we denote by $\Par(\mu_{i})$ the set of partitions of 
length less than $\mu_{i} \in \BZ_{\ge 1}$. 
We rewrite \eqref{eq:gen} as:
\begin{equation*}
\prod_{i \in I \setminus \Jm} \,
\underbrace{
\sum_{
 \begin{subarray}{c}
 \chi^{(i)} = (\chi^{(i)}_{1} \ge \cdots \ge \chi^{(i)}_{\mu_i-1} \ge 0) \\
  \in \Par(\mu_{i})
 \end{subarray}} \, 
\sum_{ 
  \begin{subarray}{c}
  c_{i} \in \BZ \\ c_{i} \ge \chi^{(i)}_{1}
  \end{subarray}} 
q^{ |\chi^{(i)}|-\lambda_{i}c_{i} } }_{%
\displaystyle = \sum_{(\chi^{(i)},c_{i}) \in \sP_{1}^{(i)}} 
q^{|\chi^{(i)}|-\lambda_{i}c_{i} } }
=
\prod_{i \in I \setminus \Jm}
\underbrace{%
\prod_{k = \lambda_{i}-\mu_{i}+1}^{\lambda_{i}} \frac{1}{1-q^{-k}} }_{%
\displaystyle = \sum_{\tau \in \sP_{2}^{(i)}} q^{-|\tau|} },
\end{equation*}
where for each $i \in I \setminus \Jm$, we set 
\begin{align*}
& \sP_{1}^{(i)} := \bigl\{ 
(\chi^{(i)},c_{i}) \mid 
\chi^{(i)} = (\chi^{(i)}_{1} \ge \cdots \ge \chi^{(i)}_{\mu_i-1} \ge 0) 
\in \Par(\mu_{i}) \text{ and } 
c_{i} \in \BZ,\,c_{i} \ge \chi^{(i)}_{1} \bigr\}, \\
& \sP_{2}^{(i)} := \bigl\{
\tau = (\tau_{1} \ge \tau_{2} \ge \cdots \ge \tau_{m} \ge 0) \mid m \ge 0, 
\lambda_{i}-\mu_{i}+1 \le \tau_{m} \le \tau_{1} \le \lambda_{i} \bigr\};
\end{align*}
recall that $\lambda_{i} \ge \mu_{i}$ for all $i \in I \setminus \Jm$. 
We show that for each $i \in I \setminus \Jm$, 
\begin{equation} \label{eq:gen2}
\sum_{(\chi^{(i)},c_{i}) \in \sP_{1}^{(i)}} 
q^{|\chi^{(i)}|-\lambda_{i}c_{i} } = 
\sum_{\tau \in \sP_{2}^{(i)}} q^{-|\tau|}. 
\end{equation}
Let $(\chi^{(i)},c_{i}) \in \sP_{1}^{(i)}$, and 
write $\chi^{(i)} \in \Par(\mu_{i})$ as
$\chi^{(i)} = (\chi^{(i)}_{1} \ge \cdots \ge \chi^{(i)}_{\mu_i-1} \ge 0)$. 
Observe that 
\begin{equation} \label{eq:par}
(\underbrace{c_{i},\,c_{i},\,\dots,\,c_{i},\,c_{i}}_{\text{$(\lambda_{i}-\mu_{i}+1)$ times}},\,
 c_{i}-\chi^{(i)}_{\mu_i-1},\,\dots,\,c_{i}-\chi^{(i)}_{2},\,c_{i}-\chi^{(i)}_{1})
\end{equation}
is a partition of length less than or equal to $\lambda_{i}$. 
We define $\Psi(\chi^{(i)},c_{i})$ to be the conjugate (or transposed) partition 
of the partition \eqref{eq:par}; it is easily checked that 
$\Psi(\chi^{(i)},c_{i}) \in \sP_{2}^{(i)}$. 
Also, we can easily check that 
the map $\Psi:\sP_{1}^{(i)} \rightarrow \sP_{2}^{(i)}$ is bijective, and 
$|\Psi(\chi^{(i)},c_{i})| = - |\chi^{(i)}| + \lambda_{i}c_{i}$ 
for all $(\chi^{(i)},c_{i}) \in \sP_{1}^{(i)}$. 
Thus we have shown \eqref{eq:gen2}, as desired. 
This completes the proof of Corollary~\ref{cor:main}. 

\appendix
%
\section*{Appendix.}
%
%
\section{Examples of formula \eqref{eq:minu}.}
\label{sec:example}

We give some examples of formula \eqref{eq:minu} in the Introduction 
in the case that $\Fg$ is of type $A_{2}$ and $\mu = \vpi_{1}$; 
recall that $\vpi_{1}$ is a minuscule weight. 
Note that $\Jm=\J_{\vpi_{1}}=\{2\}$, and 
$\WSu{\Jm} = \WSu{\J_{\vpi_1}}= \bigl\{ e, s_{1}, s_{2}s_{1} \bigr\}$, 
$\WS{\Jm} = \WS{\J_{\vpi_1}}= \bigl\{ e, s_{2} \bigr\}$. 
The quantum Bruhat graph in type $A_{2}$ is as follows: 

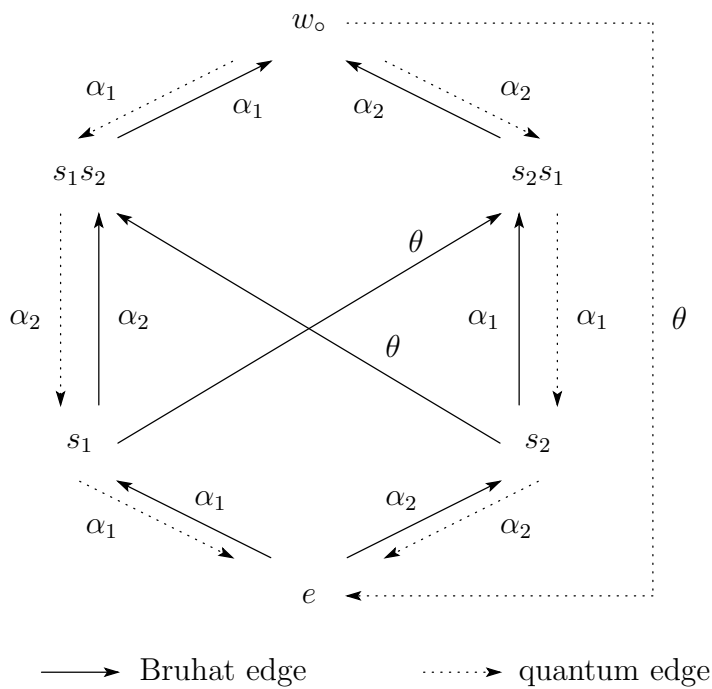
\begin{center} 
{\unitlength 0.1in%
\begin{picture}(38.6000,34.9200)(2.4000,-38.2700)%
\put(22.0000,-4.0000){\makebox(0,0){$w_{\circ}$}}%
\put(10.0000,-12.0000){\makebox(0,0){$s_1s_2$}}%
\put(10.0000,-26.0000){\makebox(0,0){$s_1$}}%
\put(22.0000,-34.0000){\makebox(0,0){$e$}}%
\put(34.0000,-26.0000){\makebox(0,0){$s_2$}}%
\put(34.0000,-12.0000){\makebox(0,0){$s_2s_1$}}%
%
\special{pn 8}%
\special{pa 2000 3200}%
\special{pa 1200 2800}%
\special{fp}%
\special{sh 1}%
\special{pa 1200 2800}%
\special{pa 1251 2848}%
\special{pa 1248 2824}%
\special{pa 1269 2812}%
\special{pa 1200 2800}%
\special{fp}%
%
\special{pn 8}%
\special{pa 2400 3200}%
\special{pa 3200 2800}%
\special{fp}%
\special{sh 1}%
\special{pa 3200 2800}%
\special{pa 3131 2812}%
\special{pa 3152 2824}%
\special{pa 3149 2848}%
\special{pa 3200 2800}%
\special{fp}%
%
\special{pn 8}%
\special{pa 3300 2400}%
\special{pa 3300 1400}%
\special{fp}%
\special{sh 1}%
\special{pa 3300 1400}%
\special{pa 3280 1467}%
\special{pa 3300 1453}%
\special{pa 3320 1467}%
\special{pa 3300 1400}%
\special{fp}%
%
\special{pn 8}%
\special{pa 3200 1000}%
\special{pa 2400 600}%
\special{fp}%
\special{sh 1}%
\special{pa 2400 600}%
\special{pa 2451 648}%
\special{pa 2448 624}%
\special{pa 2469 612}%
\special{pa 2400 600}%
\special{fp}%
%
\special{pn 8}%
\special{pa 1200 1000}%
\special{pa 2000 600}%
\special{fp}%
\special{sh 1}%
\special{pa 2000 600}%
\special{pa 1931 612}%
\special{pa 1952 624}%
\special{pa 1949 648}%
\special{pa 2000 600}%
\special{fp}%
%
\special{pn 8}%
\special{pa 1100 2400}%
\special{pa 1100 1400}%
\special{fp}%
\special{sh 1}%
\special{pa 1100 1400}%
\special{pa 1080 1467}%
\special{pa 1100 1453}%
\special{pa 1120 1467}%
\special{pa 1100 1400}%
\special{fp}%
%
\special{pn 8}%
\special{pa 1000 2800}%
\special{pa 1800 3200}%
\special{dt 0.045}%
\special{sh 1}%
\special{pa 1800 3200}%
\special{pa 1749 3152}%
\special{pa 1752 3176}%
\special{pa 1731 3188}%
\special{pa 1800 3200}%
\special{fp}%
%
\special{pn 8}%
\special{pa 3400 2800}%
\special{pa 2600 3200}%
\special{dt 0.045}%
\special{sh 1}%
\special{pa 2600 3200}%
\special{pa 2669 3188}%
\special{pa 2648 3176}%
\special{pa 2651 3152}%
\special{pa 2600 3200}%
\special{fp}%
%
\special{pn 8}%
\special{pa 3500 1400}%
\special{pa 3500 2400}%
\special{dt 0.045}%
\special{sh 1}%
\special{pa 3500 2400}%
\special{pa 3520 2333}%
\special{pa 3500 2347}%
\special{pa 3480 2333}%
\special{pa 3500 2400}%
\special{fp}%
%
\special{pn 8}%
\special{pa 900 1400}%
\special{pa 900 2400}%
\special{dt 0.045}%
\special{sh 1}%
\special{pa 900 2400}%
\special{pa 920 2333}%
\special{pa 900 2347}%
\special{pa 880 2333}%
\special{pa 900 2400}%
\special{fp}%
%
\special{pn 8}%
\special{pa 1800 600}%
\special{pa 1000 1000}%
\special{dt 0.045}%
\special{sh 1}%
\special{pa 1000 1000}%
\special{pa 1069 988}%
\special{pa 1048 976}%
\special{pa 1051 952}%
\special{pa 1000 1000}%
\special{fp}%
%
\special{pn 8}%
\special{pa 2600 600}%
\special{pa 3400 1000}%
\special{dt 0.045}%
\special{sh 1}%
\special{pa 3400 1000}%
\special{pa 3349 952}%
\special{pa 3352 976}%
\special{pa 3331 988}%
\special{pa 3400 1000}%
\special{fp}%
%
\special{pn 8}%
\special{pa 1200 2600}%
\special{pa 3200 1400}%
\special{fp}%
\special{sh 1}%
\special{pa 3200 1400}%
\special{pa 3133 1417}%
\special{pa 3154 1427}%
\special{pa 3153 1451}%
\special{pa 3200 1400}%
\special{fp}%
%
\special{pn 8}%
\special{pa 3200 2600}%
\special{pa 1200 1400}%
\special{fp}%
\special{sh 1}%
\special{pa 1200 1400}%
\special{pa 1247 1451}%
\special{pa 1246 1427}%
\special{pa 1267 1417}%
\special{pa 1200 1400}%
\special{fp}%
%
\special{pn 8}%
\special{pa 2400 400}%
\special{pa 4000 400}%
\special{dt 0.045}%
%
\special{pn 8}%
\special{pa 4000 400}%
\special{pa 4000 3400}%
\special{dt 0.045}%
%
\special{pn 8}%
\special{pa 4000 3400}%
\special{pa 2400 3400}%
\special{dt 0.045}%
\special{sh 1}%
\special{pa 2400 3400}%
\special{pa 2467 3420}%
\special{pa 2453 3400}%
\special{pa 2467 3380}%
\special{pa 2400 3400}%
\special{fp}%
\put(26.0000,-29.5000){\makebox(0,0)[lb]{$\alpha_2$}}%
\put(16.0000,-29.5000){\makebox(0,0)[lb]{$\alpha_1$}}%
\put(26.0000,-21.5000){\makebox(0,0)[lb]{$\theta$}}%
\put(28.0000,-16.0000){\makebox(0,0)[rb]{$\theta$}}%
\put(36.0000,-20.0000){\makebox(0,0)[lb]{$\alpha_1$}}%
\put(8.0000,-20.0000){\makebox(0,0)[rb]{$\alpha_2$}}%
\put(32.0000,-8.0000){\makebox(0,0)[lb]{$\alpha_2$}}%
\put(12.0000,-8.0000){\makebox(0,0)[rb]{$\alpha_1$}}%
\put(41.0000,-20.0000){\makebox(0,0)[lb]{$\theta$}}%
%
\special{pn 8}%
\special{pa 800 3800}%
\special{pa 1200 3800}%
\special{fp}%
\special{sh 1}%
\special{pa 1200 3800}%
\special{pa 1133 3780}%
\special{pa 1147 3800}%
\special{pa 1133 3820}%
\special{pa 1200 3800}%
\special{fp}%
\put(17.5000,-38.0000){\makebox(0,0){Bruhat edge}}%
%
\special{pn 8}%
\special{pa 2800 3800}%
\special{pa 3200 3800}%
\special{dt 0.045}%
\special{sh 1}%
\special{pa 3200 3800}%
\special{pa 3133 3780}%
\special{pa 3147 3800}%
\special{pa 3133 3820}%
\special{pa 3200 3800}%
\special{fp}%
\put(38.0000,-38.0000){\makebox(0,0){quantum edge}}%
\put(12.0000,-30.0000){\makebox(0,0)[rt]{$\alpha_1$}}%
\put(32.0000,-30.0000){\makebox(0,0)[lt]{$\alpha_2$}}%
\put(32.0000,-20.0000){\makebox(0,0)[rb]{$\alpha_1$}}%
\put(12.0000,-20.0000){\makebox(0,0)[lb]{$\alpha_2$}}%
\put(18.0000,-8.0000){\makebox(0,0)[lt]{$\alpha_1$}}%
\put(26.0000,-8.0000){\makebox(0,0)[rt]{$\alpha_2$}}%
\end{picture}}%
 \end{center}

\noindent
Note that $\theta$ denotes the highest root 
$\alpha_{1}+\alpha_{2}$. By direct calculation, we obtain
%
%
\begin{equation} \label{eq:ex0}
{\renewcommand{\arraystretch}{1.5}
\begin{array}{|c||c|c||c|c||c|c|} \hline
& \multicolumn{2}{c||}{e \WS{\J_{\vpi_1}}} 
& \multicolumn{2}{c||}{s_{1} \WS{\J_{\vpi_1}}} 
& \multicolumn{2}{c|}{s_{2}s_{1} \WS{\J_{\vpi_1}}} \\ \hline\hline
e & e & 0 & s_{1} & \alpha_1^{\vee}& \lng & \theta^{\vee} \\ \hline
s_{2} &  s_{2} & 0 & s_{1} & \alpha_1^{\vee} & s_{2}s_{1} & \alpha_{1}^{\vee} \\ \hline
s_{1} & e & 0 & s_{1} & 0 & \lng & \theta^{\vee} \\ \hline
s_{1}s_{2} & s_{2} & 0 & s_{1}s_{2} & 0 & \lng & \alpha_{1}^{\vee} \\ \hline
s_{2}s_{1} & s_{2} & 0 & s_{1} & 0 & s_{2}s_{1} & 0 \\ \hline
\lng & s_{2} & 0 & s_{1}s_{2} & 0 & \lng & 0 \\ \hline
\end{array}\,,
}
\end{equation}
where for $v \in W$ and $w \in \WSu{\J_{\vpi_1}}$, 
we have
\begin{equation*}
{\renewcommand{\arraystretch}{1.5}
\begin{array}{|c||c|c|} \hline
& \multicolumn{2}{c|}{w \WS{\J_{\vpi_1}}} \\ \hline\hline
v & \tbmax{w}{\J_{\vpi_1}}{v}=:\ha{w} & 
    \wt\bigl( \ha{w} \Rightarrow v \bigr) \\ \hline
\end{array}\,.
}
\end{equation*}

Since $\mu = \vpi_{1}$ is a minuscule weight, 
we see by \eqref{eq:QLS-min} that 
\begin{equation*}
\QLS(\mu)=\QLS(\vpi_{1})=
\bigl\{ (e\,;\,0,1),\ (s_{1}\,;\,0,1), \ (s_{2}s_{1}\,;\,0,1) \bigr\}. 
\end{equation*}
Let us recall formula \eqref{eq:minu} from the Introduction: 
for $w \in W$, 
\begin{equation} \label{eq:minua}
\begin{split}
& [\CO_{\bQG(s_{1})}] \cdot [\CO_{\bQG(w)}]  \\
& \quad 
= [\CO_{\bQG(w)}] - \be^{ \vpi_{1} - w \vpi_{1} } 
\sum_{
  \begin{subarray}{c}
  v \in W \\[1mm] 
  \tbmax{w}{I \setminus \{1\}}{v} = w
  \end{subarray}}
  (-1)^{\ell(v) - \ell(w)} 
  [\CO_{\bQG(v t_{\wt (w \Rightarrow v)})}]. 
\end{split}
\end{equation}
For example, if $w = s_{2}$, then 
it follows from \eqref{eq:ex0} that
\begin{equation*}
\bigl\{ v \in W 
\mid \tbmax{s_{2}}{\J_{\vpi_1}}{v} = s_{2} \bigr\} = 
\bigl\{ s_{2}, \, s_{1}s_{2}, \, s_{2}s_{1}, \, \lng \bigr\}, 
\end{equation*}
and hence by \eqref{eq:minua} 
\begin{equation*}
[\CO_{\bQG(s_{1})}] \cdot [\CO_{\bQG(s_{2})}] = 
[\CO_{\bQG(s_{1}s_{2})}] + [\CO_{\bQG(s_{2}s_{1})}] - [\CO_{\bQG(\lng)}]. 
\end{equation*}
Similarly, if $w=s_{1}$, then 
\begin{equation*}
\begin{split}
& [\CO_{\bQG(s_{1})}] \cdot [\CO_{\bQG(s_{1})}] =  \\[2mm]
& \quad [\CO_{\bQG(s_{1})}] - \be^{\vpi_{1}-s_{1}\vpi_{1}}
  \bigl([\CO_{\bQG(s_{1})}] 
   - [\CO_{\bQG(t_{\alpha_1^{\vee}})}] 
   - [\CO_{\bQG(s_{2}s_{1})}]
   + [\CO_{\bQG(s_{2}t_{\alpha_1^{\vee}})}] \bigr); 
\end{split}
\end{equation*}
if $w=s_{1}s_{2}$, then 
\begin{equation*}
[\CO_{\bQG(s_{1})}] \cdot [\CO_{\bQG(s_{1}s_{2})}] = 
[\CO_{\bQG(s_{1}s_{2})}] - \be^{\vpi_{1}-s_{1}\vpi_{1}}
  \bigl( [\CO_{\bQG(s_{1}s_{2})}] - [\CO_{\bQG(\lng)}] \bigr); 
\end{equation*}
if $w=s_{2}s_{1}$, then 
\begin{equation*}
[\CO_{\bQG(s_{1})}] \cdot [\CO_{\bQG(s_{2}s_{1})}] = 
[\CO_{\bQG(s_{2}s_{1})}] - \be^{\vpi_{1}-s_{2}s_{1}\vpi_{1}}
\bigl( [\CO_{\bQG(s_{2}s_{1})}] - 
[\CO_{\bQG( s_2 t_{\alpha_1^{\vee}} )}] \bigr); 
\end{equation*}
if $w=\lng$, then
\begin{equation*}
\begin{split}
& [\CO_{\bQG(s_{1})}] \cdot [\CO_{\bQG(\lng)}] = \\[2mm]
& \quad [\CO_{\bQG(\lng)}] - \be^{\vpi_{1}-s_{2}s_{1}\vpi_{1}}\bigl(
  [\CO_{\bQG(\lng)}] - [\CO_{\bQG( t_{\theta^{\vee}} )}] 
  - [\CO_{\bQG( s_1s_2 t_{\alpha_{1}^{\vee}} )}] + [\CO_{\bQG( s_1 t_{\theta^{\vee}} )}]\bigr).
\end{split}
\end{equation*}
%
%
\section{Right-hand side of \eqref{eq:main} 
in the case that $\lambda-\mu$ is not dominant.}
\label{sec:notdom}

Let $\lambda,\,\mu \in P^{+}$, 
and define $\Jl,\,\Jm \subset I$ as in \eqref{eq:J}. 
We write $\lambda$ and $\mu$ as
$\lambda = \sum_{i \in I} \lambda_{i} \vpi_{i}$ and 
$\mu = \sum_{i \in I} \mu_{i} \vpi_{i}$, respectively. 
We will prove that 
if $\lambda - \mu$ is not dominant, then 
the right-hand side of \eqref{eq:main} 
is identical to $0$ for all $x \in W_{\af}$; 
we prove this assertion by induction on 
$|\mu|=\sum_{i \in I} \mu_{i}$.
Assume that $\lambda - \mu$ is not dominant. 
Also in this case, we denote by $\bF(x)=\bF_{\lambda\mu}(x)$ 
the right-hand side of \eqref{eq:main}, as in \S\ref{subsec:rec2}.

Assume first that $\mu=\vpi_{r}$ for some $r \in I$ 
and $x=t_{\xi}$ for some $\xi \in Q^{\vee}$; 
since $\lambda-\mu = \lambda - \vpi_{r}$ is not dominant, 
it follows that $\lambda_{r}=0$. 
By the same argument as for \eqref{eq:e3}, 
we deduce that 
\begin{equation*}
\bF_{\lambda\vpi_{r}}(t_{\xi}) = 
\sum_{m \in \BZ_{\ge 0}} 
  q^{-m\lambda_{r}}\be^{-t_{\xi}\vpi_{r}}
  \bigl( \gch V_{ t_{\xi} }^{-}(\lambda) -
         \gch V_{ s_{r}t_{\xi} }^{-}(\lambda) \bigr).
\end{equation*}
Since $\lambda_{r}=0$ as seen above, we have 
$V_{ s_{r}t_{\xi} }^{-}(\lambda) = V_{ t_{\xi} }^{-}(\lambda)$ 
(see Remark~\ref{rem:412}). Hence we obtain 
$\bF_{\lambda\vpi_{r}}(t_{\xi})=0$ in this case.

Also, we remark that the assertion of Proposition~\ref{prop:rec2} 
and the arguments in its proof are still valid even if 
$\lambda-\mu$ is not dominant. 

Assume next that $\mu=\vpi_{r}$ for some $r \in I$ and $x \in W_{\af}$. 
If we take a sequence $i_{1},\,\dots,\,i_{n} \in I_{\af}$ 
and $\xi \in Q^{\vee}$ as in \S\ref{subsec:s2}, then we have 
\begin{equation*}
\bF_{\lambda\vpi_{r}} (x) = 
\sD_{i_n}\sD_{i_{n-1}} \cdots \sD_{i_2}\sD_{i_1}
\bF_{\lambda\vpi_{r}} (t_{\xi}).
\end{equation*}
Since $\bF_{\lambda\vpi_{r}} (t_{\xi}) = 0$ as shown above, 
we obtain $\bF_{\lambda\vpi_{r}} (x) = 0$. 
This proves the desired equality in the case that $|\mu| = 1$. 

Assume now that $|\mu| > 1$, and take $r \in I$ such that 
$\mu_{r} \ge 1$. We set $\nu:=\mu-\vpi_{r} \in P_{+}$; 
notice that $|\nu| < |\mu|$. By the same argument as 
for $\bH(x)$ in \S\ref{subsec:s3}, we see that 

\begin{align}
& \bF_{\lambda\mu}(x) = 
\sum_{ \begin{subarray}{c}
   z \in W_{\af} \\[1mm]
   z \sige x \end{subarray}}\ 
 \sum_{ \begin{subarray}{c}
  \psi \in \SLS(\mu) \\[1mm] 
  \iota(\psi) \sile \Pi^{J_{\mu}}(z),\, \kap{\psi}{z}=x 
  \end{subarray}} (-1)^{\sell(z)-\sell(x)}
  \be^{-\wt(\psi)} 
  \gch V_{z}^{-}(\lambda) \nonumber \\[3mm]
& \quad = \sum_{ \begin{subarray}{c}
   y \in W_{\af} \\
   y \sige x \end{subarray}}\,
 \sum_{ \begin{subarray}{c}
  \pi \in \SLS(\nu) \\ 
  \iota(\pi) \sile \PS{\J_{\nu}}(y),\, \kap{\pi}{y}=x 
  \end{subarray}}  (-1)^{\sell(y)-\sell(x)}
  \be^{-\wt(\pi)} \times \nonumber \\[2mm]
& \hspace*{40mm}
\underbrace{
 \sum_{ \begin{subarray}{c}
   z \in W_{\af} \\
   z \sige y \end{subarray}}\,
 \sum_{ \begin{subarray}{c}
  \eta \in \SLS(\vpi_{r}) \\
  \iota(\eta) \sile \PS{\J_{\vpi_{r}}}(z),\, \kap{\eta}{z}=y 
  \end{subarray}}  (-1)^{\sell(z)-\sell(y)}
  \be^{-\wt(\eta)}
  \gch V_{z}^{-}(\lambda)}_{=\bF_{\lambda\vpi_{r}}(y)}. \label{eq:b1}
\end{align}
If $\lambda-\vpi_{r}$ is not dominant, then we have 
$\bF_{\lambda\vpi_{r}}(y)= 0$ by the argument above, and hence 
$\bF_{\lambda\mu}(x) = 0$. Hence we assume that $\lambda-\vpi_{r}$ is 
dominant. Then it follows from Theorem~\ref{thm:main} that 
$\bF_{\lambda\vpi_{r}}(y) = (1-q^{-\lambda_{r}})^{-1} \gch V_{y}^{-}(\lambda-\vpi_{r})$. 
Substituting this equality into \eqref{eq:b1}, we obtain 
\begin{align*}
& \bF_{\lambda\mu}(x) \\
& = (1-q^{-\lambda_{r}})^{-1}
 \underbrace{
 \sum_{ \begin{subarray}{c}
   y \in W_{\af} \\
   y \sige x \end{subarray}}\,
 \sum_{ \begin{subarray}{c}
  \pi \in \SLS(\nu) \\ 
  \iota(\pi) \sile \PS{\J_{\nu}}(y),\, \kap{\pi}{y}=x 
  \end{subarray}}  (-1)^{\sell(y)-\sell(x)}
  \be^{-\wt(\pi)}  \gch V_{y}^{-}(\lambda-\vpi_{r})}_{=\bF_{\lambda-\vpi_{r},\nu}(x)}, 
\end{align*}
and hence $\bF_{\lambda\mu}(x) = (1-q^{-\lambda_{r}})^{-1} \bF_{\lambda-\vpi_{r},\nu}(x)$.
Since $(\lambda-\vpi_{r})-\nu = \lambda-\mu$ is not dominant, and since 
$|\nu| < |\mu|$, it follows by the induction hypothesis that 
$\bF_{\lambda-\vpi_{r},\nu}(x) = 0$, and hence $\bF_{\lambda\mu}(x) = 0$. 
This completes the proof of the desired equality 
$\bF_{\lambda\mu}(x) = 0$ in the case that 
$\lambda-\mu$ is not dominant. 
%
%
\section{Chevalley formula for dominant weights in terms of quantum LS paths.}
\label{sec:PC-QLS}

We know from \cite[(3.3)]{KNS} that 
for $\lambda,\,\mu \in P^{+}$ and $x \in W$, 
%
%
\begin{equation} \label{eq:KNS}
\gch V_{x}^{-}(\lambda+\mu) = 
\sum_{ \pi \in \SLS_{\sige x}(\mu) } \be^{\fin(\wt(\pi))} q^{\nul(\wt(\pi))} 
\gch V_{\io{\pi}{x}}^{-}(\lambda), 
\end{equation}
where $\io{\pi}{x} \in W_{\af}$ is defined as in \eqref{eq:tix} 
with $y$ replaced by $x$. We will 
reformulate this formula in terms of quantum LS paths, 
as in Corollary~\ref{cor:main} and \S\ref{subsec:prf-cor}. 

We write $\mu$ as $\mu = \sum_{i \in I} \mu_{i} \vpi_{i}$, and 
define $\ol{\Par(\mu)}$ to be the set of $I$-tuples of partitions 
$\bchi = (\chi^{(i)})_{i \in I}$ such that $\chi^{(i)}$ is a partition of 
length less than or equal to $\mu_{i}$ for each $i \in I$; 
for $\bchi \in \ol{\Par(\mu)}$, we define $|\bchi| \in \BZ_{\ge 0}$ 
and $\iota(\bchi) \in Q^{\vee,+}$ as in \S\ref{subsec:conn}. 
For $\eta = (w_{1},\,\dots,\,w_{s} \,;\, 
a_{0},\,a_{1},\,\dots,\,a_{s}) \in \QLS(\mu)$ and $x \in W$, 
we define $\io{\eta}{x} \in W$ 
by the following recursive formula 
(cf. \eqref{eq:tix} and Proposition~\ref{prop:mins}): 
%
%
\begin{equation} \label{eq:tiw}
\begin{cases}
\ti{w}_{s+1}:=x, & \\[2mm]
\ti{w}_{u}:=\tbmin{w_{u}}{\J_{\mu}}{\ti{w}_{u+1}} & \text{for $1 \le u \le s$}, \\[2mm]
\io{\eta}{x}:=\ti{w}_{1}. 
\end{cases}
\end{equation}
%
%
Also, we set
%
%
\begin{equation} \label{eq:xiv}
\xi(\eta,x):= 
\sum_{u=1}^{s} \wt (\ti{w}_{u+1} \Rightarrow \ti{w}_{u}), 
\end{equation}
and
%
%
\begin{equation} \label{eq:degx}
\Deg_{x}(\eta):= - \sum_{u=1}^{s} a_{u} \pair{\mu}{\wt(\ti{w}_{u+1} \Rightarrow \ti{w}_{u})};
\end{equation}
notice that $\Deg_{x}(\eta) = 
\Deg(\eta) - \pair{\mu}{\wt(x \Rightarrow \kappa(\eta))}$ 
by \eqref{eq:deg} and Lemma~\ref{lem:wtS}.  
%
%
\begin{cor} \label{cor:KNS}
Keep the notation and setting above.
For $x \in W$, we have
\begin{equation} \label{eq:KNSc}
\gch V_{x}^{-}(\lambda+\mu) = 
\sum_{\eta \in \QLS(\mu)}
\sum_{\bchi \in \ol{\Par(\mu)}}
\be^{\wt(\eta)} q^{\Deg_{x}(\eta) - |\bchi|} 
\gch V_{\io{\eta}{x} t_{\xi(\eta,x) + \iota(\bchi)}}^{-}(\lambda). 
\end{equation}
\end{cor}

\begin{proof} We see that 
\begin{equation*}
\SLS_{\sige x}(\mu) = \bigsqcup_{\eta \in \QLS(\mu)} 
\cl^{-1}(\eta) \cap \SLS_{\sige x}(\mu),
\end{equation*}
where $\cl:\SLS(\mu) \twoheadrightarrow \QLS(\mu)$ is 
the projection given in \S\ref{subsec:QLS}. 
We deduce from \cite[(5.4)]{NNS1}, together with Lemma~\ref{lem:wtS}, that 
\begin{equation}
\cl^{-1}(\eta) \cap \SLS_{\sige x}(\mu) = 
\bigl\{ \pi_{\bchi,\eta} \cdot t_{\wt (x \Rightarrow \kappa(\eta)) + \gamma}
 \mid \bchi \in \Par(\mu), \, 
 \gamma \in Q^{\vee,+}_{I \setminus \J_{\mu}} \bigr\}; 
\end{equation}
here, $\pi_{\bchi,\eta}$ is the (unique) element 
in the connected component $\SLS_{\bchi}(\mu)$ 
corresponding to $\bchi \in \Par(\mu)$ such that 
$\cl(\pi_{\bchi,\eta})=\eta$ and 
$\kappa(\pi_{\bchi,\eta}) = \kappa(\eta) \in W$ (see \S\ref{subsec:prf-cor}), 
and for $\pi = (x_{1},\,\dots,\,x_{s} \,;\, 
a_{0},\,a_{1},\,\dots,\,a_{s}) \in \SLS(\mu)$ and $\xi \in Q^{\vee}$, 
we set 
\begin{equation*}
\pi \cdot t_{\xi} := 
(\PS{\J_{\mu}}(x_{1}t_{\xi}),\,\dots,\,\PS{\J_{\mu}}(x_{s}t_{\xi}) \,;\, 
a_{0},\,a_{1},\,\dots,\,a_{s}) \in \SLS(\mu)
\end{equation*}
(see \cite[(7.7)]{KNS}). In the same manner 
as in the calculation on page 43 of \cite{NNS1}, we see 
that if $\bchi = (\chi^{(i)})_{i \in I} \in \Par(\mu)$ and 
$\gamma=\sum_{i \in I \setminus \J_{\mu}}\gamma_{i} \alpha_{i}^{\vee} 
\in Q^{\vee}_{I \setminus \J_{\mu}}$, 
then 
\begin{equation*}
\wt (\pi_{\bchi,\eta} \cdot t_{\wt (x \Rightarrow \kappa(\eta)) + \gamma}) = 
\wt(\eta) + \bigl( \Deg_{x}(\eta) - |(\gamma_{i}+\chi^{(i)})_{i \in I}| \bigr) \delta,
\end{equation*}
where $\gamma_{i}+\chi^{(i)}$ is defined to be the partition 
$(\gamma_{i}+\chi^{(i)}_1 \ge \gamma_{i}+\chi^{(i)}_2 \ge \cdots \ge 
\gamma_{i}+\chi^{(i)}_{\mu_i-1} \ge \gamma_{i})$ of length less than or equal to $\mu_{i}$ 
for each $i \in I \setminus \J_{\mu}$, and $\emptyset$ for $i \in \J_{\mu}$; 
note that $(\gamma_{i}+\chi^{(i)})_{i \in I} \in \ol{\Par(\mu)}$. 
Also, we deduce by Proposition~\ref{prop:mins}, together with Lemma~\ref{lem:wtS} and 
\eqref{eq:pieta}, that 
\begin{equation*}
\io{\pi_{\bchi,\eta} \cdot t_{\wt (x \Rightarrow \kappa(\eta)) + \gamma}}{x} = 
\io{\eta}{x} t_{\xi(\eta,x) + \iota(\bchi) + \gamma}. 
\end{equation*}
Therefore, the right-hand side of \eqref{eq:KNS} is identical to: 
\begin{align*}
& 
\sum_{ \eta \in \QLS(\mu) } 
\sum_{ \bchi \in \Par(\mu) } 
\sum_{ \gamma \in Q^{\vee,+}_{I \setminus \J_{\mu}} } 
\be^{\wt(\eta)} q^{\deg_{x}(\eta) - |(\gamma_{i}+\chi^{(i)})_{i \in I}|}
\gch V_{\io{\eta}{x} t_{\xi(\eta,x) + \iota(\bchi) + \gamma} }^{-}(\lambda) \\[3mm]
& \qquad = 
\sum_{\eta \in \QLS(\mu)}
\sum_{\bchi \in \ol{\Par(\mu)}}
\be^{\wt(\eta)} q^{\deg_{x}(\eta) - |\bchi|} 
\gch V_{ \io{\eta}{x} t_{\xi(\eta,x) + \iota(\bchi)} }^{-}(\lambda).
\end{align*}
This proves Corollary~\ref{cor:KNS}. 
\end{proof}

\begin{rem}
In the sum on the right-hand side of \eqref{eq:KNSc}, 
for each $\eta \in \QLS(\mu)$ and an arbitrary fixed $\gamma \in Q^{\vee,+}$, 
the number of the elements $\bchi \in \ol{\Par(\mu)}$ contributing to 
the coefficient of $\gch V_{\io{\eta}{x} t_{\xi(\eta,x) + \gamma}}^{-}(\lambda)$ is finite. 
\end{rem}

By the same argument as for Theorem~\ref{ithm1} in the Introduction, 
we obtain the following corollary, which is a rephrasement of \cite[Theorem~5.10]{KNS} 
in terms of quantum LS paths. 
\begin{cor}
Let $\lambda \in P^{+}$ be a dominant integral weight, and 
let $x = w t_{\xi} \in W_{\af}^{\geq 0}$. 
Then, in the $H \times \mathbb{C}^{*}$-equivariant $K$-group $K_{H \times \mathbb{C}^{*}}(\bQG)$, 
we have
\begin{equation} \label{eq:KNSc2}
\begin{split}
& [\CO_{\bQG}(\lambda)] \cdot [\CO_{\bQG(x)}] \\[3mm]
& \qquad = 
\sum_{\eta \in \QLS(-\lng \lambda)}
\sum_{\bchi \in \ol{\Par(-\lng\lambda)}}
\be^{\wt(\eta)} q^{\Deg_{w}(\eta) - \langle -w_{\circ} \lambda, \xi \rangle - |\bchi|} 
[\CO_{\bQG(\io{\eta}{w} t_{\xi + \xi(\eta,w) + \iota(\bchi)})}].
\end{split}
\end{equation}
\end{cor}
%
%
{\small
 \setlength{\baselineskip}{10pt}

}
\end{document}